\documentclass[10pt]{article}
\pdfoutput=1
\usepackage{physics, amsmath, amssymb, amscd, amsthm, amsfonts, dsfont, xparse}
\usepackage{enumitem}
\usepackage{graphicx}
\usepackage{pgfplots}
\usepackage{float}
\pgfplotsset{compat=1.15}
\usepackage{mathrsfs}
\usepackage{subcaption}
\usepackage{mwe}
\usepackage{caption}

\usepackage{url}
\makeatletter
\newcommand{\address}[1]{\gdef\@address{#1}}
\newcommand{\email}[1]{\gdef\@email{\url{#1}}}
\newcommand{\@endstuff}{\par\vspace{\baselineskip}\noindent\small
\begin{tabular}{@{}l}\scshape\@address\\\textit{E-mail address:} \@email\end{tabular}}
\AtEndDocument{\@endstuff}
\makeatother
\title{A title}
\author{Jin Heo}
\address{Department of Mathematics, Korea University, Seoul 02841, South Korea}
\email{trueheo2000@korea.ac.kr}

\usetikzlibrary{arrows}

\NewDocumentCommand{\dslash}{s}{%
  \IfBooleanTF{#1}
    {\big/\mkern-7mu\big/}
    {/\mkern-6mu/}%
}

\newtheorem{proposition}{Proposition}[section]
\newtheorem{theorem}[proposition]{Theorem}
\newtheorem{lemma}[proposition]{Lemma}
\newtheorem{corollary}[proposition]{Corollary}
\newtheorem{conjecture}[proposition]{Conjecture}

\theoremstyle{definition}

\newcommand{\ml}[1]{\mathbf{ml}\left(#1\right)}
\newcommand{\ssl}[1]{\mathbf{sl}\left(#1\right)}
\newcommand{\ub}[1]{\mathbf{ub}\left(#1\right)}
\newcommand{\lb}[1]{\mathbf{lb}\left(#1\right)}
\newcommand{\qh}{\mathds{Q}_h}
\newcommand{\te}{T^{\mathfrak{T}}}
\newcommand{\ov}[2]{\left[#1,#2\right]}
\newcommand{\ten}{T^{\mathfrak{T}}_{-}}
\newcommand{\tep}{T^{\mathfrak{T}}_{+}}

\title{Reptile trapezoids}
\author{Jin Heo}
\date{April 2024}

\begin{document}

\maketitle
\begin{abstract}
	A geometric figure is a reptile if it can be dissected into at least two similar copies congruent to each other. We prove that if a trapezoid is a reptile and not a parallelogram, then the length of each base is a linear combination of the lengths of its legs with rational coefficients. We then rule out isosceles trapezoids and right trapezoids which are not reptile. In particular, we prove that, up to similarity, there are at most six reptile right trapezoids, not a parallelogram, whose acute internal angle is a rational multiple of $\pi$. Finally, we present a rep-25 right trapezoid that is not a parallelogram and is not similar to any of the known reptile trapezoids.
\end{abstract}
\section{Introduction}
A geometric figure is said to be rep-$n$ if it can be dissected into $n$ similar copies congruent to each other. When a figure is rep-$n$ for some $n\geq 2$, it is called a reptile. In \cite{langford19401464}, C. D. Langford proposed the problem of characterizing reptile polygons while giving some examples of them, three of which are reptile trapezoids R1, R2, and R3 illustrated in Figure \ref{fig:Int1}. Later on, S. W. Golomb presented further examples of reptile polygons in \cite{golomb1964replicating}, including the `sphinx,' so far the only known non-convex reptile polygon with an odd number of sides according to \cite[Kapitel 5]{osburg2004selbstahnliche}. Furthermore, in \cite[Section 3.2]{osburg2004selbstahnliche}, I. Osburg provided a method for constructing a non-convex reptile $2n$-gon with $n\geq 3$, showing that there are infinitely many non-convex reptile polygons.

On the other hand, every convex reptile polygon is known to be either a triangle or a quadrilateral according to the results of U. Betke in \cite{betke1976zerlegungen} (as cited in \cite{laczkovich2023quadrilateral}) and of Osburg in \cite[Satz 2.23]{osburg2004selbstahnliche}. Osburg also provided a narrower characterization for reptile quadrilaterals in \cite[Satz 2.9 and Folgerung 3.2]{osburg2004selbstahnliche}: a quadrilateral is a reptile only if it is convex and is either a trapezoid or a cyclic quadrilateral. In addition, M. Laczkovich made a further breakthrough in \cite[Corollary 1.3]{laczkovich2023quadrilateral}, proving that only trapezoids can be a reptile quadrilateral. Consequently, the only problem for classifying convex reptile polygons is finding all reptile trapezoids. I could not find the exact reference, but except for parallelograms, there seem to be only four known examples of reptile trapezoids so far, as illustrated in Figure \ref{fig:Int1}, where the last one, R4, was mentioned by L. Sallows in \cite{sallows2014more}.

\begin{figure}[ht]
\captionsetup[subfigure]{labelformat=empty}
      \centering
	   \begin{subfigure}{0.25\linewidth}
		\includegraphics[width=\linewidth]{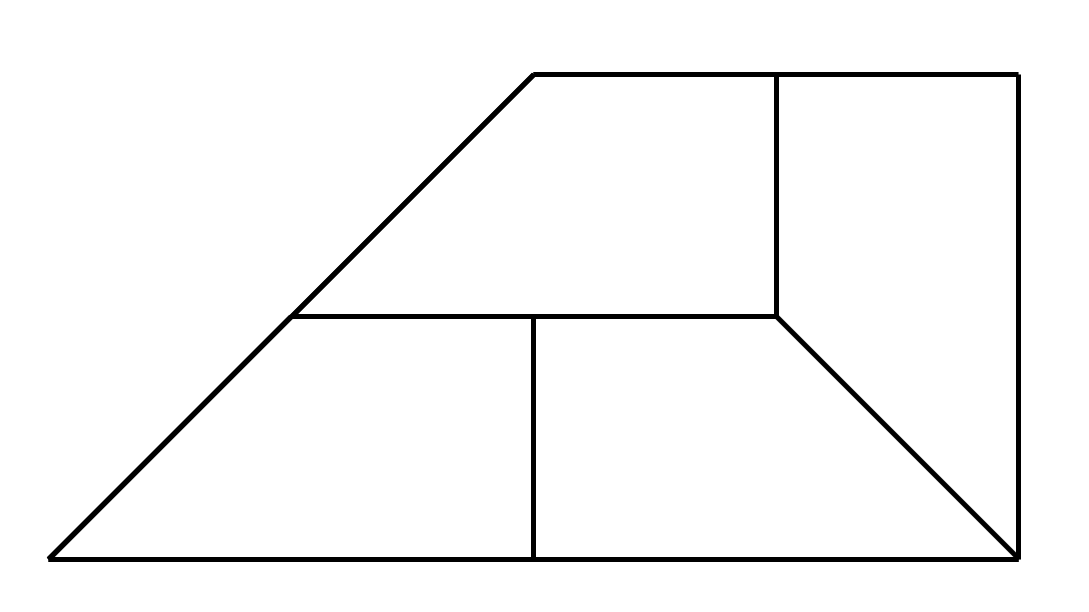}
		\caption{R1}
		\label{fig:Int1a}
	   \end{subfigure}
	   \begin{subfigure}{0.2\linewidth}
		\includegraphics[width=\linewidth]{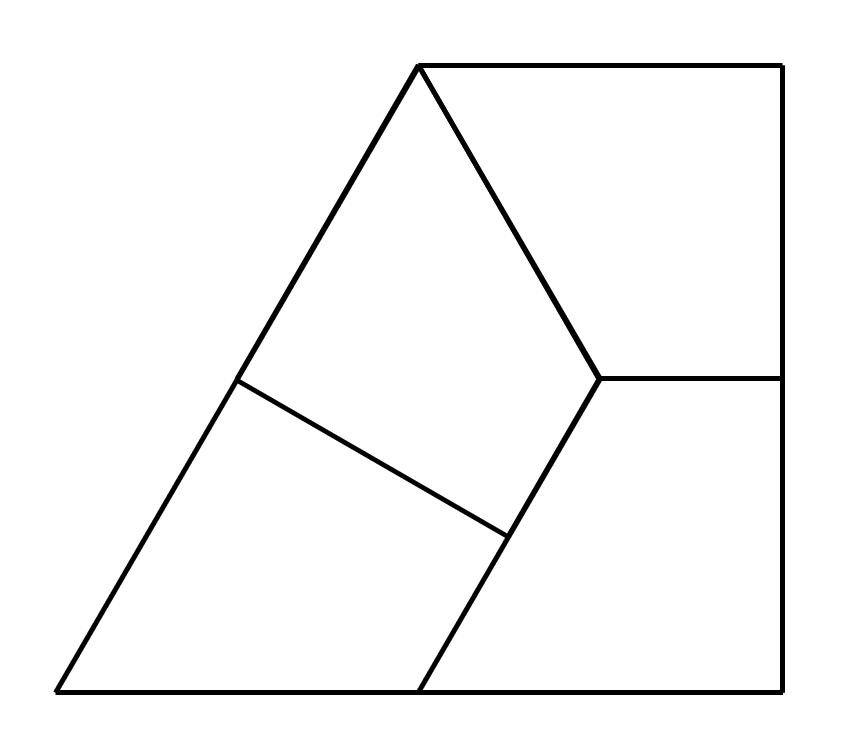}
		\caption{R2}
		\label{fig:Int1b}
	    \end{subfigure}
	     \begin{subfigure}{0.25\linewidth}
		 \includegraphics[width=\linewidth]{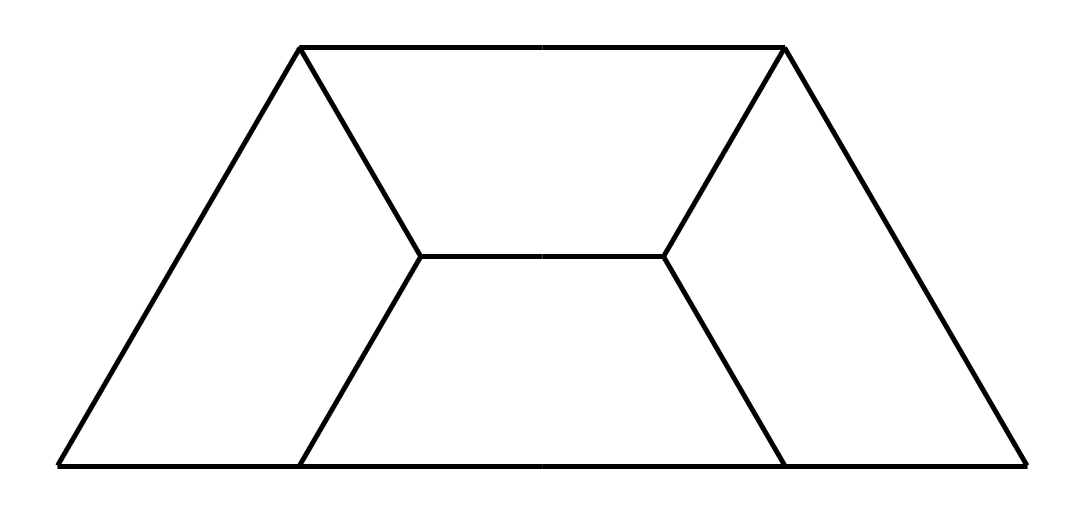}
		 \caption{R3}
		 \label{fig:Int1c}
	      \end{subfigure}
	       \begin{subfigure}{0.25\linewidth}
		  \includegraphics[width=\linewidth]{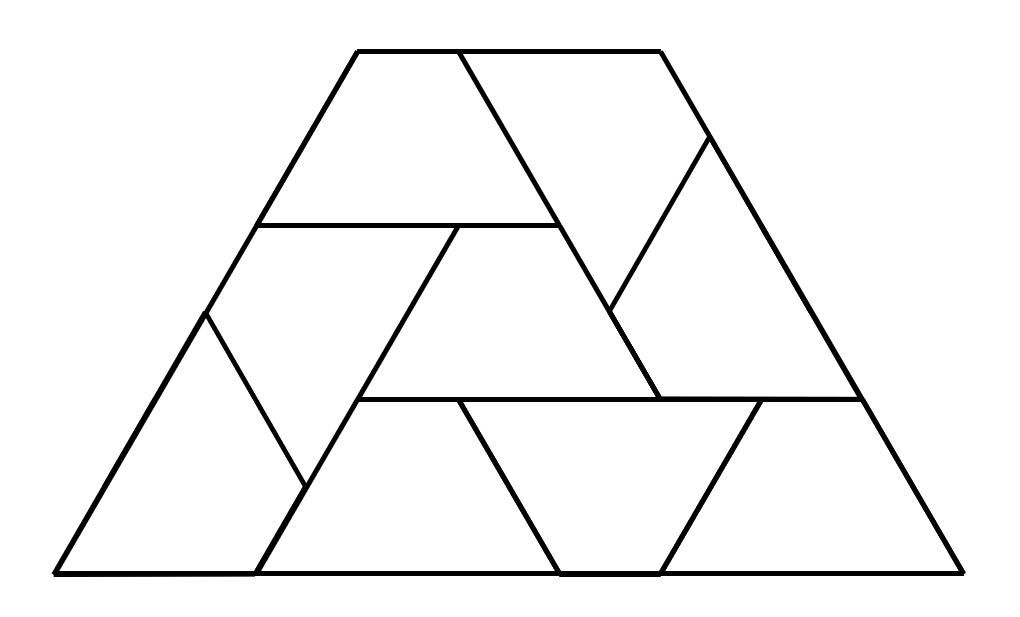}
		  \caption{R4}
		  \label{fig:Int1d}
	       \end{subfigure}
	\caption{The only known examples of reptile trapezoids, not a parallelogram}
	\label{fig:Int1}
\end{figure}

In this paper, we rule out some trapezoids that are not a reptile by providing stronger necessary conditions for being a reptile (see Theorem \ref{GT} and \ref{GT2}). In particular, we investigate isosceles trapezoids and right trapezoids whose acute internal angle is a rational multiple of $\pi$ and narrow down the candidates of such trapezoids that are reptile (see Corollary \ref{Cii}, \ref{Rii}, and Theorem \ref{RR}). Moreover, we present a rep-25 trapezoid, not a parallelogram, to which none of the four trapezoids in Figure \ref{fig:Int1} is similar and observe that it cannot be rep-$n$ for any $1<n<25$ (see Figure \ref{fig:newrep} and Theorem \ref{NEW}).

The structure of this paper is as follows. In Section \ref{TEM}, we introduce some terminologies and notations that will be frequently used throughout the discussion. Then, in Section \ref{MRS}, we describe the main results of this paper in advance, followed by a comment for further research and a proposal for some related conjectures. In Section \ref{Props}, we discuss the properties and lemmas essential to the proof of the main results. Using them, we prove Theorem \ref{GT} in Section \ref{GTP}. We then utilize the results of Theorem \ref{GT} and the lemmas in Section \ref{Props} to prove Theorem \ref{GT2} in Section \ref{GT2P}. Furthermore, in Section \ref{MI}, we prove Corollary \ref{Cii} using the result of Theorem \ref{GT} and prove Corollary \ref{Rii} similarly in Section \ref{RP1}. Finally, we prove Theorem \ref{RR} throughout Section \ref{RP2}, \ref{RP3} and \ref{RP4}, and prove Theorem \ref{NEW} in Section \ref{RP5}.

\section{Terminologies and notations}\label{TEM}

\subsection{Line segments and polygons}
Given a line segment $L$, we define $|L|$ as the length of $L$. For readability, given two points $v$ and $w$ in $\mathds{R}^2$, we denote by $\ov{v}{w}$ the closed line segment with two endpoints $v$ and $w$. Furthermore, given $n$ points $v_1, v_2, \cdots, v_n$ in $\mathds{R}^2$, we put
$$
\left[v_1, v_2, \cdots, v_n\right]=\cup_{i=1}^{n-1}\ov{v_i}{v_{i+1}}
$$
and call it a \textbf{polygonal chain}. In addition, we define the polygon $v_1v_2\cdots v_n$ as the $n$-gon whose boundary is a polygonal chain $\left[v_1, v_2, \cdots, v_n,v_1\right]$. Given a polygon, we refer to the union of its boundary and the set of all points enclosed by the boundary as the \textbf{region} of the polygon.

 We often use the term `interior' differently for line segments with positive length and polygons with positive area. For the former usage, the interior of a given line segment means its subset containing all but two endpoints of the line segment. On the contrary, for the latter usage, the interior of a given polygon refers to the interior of a subset of Euclidean space $\mathds{R}^2$.

\subsection{Trapezoids}
Let $T$ be a trapezoid that is not a parallelogram. If the length of its longest leg equals $1$, then we call it a \textbf{unit trapezoid}. Since $T$ is not a parallelogram, its two bases have different lengths; we refer to the longer base of $T$ as the \textbf{lower base} of $T$ and the shorter base of $T$ as the \textbf{upper base} of $T$ and denote $\lb{T}$ and $\ub{T}$, respectively. If $T$ is not an isosceles trapezoid, then its two legs also have different lengths. In such a case, We call the longer leg of $T$ the \textbf{main leg} of $T$ and the shorter leg of $T$ the \textbf{subsidiary leg} of $T$ and denote $\ml{T}$ and $\ssl{T}$, respectively. If $T$ is isosceles, we consider each of its legs as the main leg and, at the same time, the subsidiary leg of $T$, and we will not distinguish these two legs unless otherwise noticed. We denote by $\theta(T)$ the internal angle of $T$ between $\lb{T}$ and $\ml{T}$ and call it the \textbf{main acute angle} of $T$. If $T$ is not a right trapezoid, then there is an acute internal angle of $T$ between $\ssl{T}$ and some base of $T$; the corresponding base is $\ub{T}$ if $T$ is obtuse, and is $\lb{T}$ otherwise. We call that acute angle the \textbf{subsidiary acute angle} of $T$ and write $\psi(T)$. As before, if $T$ is isosceles, then each of its two acute internal angles is called the main acute angle and, simultaneously, the subsidiary acute angle of $T$, and they will not be distinguished unless otherwise noticed. If $T$ is an isosceles trapezoid (resp. right trapezoid) with $\theta(T)=\vartheta$, then we will call it a $\vartheta$-\textbf{isosceles trapezoid} (resp. $\vartheta$-\textbf{right trapezoid}).

In some parts of the discussion, trapezoids and other polygons will be considered subsets of real plane $\mathds{R}^2$ with the usual $x$-axis and $y$-axis. Given a right trapezoid $T_0$ on $\mathds{R}^2$ whose subsidiary leg is parallel to the $y$-axis, we say $T_0$ is \textbf{positive} if points on $\ub{T_0}$ have higher $y$-coordinate value than those on $\lb{T_0}$, and say \textbf{negative} otherwise.

Finally, given a positive number $\mu$, we denote by $\mu T$ a copy similar to $T$ obtained by multiplying the length of each edge of $T$ by $\mu$.

\subsection{Tiles and tilings}
Suppose a polygon $P$ is tiled with congruent copies of $T$. We refer to the collection consisting of all these copies as a \textbf{tiling} $\mathfrak{T}$ of $P$ and each copy in the tiling as a \textbf{tile}. When we mention some tiles in the tiling $\mathfrak{T}$ unspecifically, we denote them by $\te$ identically unless otherwise specified or indexed. Provided that $T$ is a right trapezoid, we denote each unspecified positive (resp. negative) tile by $\tep$ (resp. $\ten$).

If a given angle is dissected into some internal angles of tiles, then we say these internal angles \textbf{fill} the given angle. Given two closed line segments $A$ and $B$, we say that $A$ \textbf{lies on} $B$ if $A\subset B$. Put $n_0=0$, and let $n_1, n_2, \cdots$ be a strictly increasing sequence of positive integers; let $T_1, T_2,\cdots, T_{n_k}\in\mathfrak{T}$ $(k\geq 1)$ be tiles and $E_1, E_2, \cdots, E_{n_k}$ be closed line segments lying on the same straight line, where each $E_i$ is an edge of $T_i$. Suppose that $L_i=\cup_{j=n_{(i-1)}+1}^{n_i}E_j$ is a closed line segment for each $i=1, 2, \cdots, k$. Given a closed line segment $L$, we say that edges $L_1, \cdots, L_k$ \textbf{lie on} $L$ \textbf{in a row} if the following three conditions hold.
\begin{enumerate}
    \item $L_i$ lies on $L$ for each $i=1, 2, \cdots, k$;
    \item $E_i\cap E_j$ is either an empty set or a one-point set for each $1\leq i<j\leq n_k$;
    \item $T_1, T_2,\cdots, T_{n_k}$ are contained in the same closed half-plane whose boundary is the straight line extended from $L$.
\end{enumerate}
In particular, if $L_{i}\cap L_{i+1}$ is a one-point set for each $i=1, 2, \cdots, k-1$, then we say these edges \textbf{lie on} $L$ \textbf{in a row in the order} $\left(L_1, L_2, \cdots, L_k\right)$. Furthermore, we say they \textbf{perfectly cover} (resp. \textbf{properly cover}) $L$ if $L=\cup_{i=1}^{k}L_i$ (resp. $L\varsubsetneq\cup_{i=1}^{k}L_i$). During the proof, we may use these terms without mentioning the underlying tiles $T_1, T_2,\cdots, T_{n_k}$ or edges $E_1, E_2,\cdots, E_{n_k}$ unless some $E_i$ is an edge of more than one tile. Finally, if a polygon is a union of some tiles in $\mathfrak{T}$, then we call such a polygon \textbf{cluster polygon} or shortly a \textbf{cluster in} $\mathfrak{T}$.

\subsection{Vector space $\mathds{Q}_x$}
Given a positive real number $x$, we denote by $\mathds{Q}_x$ the vector subspace of $\mathds{R}$ (as a vector space over $\mathds{Q}$) spanned by $1$ and $x$. If $x$ is irrational, then for each $y\in\mathds{Q}_x$, there is a unique pair of rational numbers $\left(c_{1,y},c_{x,y}\right)$ such that $y=c_{1,y}+c_{x,y}x$; we will use this notation in further discussion.

\section{Main results}\label{MRS}
Before getting to the point, it is worth noting that every parallelogram is trivially a reptile. Moreover, any two similar polygons are reptiles if one is a reptile. For this reason, we will only deal with unit trapezoids throughout the discussion. The following result provides a general necessary condition for trapezoids to be reptiles.

\begin{theorem}\label{GT}
    If a unit trapezoid $T$ is a reptile, then $\left|\ub{T}\right|$ and $\left|\lb{T}\right|$ are contained in $\mathds{Q}_{\left|\ssl{T}\right|}$.
\end{theorem}

\noindent
This result shows that if a trapezoid (not necessarily a unit trapezoid) is a reptile and not a parallelogram, then the length of each base is a linear combination of the lengths of its legs with rational coefficients. In particular, if a given unit trapezoid is wide enough and its subsidiary leg is of irrational length, then we can further narrow the previous necessary condition as follows.

\begin{theorem}\label{GT2}
    Suppose that $T$ is a unit trapezoid, not a $\pi/3$-right trapezoid, with $\left|\ssl{T}\right|=h\notin\mathds{Q}$, $\left|\ub{T}\right|=a$, and $\left|\lb{T}\right|=b>1$. If $T$ is a reptile, then one of the following holds.
    \begin{enumerate}
        \item $c_{1,a}>0$, $c_{h,a}<0$, $c_{1,b}\leq 0$;
        \item $c_{1,a}<0$, $c_{h,a}>0$, $c_{h,b}\leq 0$;
        \item $T$ is obtuse and $c_{h,a}=c_{h,b}=0$.
    \end{enumerate}
\end{theorem}

\noindent
Theorem \ref{GT} immediately implies the following result.

\begin{corollary}
    Suppose that a unit trapezoid $T$ is a reptile and $\left|\ssl{T}\right|\in\mathds{Q}$. Then, $\left|\ub{T}\right|$ and $\left|\lb{T}\right|$ are both rational, and $T$ is not rep-$n$ for all square-free integers $n$.
\end{corollary}

Let $I\left(\vartheta, a\right)$ (resp. $R\left(\vartheta, a\right)$) be $\vartheta$-isosceles trapezoid (resp. $\vartheta$-right trapezoid) such that $\left|\ub{T}\right|=a$. Applying Theorem \ref{GT} to isosceles trapezoids, we derive the following necessary condition for $I\left(\theta, a\right)$ being a reptile.

\begin{corollary}\label{Cii}
    If $I\left(\vartheta, a\right)$ is a reptile and $\vartheta/\pi\in\mathds{Q}$, then $\vartheta=\pi/3$.
\end{corollary}
\noindent
For instance, $I\left(\frac{\pi}{3},1\right)$ and $I\left(\frac{\pi}{3},\frac{3}{2}\right)$ are reptiles, as shown in Figure \ref{fig:Int1}. We also obtain the following necessary conditions for right trapezoids $R\left(\vartheta,a\right)$ being a reptile.
\begin{corollary}\label{Rii}
    If $R\left(\vartheta,a\right)$ is a reptile, then $\cos\vartheta$ and $\sin\vartheta$ are both algebraic integers of degree at most $2$.
\end{corollary}

\begin{theorem}\label{RR}
    If $R\left(\vartheta,a\right)$ is a reptile and $\vartheta/\pi\in\mathds{Q}$, then the pair $\left(\vartheta,a\right)$ equals one of $\left(\frac{\pi}{4},\frac{1}{\sqrt{2}}\right)$, $\left(\frac{\pi}{3},1\right)$, $\left(\frac{\pi}{3},\frac{1}{2}\right)$, $\left(\frac{\pi}{3},\frac{1}{4}\right)$, $\left(\frac{\pi}{3},\frac{1}{6}\right)$, and $\left(\frac{\pi}{3},\frac{1}{8}\right)$.
\end{theorem}

\noindent
Therefore, there are at most six possible candidates for reptile unit right trapezoids whose acute angle is a rational multiple of $\pi$, where $R\left(\frac{\pi}{4},\frac{1}{\sqrt{2}}\right)$ and $R\left(\frac{\pi}{3},\frac{1}{2}\right)$ are already confirmed to be a reptile as shown in Figure \ref{fig:Int1}. Among the four remaining candidates, we show that $R\left(\frac{\pi}{3},\frac{1}{8}\right)$ is also a reptile as follows.

\begin{theorem}\label{NEW}
    The right trapezoid $R\left(\frac{\pi}{3},\frac{1}{8}\right)$ is rep-25 but not rep-$n$ for all integers $1<n<25$.
\end{theorem}
\noindent
The first property is illustrated in the figure below. It is clear that $R\left(\frac{\pi}{3},\frac{1}{8}\right)$ is similar to neither of the four reptile trapezoids in Figure \ref{fig:Int1}. 
\begin{figure}[H]
    \centering
    \includegraphics[width=0.25\linewidth]{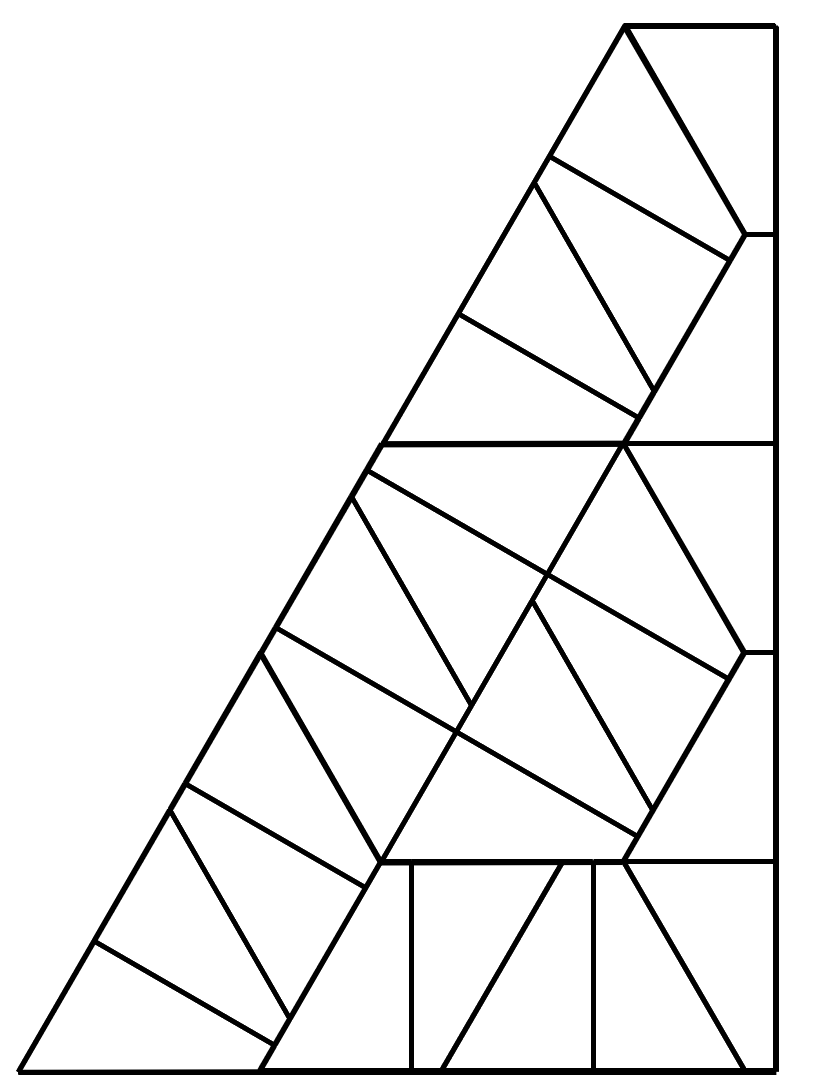}
    \caption{Reptile right trapezoid $R\left(\frac{\pi}{3},\frac{1}{8}\right)$}
    \label{fig:newrep}
\end{figure}

\begin{figure}[H]
\captionsetup[subfigure]{labelformat=empty}
      \centering
	   \begin{subfigure}{0.2\linewidth}
		\includegraphics[width=\linewidth]{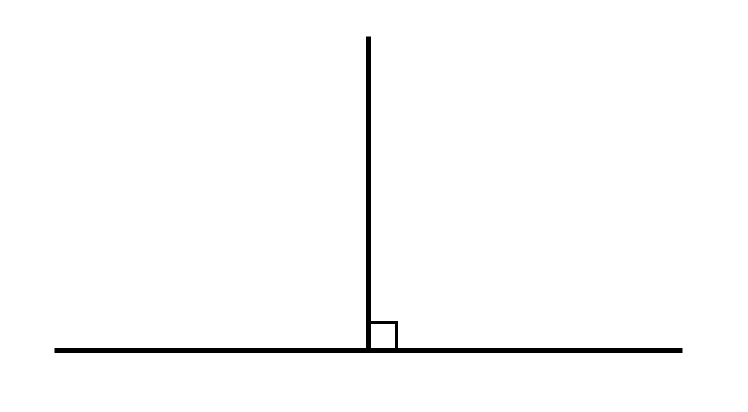}
		\caption{$\frac{\pi}{2}+\frac{\pi}{2}+\pi$}
		\label{fig:dis3}
	   \end{subfigure}
	   \begin{subfigure}{0.2\linewidth}
		\includegraphics[width=\linewidth]{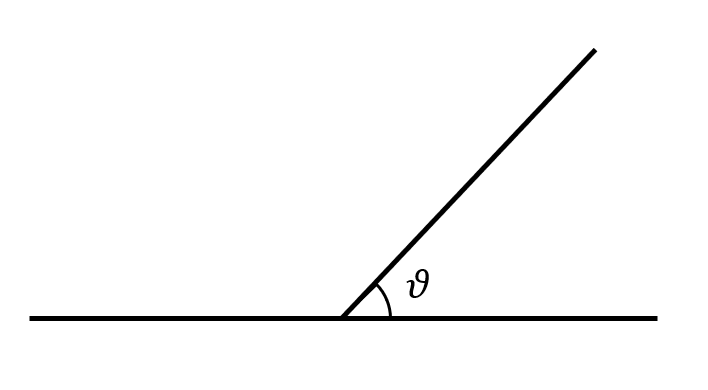}
		\caption{$\vartheta+\vartheta^{*}+\pi$}
		\label{fig:dis4}
	    \end{subfigure}
	     \begin{subfigure}{0.2\linewidth}
		 \includegraphics[width=\linewidth]{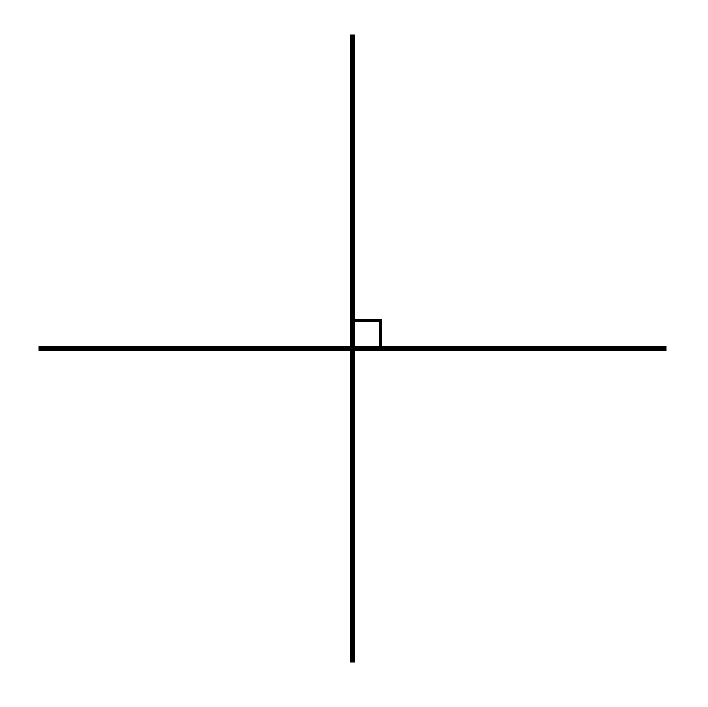}
		 \caption{$\frac{\pi}{2}+\frac{\pi}{2}+\frac{\pi}{2}+\frac{\pi}{2}$}
		 \label{fig:dis5}
	      \end{subfigure}
	       \begin{subfigure}{0.2\linewidth}
		  \includegraphics[width=\linewidth]{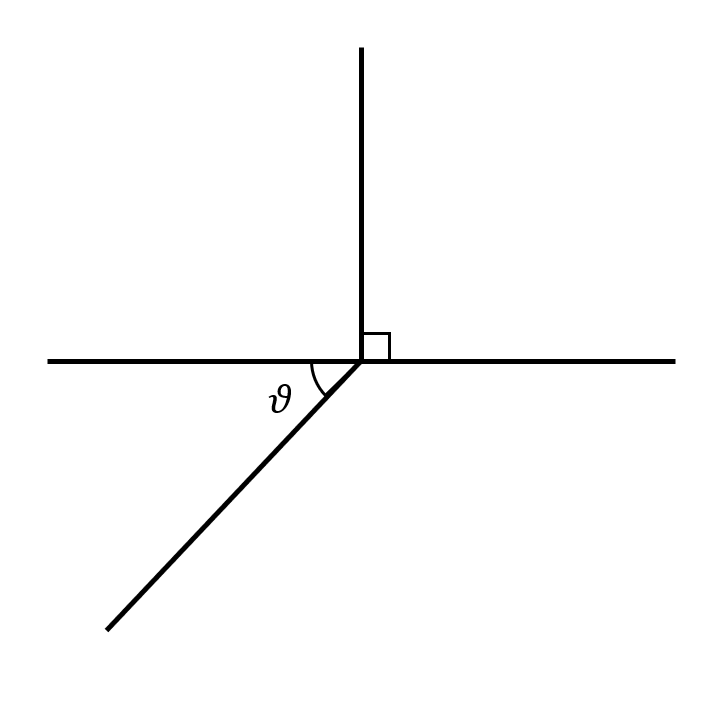}
		  \caption{$\frac{\pi}{2}+\frac{\pi}{2}+\vartheta+\vartheta^{*}$}
		  \label{fig:dis6}
	       \end{subfigure}
        	       \begin{subfigure}{0.2\linewidth}
		  \includegraphics[width=\linewidth]{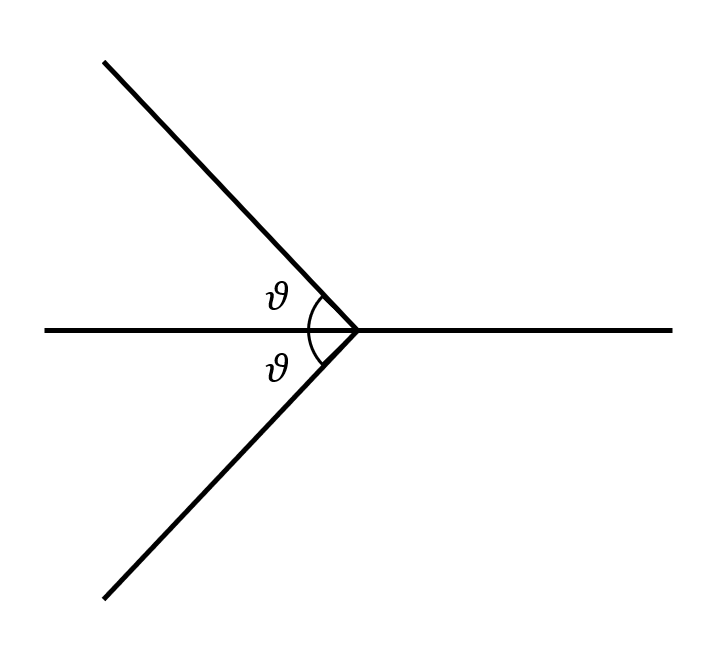}
		  \caption{$\vartheta^{*}+\vartheta+\vartheta+\vartheta^{*}$}
		  \label{fig:dis7}
	       \end{subfigure}
        	       \begin{subfigure}{0.2\linewidth}
		  \includegraphics[width=\linewidth]{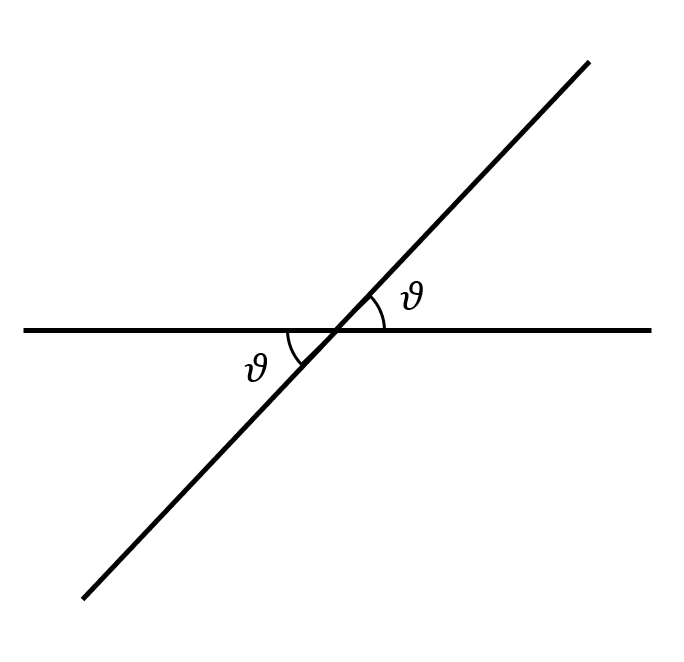}
		  \caption{$\vartheta+\vartheta^{*}+\vartheta+\vartheta^{*}$}
		  \label{fig:dis8}
	       \end{subfigure}
	\caption{All possible dissections of $2\pi$ up to reflection and rotation}
	\label{fig:dissectang1}
\end{figure}

Although the above results provide a pretty clear view of characterizing reptile $I\left(\vartheta,a\right)$ and $R\left(\vartheta,a\right)$ with rational $\vartheta/\pi$, there is yet relatively fewer clues for characterizing those with irrational $\vartheta/\pi$. Fortunately, when $\vartheta/\pi\notin\mathds{Q}$, ways of dissecting an angle of size $2\pi$ into internal angles of a given trapezoid are limited, compared to cases $\vartheta/\pi\in\mathds{Q}$, and can be characterized identically regardless of the size of $\vartheta$ (see Figure \ref{fig:dissectang1} for the complete characterization, where $\vartheta^{*}=\pi-\vartheta$). This might help exclude many possibilities when directly examining whether congruent copies of a given trapezoid can tile their bigger copy.

Along with the preceding main results, we propose some conjectures as follows.
\begin{conjecture}
    If a unit trapezoid $T$ is a reptile, then $\left|\ub{T}\right|\leq 1$.
\end{conjecture}
\begin{conjecture}
    If a unit trapezoid $T$ is a reptile, then both $\theta(T)/\pi$ and $\psi(T)/\pi$ are rational, where $\psi(T)=\pi/2$ when $T$ is a right trapezoid.
\end{conjecture}
\begin{conjecture}
    If $R\left(\frac{\pi}{3},a\right)$ is a reptile, then $a$ equals either $1$ or $1/8$.
\end{conjecture}

The outline of the proof of the main results is as follows. As a preliminary step, we first consider a polygon containing a specific polygonal line in its boundary and some area surrounded by the polygonal line in its region. Then, provided that the polygon is tiled with congruent copies of a given unit trapezoid $T$, we observe how the tiles should cover the polygonal line and a neighboring area if $T$ or the given tiling satisfies some properties; this is similar to the argument of Laczkovich in \cite[Section 4]{laczkovich2023quadrilateral} that concentrates on a local structure of a polygon. Based on this observation, we prove that such properties that are met by $T$ or the tiling lead to the non-existence of certain tilings of trapezoids in some families consisting of those obtained from a larger copy of $T$ by elongating or shortening its bases. During this step, we distinguish the case when $T$ is a right trapezoid from when it is not, for a right internal angle of $T$ necessitates a different approach.

After the preliminary step, we prove Theorem \ref{GT} by contradiction. For the sake of contradiction, we assume that congruent copies of a given unit trapezoid $T$ tile its larger copy $\mu T$ for some integer $\mu\geq 3$ (i.e. $T$ is rep-$\mu^2$), but at least one of $\left|\ub{T}\right|$ and $\left|\lb{T}\right|$ is not contained in $\mathds{Q}_{\left|\ssl{T}\right|}$. We first derive a contradiction for the case when $\left|\ub{T}\right|\notin\mathds{Q}_{\left|\ssl{T}\right|}$ and $\left|\lb{T}\right|\in\mathds{Q}_{\left|\ssl{T}\right|}$ from a number-theoretical approach, which is similar to the arguments of S. L. Snover et al. in \cite{snover1991rep} and of Laczkovich in \cite[Section 3]{laczkovich2023quadrilateral}. For the remaining case, we use an analogous number-theoretical method to show that $T$ or the given tiling satisfies some of the properties introduced in the first step. We then derive a contradiction by applying the results in the preliminary step and, therefore, conclude that $T$ is not rep-$\mu^2$. This result is sufficient enough to imply that $T$ is not a reptile even though we neglected cases when $\mu$ equals either $2$ or the square root of some square-free integer; note that if $T$ is rep-$n$, then it is also rep-$n^m$ for all integers $m\geq 2$ because, for each integer $i\geq 0$, the trapezoid $n^{(i+1)/2} T$ can be tiled with $n$ congruent copies of $n^{i/2} T$ if $T$ is rep-$n$.

Based on the result of Theorem \ref{GT}, we prove Corollary \ref{Rii} by using a trigonometric equation. By additionally applying the result of P. Tangsupphathawat in \cite[Theorem 3.3]{tangsupphathawat2014algebraic} to the main acute angle of a given unit trapezoid, we prove Corollary \ref{Cii} and the angle-related part of Theorem \ref{RR}. 

The proof of Theorem \ref{GT2} and the edge-related part of Theorem \ref{RR} is straightforward. We first assume that a given unit trapezoid $T$ is a reptile, particularly rep-$\mu^2$ for some integer $\mu\geq 3$. Similar to the proof of Theorem \ref{GT}, we consider a tiling of a larger copy $\mu T$ with congruent copies of $T$. By applying the previous number-theoretical approach, we eliminate contradictory cases where some of the properties introduced in the first step are satisfied, giving us the desired results.

Finally, Theorem \ref{NEW} will be proven by investigating one by one whether congruent copies of the unit trapezoid $T=R\left(\frac{\pi}{3},\frac{1}{8}\right)$ can tile its larger copy $\sqrt{n}T$ for some integer $1<n<25$.

\section{Three properties: MTM, STS, and SSTS}\label{Props}
This section introduces three properties regarding situations tiling a given area and provides some related lemmas for further discussion. Throughout the section, $T$ refers to an arbitrarily chosen unit trapezoid. In addition, we put $a=\left|\ub{T}\right|$, $b=\left|\lb{T}\right|,$ and $h=\left|\ssl{T}\right|$.

For further reference, we introduce three terminologies for a given tiling $\mathfrak{T}$ consisting of congruent copies of $T$. We say $\mathfrak{T}$ is \textbf{main-to-main} if the main leg of each tile in $\mathfrak{T}$ cannot be perfectly covered by edges of tiles in $\mathfrak{T}$ including at least one $\lb{\te}$. Moreover, $\mathfrak{T}$ is said to be \textbf{strictly main-to-main} if the main leg of each tile in $\mathfrak{T}$ cannot be perfectly covered by edges of tiles in $\mathfrak{T}$ including at least one of $\lb{\te}$ and $\ub{\te}$. We also say $\mathfrak{T}$ is \textbf{sub-to-sub} if the subsidiary leg of each tile in $\mathfrak{T}$ cannot be perfectly covered by edges of tiles in $\mathfrak{T}$ including at least one $\lb{\te}$. Note that every subcollection of $\mathfrak{T}$ is (strictly) main-to-main (resp. sub-to-sub) if $\mathfrak{T}$ is (strictly) main-to-main (resp. sub-to-sub). Furthermore, it follows immediately from the definition that every strictly main-to-main tiling is main-to-main.

\subsection{Main-to-main property}\label{MTMP}
Let $\rho$ be a positive integer and $\alpha$ be a positive real number. In addition, let $u_0, u_1,$ and $u_2$ be points on $\mathds{R}^2$ such that $\left|\ov{u_0}{u_1}\right|=\alpha$, $\left|\ov{u_1}{u_2}\right|=\rho$, and $\angle u_0u_1u_2=\theta(T)$; let $M\left(T, \rho, \alpha\right)$ be a polygon whose boundary contains the polygonal chain $\left[u_0,u_1,u_2\right]$ and whose region contains the triangle $u_0u_1u_2$. We will say that $T$ satisfies the \textbf{main-to-main property} (abbreviated to \textbf{MTM}) if it satisfies the following condition: for every positive integer $\rho$, positive real number $\alpha$, and polygon $M\left(T, \rho, \alpha\right)$ corresponding to these two numbers, if there is a tiling $\mathfrak{T}$ of $M\left(T, \rho, \alpha\right)$ with congruent copies of $T$, and if one $\lb{\te}$ lies on the line segment $\ov{u_1}{u_2}$, then $\ov{u_1}{u_2}$ is not perfectly covered by edges of tiles in $\mathfrak{T}$. The following lemma describes how MTM constrains the possibility of perfectly covering $\ov{u_1}{u_2}$.

\begin{lemma}\label{L1}
    Let $\rho$ be a positive integer and $\alpha$ be a positive real number. In addition, let $M\left(T, \rho, \alpha\right)$ be a polygon corresponding to these two numbers and the three points $u_0,u_1,$ and $u_2$ in the way explained earlier. Suppose that $T$ satisfies the main-to-main property and there is a tiling $\mathfrak{T}$ of $M\left(T, \rho, \alpha\right)$ with congruent copies of $T$ such that $\ov{u_1}{u_2}$ is perfectly covered by some edges of tiles in $\mathfrak{T}$. Then, there are tiles $T_1,\cdots,T_{\rho}\in\mathfrak{T}$ satisfying the following (see Figure \ref{fig:L1}, for instance):
    \begin{enumerate}
        \item $\ml{T_1},\ml{T_2},\cdots,\ml{T_{\rho}}$ perfectly covers $\ov{u_1}{u_2}$ in the order
        $$
        \left(\ml{T_1},\ml{T_2},\cdots, \ml{T_{\rho}}\right);
        $$
        \item $\ov{u_0}{u_1}\subset\lb{T_1}$ or vice versa;
        \item $\ub{T_i}\varsubsetneq\lb{T_{i+1}}$ for each $i=1, 2, \cdots, \rho-1$.
    \end{enumerate}
\end{lemma}
\begin{figure}[H]
    \centering
    \includegraphics[width=0.4\linewidth]{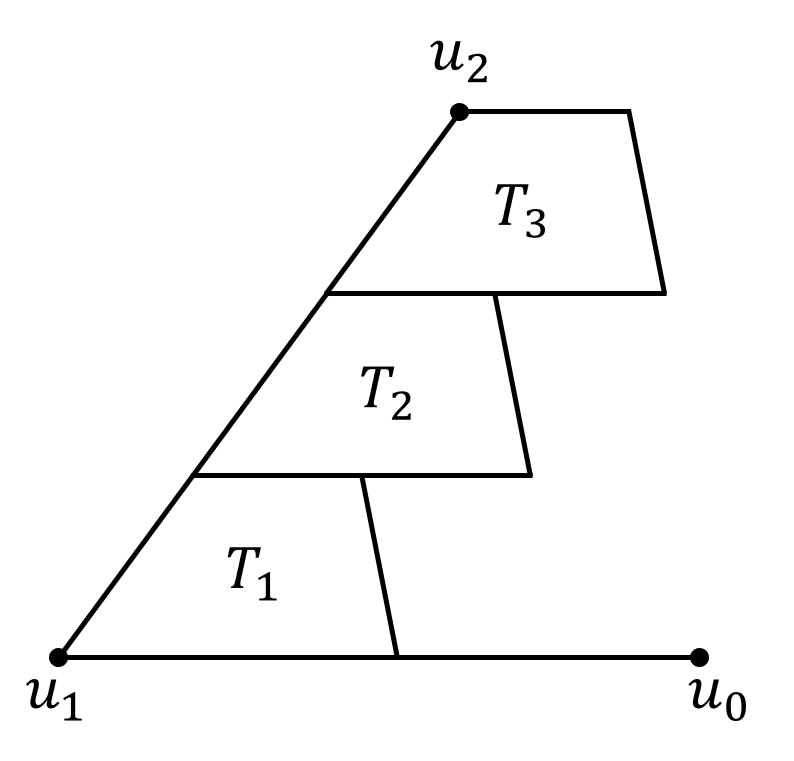}
    \caption{$M(T,3,\alpha)$ and three acute tiles $T_1,T_2,$ and $T_3$}
    \label{fig:L1}
\end{figure}

\begin{proof}
    The acute angle $\theta(T)$ is the smallest internal angle of $T$, so only a single $\theta\left(\te\right)$ can fill the internal angle $\angle u_0u_1u_2$ of $M(T,\rho,\alpha)$; let $T_1\in\mathfrak{T}$ be the tile such that $\theta\left(T_1\right)$ fills $\angle u_0u_1u_2$. Since $T$ has MTM, and since $\ov{u_1}{u_2}$ is perfectly covered by some edges of tiles in $\mathfrak{T}$, the edge $\ml{T_1}$ lies on $\ov{u_1}{u_2}$ instead of $\lb{T_1}$. Moreover, we have $\ov{u_0}{u_1}\subset\lb{T_1}$ if $\alpha\leq b$, and vice versa if $\alpha>b$. This completes the proof for the case $\rho=1$.
    
    Next, assume $\rho\geq 2$. Similar to the case $\rho=1$, since the acute angle between $\ov{u_1}{u_2}$ and $\ub{T_1}$ has the same size as $\theta(T)$, there is a tile $T_2\in\mathfrak{T}$ such that $\theta\left(T_2\right)$ fills that acute angle. For the same reason, $\ml{T_2}$ lies on $\ov{u_1}{u_2}$, and $\ml{T_1}\cap\ml{T_2}$ is a one-point set. Furthermore, $\ub{T_1}$ is properly contained in $\lb{T_2}$, for $b>a$. Inductively, we can select tiles $T_2, \cdots, T_{\rho}\in\mathfrak{T}$, along with $T_1$, satisfying the desired properties.
\end{proof}

Although the above lemma rules out any other possibilities of perfectly covering $\ov{u_1}{u_2}$ with edges of tiles in $\mathfrak{T}$, there is no guarantee that each subsidiary leg of $T_i$ is perfectly covered by a single subsidiary leg of another tile in $\mathfrak{T}$. If we assume this to be true, then we can obtain a helpful cluster polygon that can be utilized as a building block for deriving the non-existence of certain tilings of some polygons with congruent copies of $T$.

To describe this more precisely, given a positive integer $\rho$ and a positive real number $\alpha$, let $\hat{M}\left(T,\rho,\alpha\right)$ be the trapezoid $v_1v_2v_3v_4$, where
    \begin{align*}
        v_1&=(0,0),\\
        v_2&=\left(\rho\cos{\theta(T)},\rho\sin{\theta(T)}\right),\\
        v_3&=\left(\rho\cos{\theta(T)}+\alpha,\rho\sin{\theta(T)}\right)\\
        v_4&=\left(\rho\left(\cos{\theta(T)}+c(T)h\cos{\psi(T)}\right)+\alpha,0\right)
    \end{align*}
are points in $\mathds{R}^2$ and $c(T)$ is defined by
$$
c(T):=\begin{cases}
        1 &\text{if $T$ is acute}\\
        0 &\text{if $T$ is a right trapezoid}\\
        -1 &\text{if $T$ is obtuse}
\end{cases}.
$$
In Section \ref{SSTSP}, provided that $T$ is a right trapezoid, we will introduce a property called SSTS for $T$ related to perfectly covering each $\ssl{T_i}$ in the desired way and show that MTM and SSTS together imply that congruent copies of $T$ cannot tile $\hat{M}\left(T,\rho,\alpha\right)$. For the case when $T$ is not a right trapezoid, we bring property for a tiling $\mathfrak{T}$ rather than $T$ to show the non-existence of sub-to-sub tilings of $\hat{M}\left(T,\rho,\alpha\right)$ with congruent copies of $T$. The following lemma says that assuming a given tiling is sub-to-sub ensures that constructing the desired cluster polygon is feasible when $T$ is a non-right trapezoid.

\begin{lemma}\label{L1.1}
   Let $\rho$ be a positive integer, and $\alpha$ be a positive real number such that $\alpha>a$. Suppose that $T$ is not a right trapezoid and satisfies the main-to-main property. If $\mathfrak{T}$ is a sub-to-sub tiling of $\hat{M}\left(T,\rho,\alpha\right)$ with congruent copies of $T$, then the parallelogram $v_1v_2v_2'v_1'$ is a cluster in $\mathfrak{T}$, where $v'_i=v_i+\left(a+b,0\right)$ for each $i=1, 2$.
\end{lemma}
\begin{proof}

   Let $\mathfrak{T}$ be a sub-to-sub tiling of the trapezoid $v_1v_2v_3v_4$ with congruent copies of $T$, and let $T_1, T_2, \cdots, T_{\rho}\in\mathfrak{T}$ be tiles satisfying the three properties in Lemma \ref{L1} with $u_0,u_1,$ and $u_2$ replaced to $v_4,v_1,$ and $v_2$, respectively. Note that no $T_i$ intersects with $\ov{v_3}{v_4}$. We first deal with the case $\rho=1$. The two line segments $\ov{v_1}{v_4}$ and $\ov{v_2}{v_3}$ both contain an endpoint of $\ssl{T_1}$ in their interior. Thus, it immediately follows that $\ssl{T_1}$ is perfectly covered by some edges of tiles in $\mathfrak{T}$ other than $T_1$. Let $\Theta_1$ be the angle between $\ssl{T_1}$ and $\ov{v_1}{v_4}$ which is not an internal angle of $T_1$, and $\Theta_2$ be the analogous angle between $\ssl{T_1}$ and $\ov{v_2}{v_3}$ (see Figure \ref{fig:MTMCON1}). Since $\ov{v_1}{v_4}$ and $\ov{v_2}{v_3}$ are parallel, we have $\Theta_1+\Theta_2=\pi$. Thus, $\Theta_2$ is acute if $\Theta_1$ is obtuse, and vice versa. Moreover, neither of these angles equals $\pi/2$, for $T$ is not a right trapezoid.
   \begin{figure}[H]
    \centering
    \includegraphics[width=0.4\linewidth]{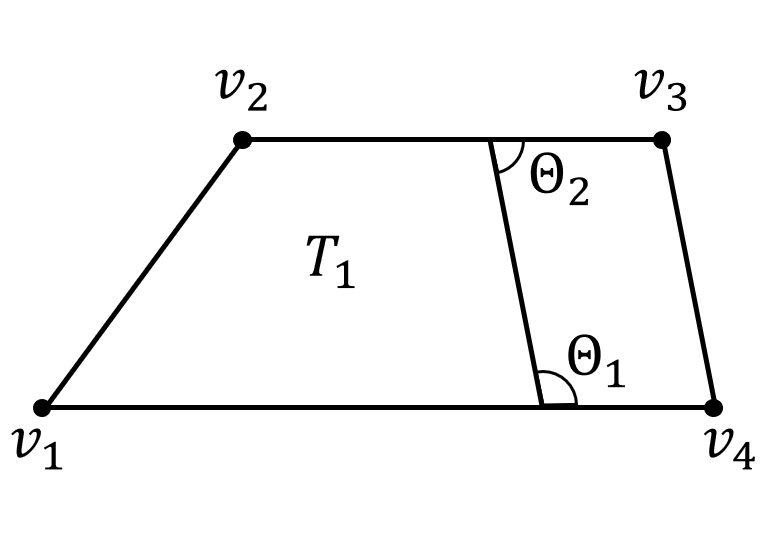}
    \caption{$M\left(T,1,\alpha\right)$ and the acute tile $T_1$}
    \label{fig:MTMCON1}
\end{figure}

    If $\Theta_1<\pi/2$, then $T$ is obtuse, and there are two possibilities for filling $\Theta_1$: filling it with one or more $\theta\left(\te\right)$'s or with a single $\psi\left(\te\right)$. This is because $\theta(T)$ is the smallest internal angle of $T$, and $\Theta_1$ has the same size as $\psi(T)$, the second smallest internal angle of $T$. For either possibility, there is a tile $T'\in\mathfrak{T}$ such that one of its edges lies on $\ssl{T_1}$ and $\Theta_1$ is filled with one or more angles including one acute internal angle of $T'$. Since $\mathfrak{T}$ is sub-to-sub, the former possibility permits only $\ml{T'}$ to lie on $\ssl{T_1}$. However, this cannot be achieved since $h<1$. On the other hand, for the latter possibility, either $\ub{T'}$ or $\ssl{T'}$ lies on $\ssl{T_1}$. If $\ub{T'}\subset\ssl{T_1}$, then a single $\theta\left(\te\right)$ fills the acute angle between $\ml{T'}$ and $\ssl{T_1}$. This implies that the main leg or the lower base of another tile in $\mathfrak{T}$ lies on $\ssl{T_1}$, which is a contradiction. Instead, if $\ssl{T'}\subset\ssl{T_1}$, then we obtain the cluster parallelogram $v_1v_2'v_2'v_1'$ tiled with $T_1$ and $T'$. Hence, the problem statement holds for the case $\Theta_1<\pi/2$.

    Next, suppose that $\Theta_1>\pi/2$. Then, $T$ is acute. Similar to the preceding case, $\Theta_2$ has the same size as $\psi(T)$ and is filled with one or more $\theta\left(\te\right)$'s or with a single $\psi\left(\te\right)$. If $T$ is not isosceles, then neither $\lb{\te}$ nor $\ml{\te}$ can lie on $\ssl{T_1}$, and thus $\ssl{T_1}$ is covered by the subsidiary leg of another tile in $\mathfrak{T}$. The same goes for the case when $T$ is isosceles except that a leg of another tile in $\mathfrak{T}$, whether it be the main leg or the subsidiary leg, can lie on the subsidiary leg of $T_1$. In either case, the parallelogram $v_1v_2'v_2'v_1'$ is a cluster in $\mathfrak{T}$, and thus the problem statement is true for $\Theta_1>\pi/2$. This ends the proof for $\rho=1$.

    Now, we assume $\rho\geq 2$. Since $b>a$, the interior of $\lb{T_2}$ contains an endpoint of $\ssl{T_1}$. We then follow the same argument for $\rho=1$ and observe that there is a tile $T_1'\neq T_1$ in $\mathfrak{T}$ whose subsidiary leg perfectly covers $\ssl{T_1}$. Moreover, from
    $$
    \left|\ub{T_1}\cup\lb{T_1'}\right|=\left|\ub{T_1}\right|+\left|\lb{T_1'}\right|=a+b>b=\left|\lb{T_2}\right|,
    $$
    it follows that the line segment $\ub{T_1}\cup\lb{T_1'}$ contains one endpoint of $\ssl{T_2}$ in its interior. The other endpoint is contained in the interior of $\ov{v_2}{v_3}$ if $\rho=2$, and that of $\lb{T_3}$ if $\rho\geq 3$. In either case, $\ssl{T_2}$ is perfectly covered by the subsidiary leg another tile $T_2'\neq T_2$ in $\mathfrak{T}$. Repeating this process, we can inductively select tiles $T_2', \cdots, T_{\rho}'$, along with $T_1'$, such that $T_i'\neq T_i$ and $\ssl{T_i'}=\ssl{T_i}$ for each $i=1, 2, \cdots, \rho$, and
    $$
    \lb{T_{j+1}}\cup\ub{T_{j+1}'}=\ub{T_j}\cup\lb{T_j'}
    $$
    for each $j=1, 2, \cdots, \rho-1$ (see Figure \ref{fig:MTMCON2}, for instance). We can observe that $\cup_{i=1}^{\rho}\ml{T_i'}$ is the closed line segment with two endpoints $v_1'$ and $v_2'$. Therefore, we conclude that the parallelogram $v_1v_2v_2'v_1'$ is a cluster in $\mathfrak{T}$ tiled with $T_1, T_2, \cdots, T_{\rho},$ $T_1', T_2', \cdots, T_{\rho}'$.

    \begin{figure}[H]
    \centering
    \includegraphics[width=0.5\linewidth]{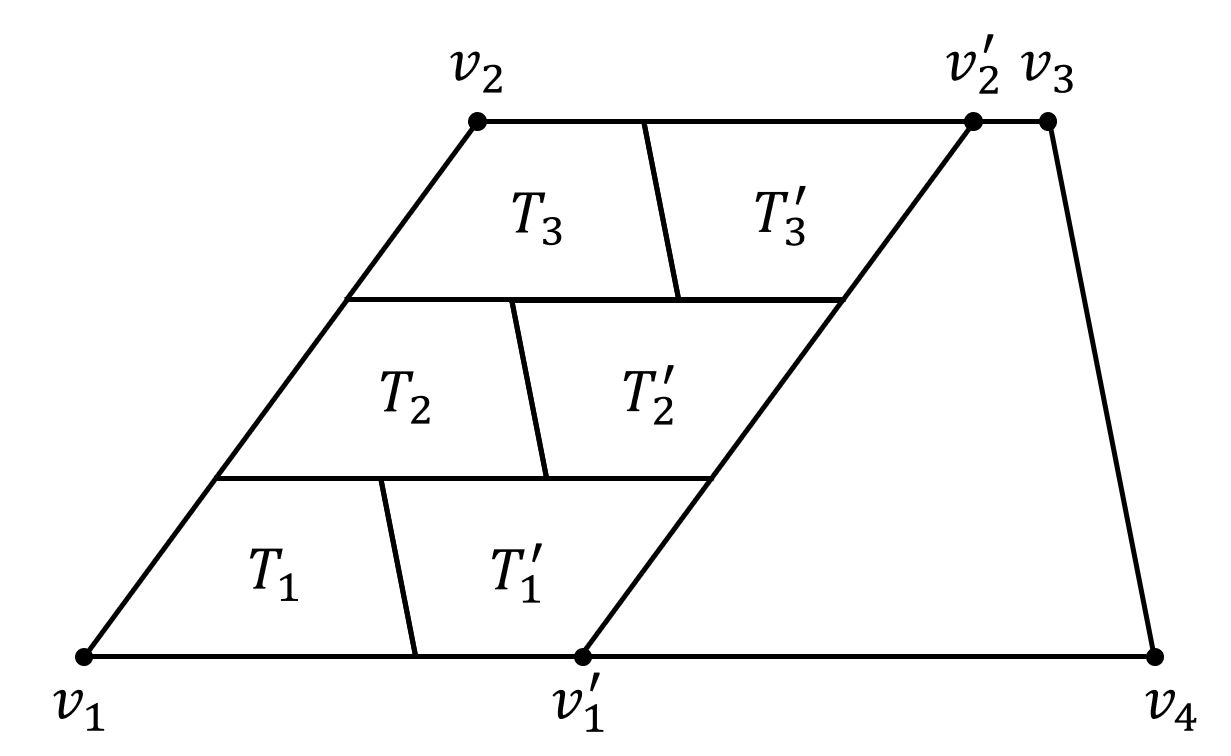}
    \caption{Cluster $v_1v_2v_2'v_1'$ when $\rho=3$}
    \label{fig:MTMCON2}
\end{figure}

\end{proof}

It turns out that MTM is sufficient enough to imply the non-existence of sub-to-sub tilings of $\hat{M}\left(T,\rho,\alpha\right)$ with congruent copies of $T$ when $\alpha\leq a+b$, as the following lemma shows.

\begin{lemma}\label{NOM}
    Let $\rho\geq 3$ be a positive integer, and $\alpha$ be a positive real number such that $\alpha\leq a+b$. If $T$ is not a right trapezoid and satisfies the main-to-main property, then there are no sub-to-sub tilings $\mathfrak{T}$ of the trapezoid $\hat{M}\left(T,\rho,\alpha\right)$ with congruent copies of $T$.
\end{lemma}

\begin{proof}
    For the sake of contradiction, assume that such a tiling $\mathfrak{T}$ exists. If $a<\alpha<a+b$, then it follows from Lemma \ref{L1.1} that the parallelogram $v_1v_2v_2'v_1'$ is a cluster in $\mathfrak{T}$ but it is not contained in $\hat{M}\left(T,\rho,\alpha\right)$; this is a contradiction. We also obtain the same result for $\alpha<a$ because at least one $\ub{\te}$ lies on the upper base of $\hat{M}\left(T,\rho,\alpha\right)$ due to Lemma \ref{L1}.
    
    Next, assume $\alpha=a$. By Lemma \ref{L1}, we can select tiles $T_1, T_2, \cdots, T_{\rho}\in\mathfrak{T}$ satisfying the three properties in Lemma \ref{L1} with $u_0,u_1,$ and $u_2$ replaced to $v_4,v_1,$ and $v_2$, respectively. Put $\mathfrak{T}_0=\mathfrak{T}\backslash\{T_{\rho}\}$. Here, $\mathfrak{T}_0$ is a sub-to-sub tiling of the trapezoid $\hat{M}\left(T,\rho, a\right)\backslash T_{\rho}$ congruent to $\hat{M}\left(T,\rho-1, b\right)$ with congruent copies of $T$. Then, by following the argument for $a<\alpha<a+b$, we derive a contradiction. Therefore, no tilings of $\hat{M}\left(T,\rho,\alpha\right)$ with congruent copies of $T$ are sub-to-sub if $\alpha<a+b$.

    The only remaining case is $\alpha=a+b$. From Lemma \ref{L1.1}, it follows that the parallelogram $v_1v_2v_2'v_1'$ is a cluster in $\mathfrak{T}$; let $\mathfrak{T}_1$ be the collection of all tiles in $\mathfrak{T}$ not contained in the cluster $v_1v_2v_2'v_1'$. Then, $\mathfrak{T}_1$ is a tiling of the triangle $v_1'v_2'v_4$. Moreover, by Lemma \ref{L1}, there is a tile in $\mathfrak{T}_1$ whose main leg is lying on $\ov{v_1'}{v_2'}$ and whose obtuse internal angle of size $\pi-\theta(T)$ fills the internal angle $\angle v_1'v_2'v_4$ of the triangle $v_1'v_2'v_4$ possibly along with other internal angles of some tiles in $\mathfrak{T}_1$. This implies that
    $$
    \pi-\theta(T)\leq\angle v_1'v_2'v_4<\pi-\psi(T),
    $$
    which is a contradiction since $\theta(T)\leq\psi(T)$. Hence, the problem statement holds for $\alpha=a+b$. This ends the proof.
\end{proof}

Combining Lemma \ref{L1.1} and \ref{NOM}, we can generalize the result of Lemma \ref{NOM} to trapezoids $\hat{M}\left(T,\rho,\alpha\right)$ with arbitrary $\alpha>0$ as follows.

\begin{corollary}\label{ENDM}
    Let $\rho\geq 3$ be a positive integer, and $\alpha$ be a positive real number. Suppose that $T$ is not a right trapezoid and satisfies the main-to-main property. Then, there are no sub-to-sub tilings $\mathfrak{T}$ of the trapezoid $\hat{M}\left(T,\rho,\alpha\right)$ with congruent copies of $T$.
\end{corollary}
\begin{proof}
    Since $\alpha>0$ and $a+b>0$, we can express $\alpha$ in a form $q(a+b)+r$, where $q$ is a non-negative integer and $r$ is a non-negative real number such that $r<a+b$. We separately consider two cases: $r>0$ and $r=0$.
    
    For the former case, fix $r$ to some positive number $r'<a+b$. We then use induction over the integer $q$. The case $q=0$ immediately follows from Lemma \ref{NOM}. Choose a positive integer $n$, and suppose that the problem statement holds when $(q,r)=\left(n-1,r'\right)$. We further assume, for the sake of contradiction, that there is a sub-to-sub tiling $\mathfrak{T}_n$ of $\hat{M}\left(T,\rho,n(a+b)+r'\right)$ with congruent copies of $T$. Then, by Lemma \ref{L1.1}, the parallelogram $v_1v_2v_2'v_1'$ is a cluster in $\mathfrak{T}_n$, where the points $v_i$ and $v_i'$ are those given in Lemma \ref{L1.1}. Let $\mathfrak{T}_{n-1}$ be the collection of tiles in $\mathfrak{T}$ not contained in the cluster $v_1v_2v_2'v_1'$. We can observe that $\mathfrak{T}_{n-1}$ is a sub-to-sub tiling of the remaining region of $\hat{M}\left(T,\rho,n(a+b)+r'\right)$ not covered by the cluster $v_1v_2v_2'v_1'$, which is a trapezoid congruent to $\hat{M}\left(T,\rho,(n-1)(a+b)+r'\right)$. This contradicts the induction hypothesis, and thus we conclude that the problem statement also holds for $(q,r)=\left(n,r'\right)$. This completes the induction.

    The latter case can be proven analogously. The initial case $(q,r)=(1,0)$ follows from Lemma \ref{L1.1}, and the further inductive steps can be proven in the same way as the case $r>0$. This ends the proof.
\end{proof}

\subsection{Sub-to-sub property}\label{STSP}
The MTM property in the preceding section can be understood as a constraint for perfectly covering the main leg of some large trapezoid whose main acute angle has the same size as $\theta(T)$. We introduce a subsidiary-leg version of this property, provided that $T$ is not an isosceles trapezoid. Here, we exclude isosceles trapezoids from our discussion because an isosceles trapezoid's main leg and subsidiary leg can be interchanged. Before we go on, we define the following functions assigning to each non-right trapezoid $Q$ an edge of $Q$.
\begin{align*}
    E(Q)&=\begin{cases}
        \lb{Q} &\text{if $Q$ is acute}\\
        \ub{Q} &\text{if $Q$ is obtuse}
    \end{cases}\\
    F(Q)&=\begin{cases}
        \ub{Q} &\text{if $Q$ is acute}\\
        \lb{Q} &\text{if $Q$ is obtuse}
    \end{cases}
\end{align*}

Suppose that $T$ is not an isosceles trapezoid. Let $\rho$ be a positive integer and $\alpha$ be a positive real number. In addition, let $r_0,r_1,$ and $r_2$ be points on $\mathds{R}^2$ such that $\left|\ov{r_0}{r_1}\right|=\alpha$, $\left|\ov{r_1}{r_2}\right|=\rho h$, and $\angle r_0r_1r_2=\psi(T)$, where we put $\psi(T)=\pi/2$ if $T$ is a right trapezoid; let $S^{\circ}\left(T,\rho,\alpha\right)$ be a polygon whose boundary contains the polygonal chain $\left[r_0,r_1,r_2\right]$ and whose region contains the triangle $r_0r_1r_2$. We will say that $T$ satisfies the \textbf{sub-to-sub property} (abbreviated to \textbf{STS}) if it satisfies the following condition: for every positive integer $\rho$, positive real number $\alpha$, and polygon $S^{\circ}\left(T,\rho,\alpha\right)$ corresponding to these two numbers, if there is a tiling $\mathfrak{T}$ of $S^{\circ}\left(T,\rho,\alpha\right)$ with congruent copies of $T$, and if one $\lb{\te}$ lies on the line segment $\ov{r_1}{r_2}$, then $\ov{r_1}{r_2}$ is not perfectly covered by edges of tiles in $\mathfrak{T}$. As the following lemma shows, STS uniquely determines how edges of tiles in a given tiling perfectly cover a given line segment, which is somewhat analogous to what MTM implies for $\ov{u_1}{u_2}$.

\begin{lemma}\label{L1.1.1}
    Suppose that $T$ is not an isosceles trapezoid and satisfies the sub-to-sub property. Let $\rho$ be a positive integer and $\alpha$ be a positive real number. Put
    \begin{align*}
    s_0&=\left(\alpha, 0\right),\\
    s_1&=\left(0,0\right),\\
    s_2&=\left(\rho h\cos{\psi(T)},\rho h\sin{\psi(T)}\right),\\
    s_3&=\left(\rho h\cos{\psi(T)}+\alpha,\rho h\sin{\psi(T)}\right),
    \end{align*}
    each of which a point in $\mathds{R}^2$, and let $S\left(T,\rho,\alpha\right)$ be a polygon whose boundary contains the polygonal chain $\left[s_0,s_1,s_2,s_3\right]$ and whose region contains the parallelogram $s_0s_1s_2s_3$. If $\mathfrak{T}$ is a tiling of $S\left(T, \rho, \alpha\right)$ with congruent copies of $T$, then there are tiles $T_1,\cdots,T_{\rho}\in\mathfrak{T}$ satisfying the following (see Figure \ref{fig:L1.1}, for instance):
    \begin{enumerate}
        \item $\ssl{T_1},\ssl{T_2},\cdots,\ssl{T_{\rho}}$ perfectly covers $\ov{s_1}{s_2}$ in the order
        $$
        \left(\ssl{T_1},\ssl{T_2},\cdots, \ssl{T_{\rho}}\right);
        $$
        \item $\ov{s_0}{s_1}\subset E\left(T_1\right)$ or vice versa, if $T$ is not a right trapezoid;
        \item $\ov{s_2}{s_3}\subset F\left(T_{\rho}\right)$ or vice versa, if $T$ is not a right trapezoid;
        \item $s_1\in T_1$ and $s_2\in T_{\rho}$, if $T$ is a right trapezoid;
        \item $\ub{T_i}\varsubsetneq\lb{T_{i+1}}$ for each $i=1, 2, \cdots, \rho-1$, if $T$ is acute;
        \item $\ub{T_{i+1}}\varsubsetneq\lb{T_i}$ for each $i=1, 2, \cdots, \rho-1$, if $T$ is obtuse.
    \end{enumerate}
    
\begin{figure}[ht]
\captionsetup[subfigure]{labelformat=empty}
      \centering
	   \begin{subfigure}{0.3\linewidth}
		\includegraphics[width=\linewidth]{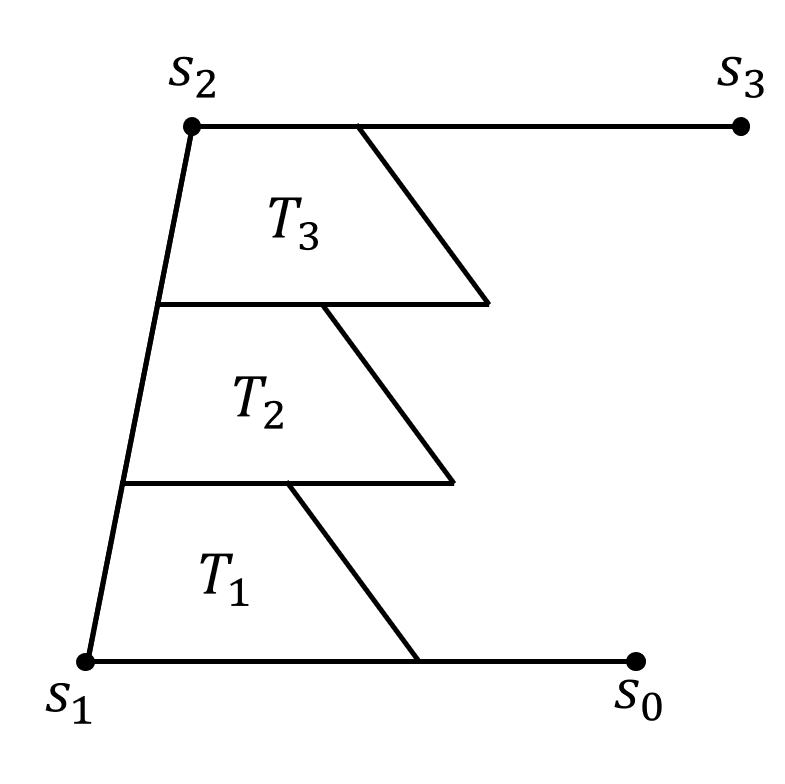}
		\caption{$T$ is acute}
		\label{fig:L1.1a}
	   \end{subfigure}
	   \begin{subfigure}{0.3\linewidth}
		\includegraphics[width=\linewidth]{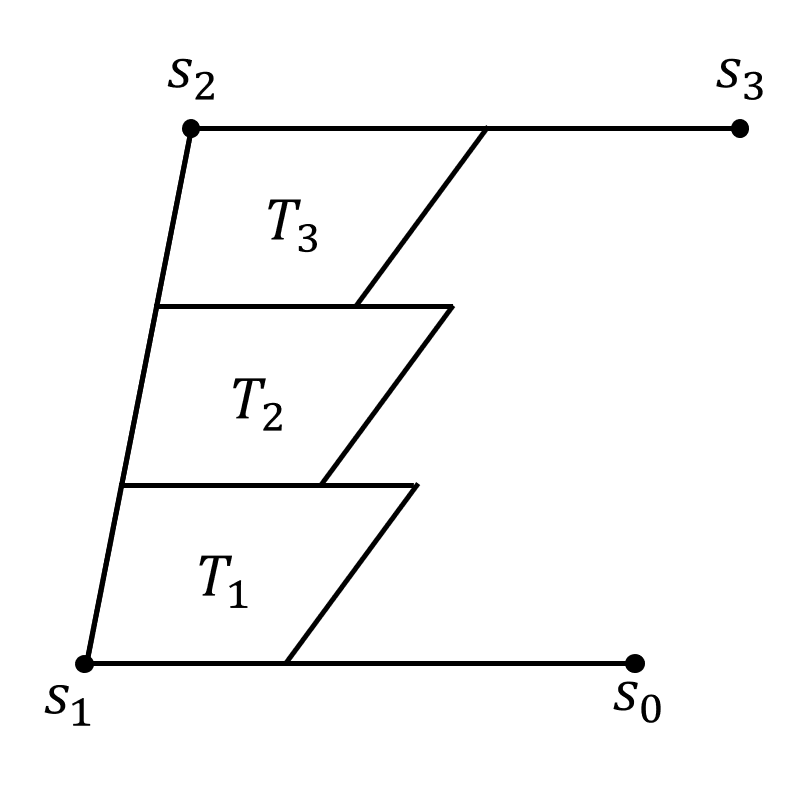}
		\caption{$T$ is obtuse}
		\label{fig:L1.1o}
	    \end{subfigure}
     	   \begin{subfigure}{0.3\linewidth}
		\includegraphics[width=\linewidth]{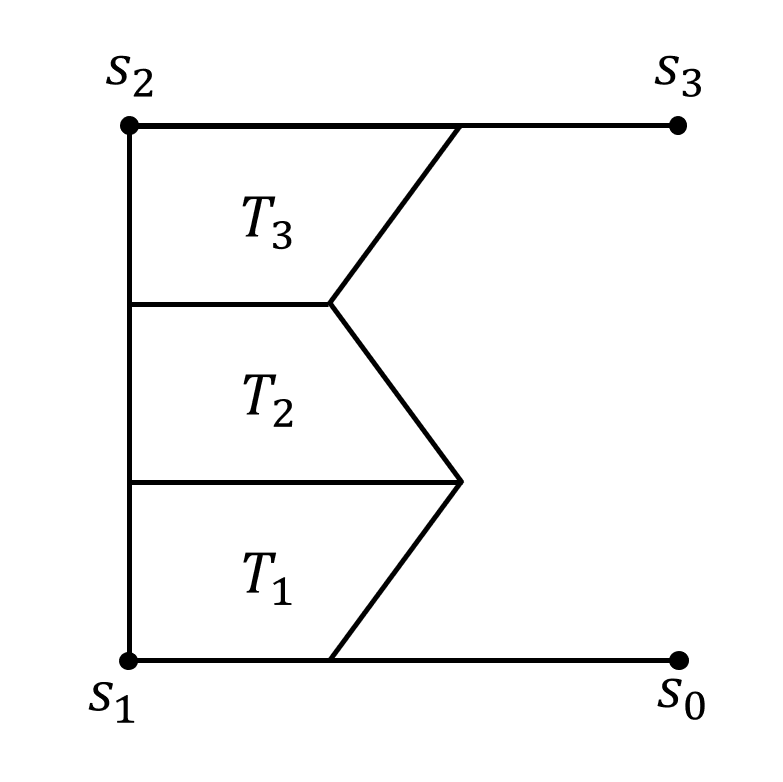}
		\caption{$T$ is a right traepzoid}
		\label{fig:L1.1r}
	    \end{subfigure}
	\caption{$S\left(T,3,\alpha\right)$ and three tiles $T_1,T_2,$ and $T_3$}
	\label{fig:L1.1}
\end{figure}

\end{lemma}

\begin{proof}
    We first show that no $\ml{\te}$ can lie on $\ov{s_1}{s_2}$. Assume, for the sake of contradiction, that some tile $T'\in\mathfrak{T}$ has its main leg lying on $\ov{s_1}{s_2}$. Since both the internal angle $\angle s_0s_1s_2$ and $\angle s_1s_2s_3$ of $S(T,\rho,\alpha)$ have the size smaller than $\pi-\theta\left(T'\right)$, neither $s_1$ nor $s_2$ is contained in $\ub{T'}\cap\ml{T'}$. Thus, a single $\theta\left(\te\right)$ fills the acute angle between $\ub{T'}$ and $\ov{s_1}{s_2}$; let $T''\in\mathfrak{T}$ be the tile whose main acute angle fills that acute angle. Then, either $\lb{T''}$ or $\ml{T''}$ lies on $\ov{s_1}{s_2}$. Considering that $T$ satisfies STS, we have $\ml{T''}\subset\ov{s_1}{s_2}$. Similar to $T'$, this gives us the acute angle between $\ub{T''}$ and $\ov{s_1}{s_2}$ which is filled with a single $\theta\left(\te\right)$. Repeating this process as much as possible, we obtain a closed line segment that lies on $\ov{s_1}{s_2}$ and no edges of tiles in $\mathfrak{T}$ can cover, which is a contradiction. Therefore, $\ml{\te}$ cannot lie on $\ov{s_1}{s_2}$.

    Next, we claim that $\ub{\te}$ cannot lie on $\ov{s_1}{s_2}$. Assume that the upper base of some tile $T'''\in\mathfrak{T}$ in $\mathfrak{T}$ lies on $\ov{s_1}{s_2}$. Then, from the argument for $T'$, it follows that the acute angle between $\ml{T'''}$ and $\ov{s_1}{s_2}$ is filled with a single $\theta\left(\te\right)$. However, this implies that at least one of $\lb{\te}$ and $\ml{\te}$ lies on $\ov{s_1}{s_2}$, which is a contradiction. Hence, no $\ub{\te}$ can lie on $\ov{s_1}{s_2}$.

    Now, we prove the problem statement. From the previous discussion and the assumption that $T$ satisfies STS, it follows that no edges of tiles in $\mathfrak{T}$ other than $\ssl{\te}$ can lie on $\ov{s_1}{s_2}$. This immediately shows that the problem statement holds when $T$ is a right trapezoid and when $\rho=1$. Thus, we may assume $T$ is not a right trapezoid and $\rho\geq 2$. The acute angle $\psi(T)$ is the second smallest internal angle of $T$, so the angle $\angle s_0s_1s_2$ is filled with one or more $\theta\left(\te\right)$'s or a single $\psi\left(\te\right)$. Considering that only $\ssl{\te}$ can lie on $\ov{s_1}{s_2}$, it follows that there is a tile $T_1\in\mathfrak{T}$ such that $\ssl{T_1}\subset\ov{s_1}{s_2}$ and $\psi\left(T_1\right)$ fills the internal angle $\angle s_0s_1s_2$ of $S(T,\rho,\alpha)$. Here, we have $\ov{s_0}{s_1}\subset E\left(T_1\right)$ if $\alpha\leq\left|E\left(T\right)\right|$, and vice versa otherwise. Similar to $T_1$, we can select a tile $T_2\in\mathfrak{T}$ such that $\psi\left(T_2\right)$ fills the acute angle of size $\psi(T)$ between $\ov{s_1}{s_2}$ and $F\left(T_1\right)$, the edge $\ssl{T_2}$ lies on $\ov{s_1}{s_2}$, and $\ssl{T_1}\cup\ssl{T_2}$ is a one-point set. Moreover, since $b>a$, we have $\ub{T_1}\varsubsetneq\lb{T_2}$ if $T$ is acute, and $\ub{T_2}\varsubsetneq\lb{T_1}$ if $T$ is obtuse. Inductively, we can select tiles $T_2, \cdots, T_{\rho}\in\mathfrak{T}$, along with $T_1$, satisfying the desired properties. In addition, we have $\ov{s_2}{s_3}\subset F\left(T_{\rho}\right)$ if $\alpha\leq\left|F\left(T\right)\right|$, and vice versa otherwise. This ends the proof.
\end{proof}

Given a positive integer $\rho$ and a positive real number $\alpha$, let $\hat{S}\left(T,\rho,\alpha\right)$ be the trapezoid $t_1t_2t_3t_4$, where
    \begin{align*}
        t_1&=(0,0),\\
        t_2&=\left(\rho h\cos{\psi(T)},\rho h\sin{\psi(T)}\right),\\
        t_3&=\left(f(T)+\alpha,\rho h\sin{\psi(T)}\right)\\
        t_4&=\left(g(T)+\alpha,0\right)
    \end{align*}
    are points in $\mathds{R}^2$ and $f(T)$ and $g(T)$ are defined as follows.
   \begin{align*}
       f(T)&=\begin{cases}
        \rho\cos\theta(T) &\text{if $T$ is obtuse}\\   
        \rho h\cos{\psi(T)} &\text{otherwise}
       \end{cases},\\
       g(T)&=\begin{cases}
        0 &\text{if $T$ is obtuse}\\  
        \rho \left(h\cos{\psi(T)}+\cos\theta(T)\right) &\text{otherwise}  
       \end{cases}.\\
   \end{align*}
As an analog of MTM, if we further assume an additional property for a tiling $\mathfrak{T}$, then we can construct a cluster parallelogram; we will derive a result analogous to Corollary \ref{ENDM} based on such a construction. As MTM was closely related to the non-existence of sub-to-sub tilings of $\hat{M}(T,\rho,\alpha)$ with congruent copies of $T$, STS deals with the non-existence of certain tilings of trapezoids $\hat{S}\left(T,\rho,\alpha\right)$ with congruent copies of $T$. The two following lemmas show that we can always construct the desired cluster parallelogram in two cases: $T$ is not a $\pi/3$-right trapezoid and a given tiling is main-to-main, or $T$ is a $\pi/3$-right trapezoid and a given tiling is strictly main-to-main.

\begin{lemma}\label{L1.1.1.1}
   Let $\rho$ be a positive integer, and $\alpha$ be a positive real number such that $\alpha>a$. Suppose that $T$ is neither a right trapezoid nor an isosceles trapezoid and has the sub-to-sub property. If $\mathfrak{T}$ is a main-to-main tiling of $\hat{S}\left(T,\rho,\alpha\right)$ with congruent copies of $T$, then the parallelogram $t_1t_2t_2't_1'$ is a cluster in $\mathfrak{T}$, where $t'_i=t_i+\left(a+b,0\right)$ for each $i=1, 2$.
\end{lemma}

\begin{proof}
   Let $\mathfrak{T}$ be a main-to-main tiling of the trapezoid $\hat{S}\left(T,\rho,\alpha\right)$ with congruent copies of $T$, and let $T_1, T_2, \cdots, T_{\rho}\in\mathfrak{T}$ be tiles satisfying the five properties in Lemma \ref{L1.1.1} with $s_0,s_1,s_2,$ and $s_3$ replaced to $t_1,t_2,t_3,$ and $t_4$, respectively. Note that no $T_i$ intersects with $\ov{t_3}{t_4}$. We first prove the problem statement for the case when $T$ is acute. If $\rho=1$, then the two line segments $\ov{t_1}{t_4}$ and $\ov{t_2}{t_3}$ both contain an endpoint of $\ml{T_1}$ in their interior. Since the acute angle between $\ml{T_1}$ and $\ov{t_2}{t_3}$ is of size $\theta(T)$, and since $\mathfrak{T}$ is main-to-main, the main leg of another tile $T'\in\mathfrak{T}$ lies on $\ml{T_1}$. This gives us the cluster parallelogram $t_1t_2t_2't_1'$ tiled with $T_1$ and $T'$.

   Next, suppose that $\rho\geq 2$. Similar to the case $\rho=1$, the interior of $\lb{T_2}$ contains an endpoint of $\ml{T_1}$, and thus the main leg of another tile $T_1'\in\mathfrak{T}$ perfectly covers $\ml{T_1}$. As we did in the proof of Lemma \ref{L1.1}, due to the fifth property of $T_i$'s in Lemma \ref{L1.1.1}, we then can inductively select tiles $T_2',T_3',\cdots,T_{\rho}'\in\mathfrak{T}$, along with $T_1'$, such that $T_i'\neq T_i$ and $\ml{T_i'}=\ml{T_i}$ for each $i=1, 2, \cdots, \rho$, and 
   $$
   \lb{T_{j+1}}\cup\ub{T_{j+1}'}=\ub{T_j}\cup\lb{T_j'}
   $$
   for each $j=1, 2, \cdots, \rho-1$. By this, we obtain the parallelogram $t_1t_2t_2't_1'$ which is a cluster in $\mathfrak{T}$ tiled with $T_1,T_2,\cdots,T_{\rho}$,$T_1',T_2'\cdots,T_{\rho}'$; the following figure illustrates the cluster when $\rho=3$.

   \begin{figure}[H]
    \centering
    \includegraphics[width=0.5\linewidth]{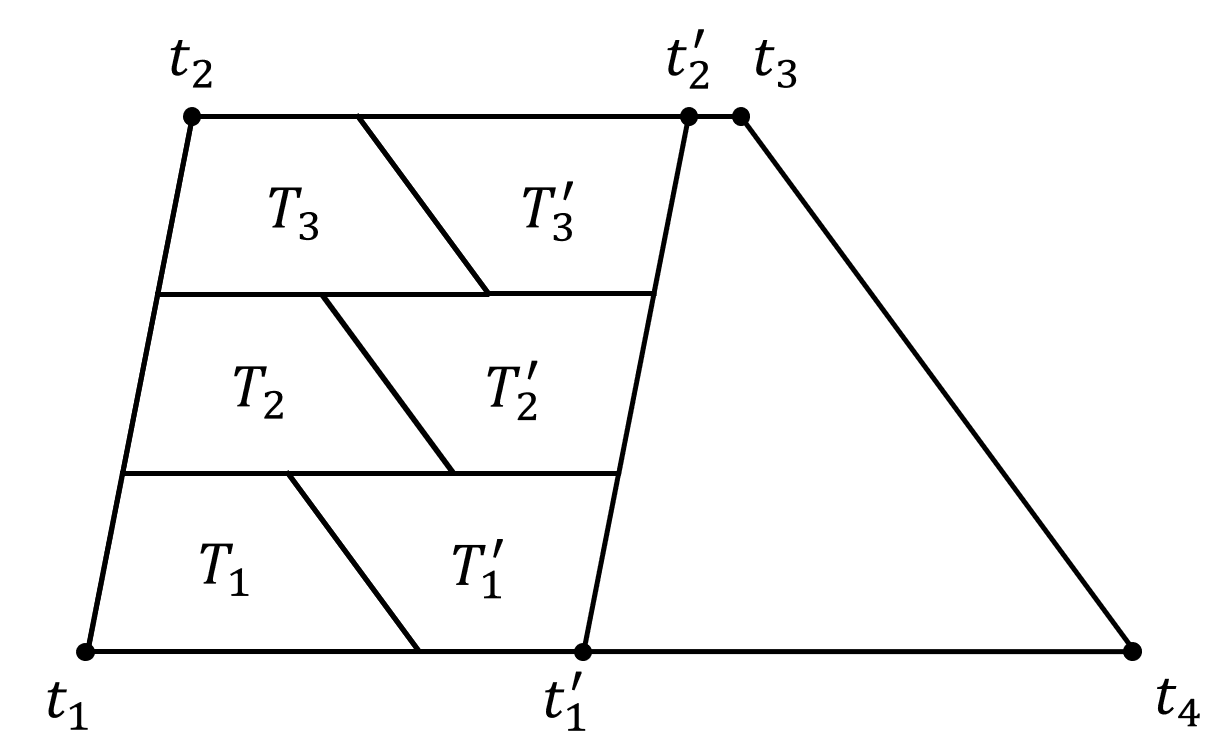}
    \caption{Cluster $t_1t_2t_2't_1'$ when $\rho=3$ and $T$ is acute}
    \label{fig:L1.1.1.1}
\end{figure}

   Now, we deal with the case when $T$ is obtuse. The situation is almost the same as in the preceding case, but there is a slight difference. First of all, $\ov{t_2}{t_3}$ and $\ov{t_1}{t_4}$ properly contain $\lb{T_{\rho}}$ and $\ub{T_1}$, respectively. Moreover, when $\rho\geq 2$, it follows from the sixth property of $T_i$'s in Lemma \ref{L1.1.1} that for each $i=1, 2, \cdots, \rho-1$, one endpoint of $\ml{T_{i+1}}$ is contained in the interior of $\lb{T_i}$ and the acute angle between $\ml{T_{i+1}}$ and $\lb{T_i}$ has the same size as $\theta(T)$. Thus, unlike the preceding case where the inductive selection of $T_1',T_2',\cdots,T_{\rho}'$ was done from $i=1$ to $i=\rho$, we inductively select analogous tiles from $i=\rho$ to $i=1$ when $T$ is obtuse. We then obtain tiles $T_1'',T_2'',\cdots,T_{\rho}''\in\mathfrak{T}$ such that $T_i''\neq T_{\rho+1-i}$ and $\ml{T_i''}=\ml{T_{\rho+1-i}}$ for each $i=1, 2, \cdots, \rho$, and 
   $$
   \lb{T_{\rho-j}}\cup\ub{T_{j+1}'}\cup=\ub{T_{\rho+1-j}}\cup\lb{T_j'}
   $$
   for each $j=1, 2, \cdots, \rho-1$. This gives us the desired cluster parallelogram $t_1t_2t_2't_1'$ tiled with $T_1,T_2,\cdots,T_{\rho}$, $T_1'',T_2'',\cdots,T_{\rho}''$; see Figure \ref{fig:L1.1.1.1-1}, for instance. This ends the proof.
      \begin{figure}[H]
    \centering
    \includegraphics[width=0.5\linewidth]{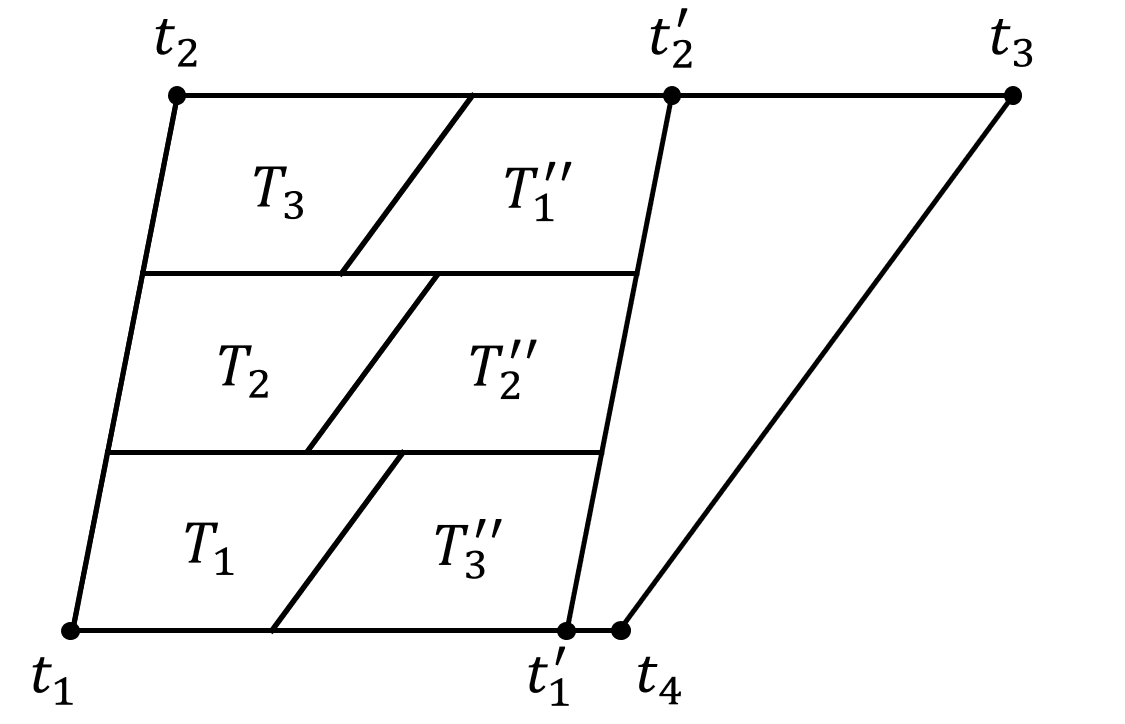}
    \caption{Cluster $t_1t_2t_2't_1'$ when $\rho=3$ and $T$ is obtuse}
    \label{fig:L1.1.1.1-1}
\end{figure}
\end{proof}

\begin{lemma}\label{L1.1.1.2}
    Let $\rho$ be a positive integer, and $\alpha$ be a positive number such that $\alpha>a$. Suppose that $T$ is a right trapezoid and satisfies a sub-to-sub property. Given a tiling $\mathfrak{T}$ of $\hat{S}\left(T,\rho,\alpha\right)$, the rectangle $t_1t_2t_2't_1'$ is a cluster in $\mathfrak{T}$, where $t_i'=t_i+(a+b,0)$ for each $i=1,2$, if one of the following holds.
    \begin{enumerate}
        \item $\theta(T)\neq\pi/3$ and $\mathfrak{T}$ is main-to-main;
        \item $\theta(T)=\pi/3$ and $\mathfrak{T}$ is strictly main-to-main.
    \end{enumerate}
    
\end{lemma}
\begin{proof}
    Let $\mathfrak{T}$ be a main-to-main tiling of $\hat{S}\left(T,\rho,\alpha\right)$. By Lemma \ref{L1.1.1}, we can select tiles $T_1,T_2,\cdots,T_{\rho}\in\mathfrak{T}$ such that subsidiary legs perfectly cover $\ov{t_1}{t_2}$ in the order
    $$
    \left(\ssl{T_1},\ssl{T_2},\cdots,\ssl{T_{\rho}}\right).
    $$
    and two points $t_1$ and $t_2$ are contained in $T_1$ and $T_{\rho}$, respectively. It suffices to show that each $\ml{T_i}$ is perfectly covered by the main leg of another tile in $\mathfrak{T}$ whose bases are parallel to those of $T_i$. If $\rho=1$, then there is a base of $\hat{S}\left(T,1,\alpha\right)$ such that one endpoint of $\ml{T_1}$ is contained in the interior of that base and the acute angle between $\ml{T_1}$ and the base, not an internal angle of $T$, has the same size as $\theta(T)$. Since $\mathfrak{T}$ is main-to-main, $\ml{T_1}$ is perfectly covered by the main leg of another tile in $\mathfrak{T}$. Hence, the problem statement holds when $\rho=1$.

    Next, assume that $\rho\geq 2$. Choose an integer $i=2, 3, \cdots, \rho$. We prove the following four claims.

    \begin{enumerate}
        \item Suppose that $T_{i-1}$ is negative and $\ml{T_{i-1}}$ is properly covered by edges of some tiles in $\mathfrak{T}$. Then, $T_i$ is positive. Moreover, if $i<\rho$, then $\ml{T_{i}}$ is perfectly covered by the main leg of another tile whose bases are not parallel to those of $T_i$;
        \item If $T_{\rho-1}$ is negative and $T_{\rho}$ is positive, then $\ml{T_{\rho-1}}$ is perfectly covered by edges of some tiles in $\mathfrak{T}$;
        \item If $\theta(T)\neq\pi/3$ and $T_{i-1}$ is positive, then $\ml{T_{i-1}}$ cannot be perfectly covered by the main leg of another tile whose bases are not parallel to those of $T_{i-1}$;
        \item Suppose that $\theta(T)=\pi/3$ and $\mathfrak{T}$ is strictly main-to-main. If $T_{i-1}$ is positive and $\ml{T_{i-1}}$ is perfectly covered by the main leg of another tile whose bases are not parallel to those of $T_{i-1}$, then $T_i$ is negative and $\ml{T_i}$ is properly covered by edges of some tiles in $\mathfrak{T}$.
    \end{enumerate}

    \paragraph{Claim 1:} Let $L$ be a closed line segment which is a union of some edges of tiles in $\mathfrak{T}$ properly covering $\ml{T_{i-1}}$. Note that the interior of $L$ contains the vertex of $T_{i-1}$ contained in $\ml{T_{i-1}}\cap\lb{T_{i-1}}$. If $T_i$ is negative, then the closure $M$ of $\lb{T_{i-1}}\backslash\ub{T_i}$ is a line segment of length $b-a=\cos\theta(T)$ shorter than $\ml{\te}$ and $\lb{\te}$. Thus, considering that the acute angle between $\ml{T_i}$ and $M$, either $\ml{\te}$ or $\lb{\te}$ properly contains the closed line segment $M$. This implies that $\ml{T_{i-1}}$ is perfectly covered by edges of some tiles in $\mathfrak{T}$ other than $T_{i-1}$, which is a contradiction because of $L$. Hence, $T_i$ is positive.
    
    Now, we show that $\ml{T_i}$ is perfectly covered in the desired way if $i<\rho$. Due to vacuous truth, we may assume $\rho\geq 3$. Regardless of the sign of $T_{i+1}$, $\ml{T_i}$ is perfectly covered by edges of some tiles in $\mathfrak{T}$ because of $L$ and some edge of $T_{i+1}$ (in particular, $\lb{T_{i+1}}$ if $T_{i+1}$ is positive, and $\ml{T_{i+1}}$ if it is negative). Moreover, since $\mathfrak{T}$ is main-to-main, $\lb{\te}$ cannot lie on $\ml{T_i}$. If $T_{i+1}$ is positive, then the acute angle between $\ml{T_i}$ and $\lb{T_{i+1}}$ is filled with a single $\theta\left(\te\right)$, and thus we can observe that $\ml{T_i}$ is perfectly covered by the main leg of some negative tile in $\mathfrak{T}$. However, this implies that $\ml{T_{i-1}}$ is perfectly covered by edges of some tiles in $\mathfrak{T}$ other than $T_{i-1}$, which is a contradiction. Thus, $T_{i+1}$ is negative.

      \begin{figure}[H]
    \centering
    \includegraphics[width=0.4\linewidth]{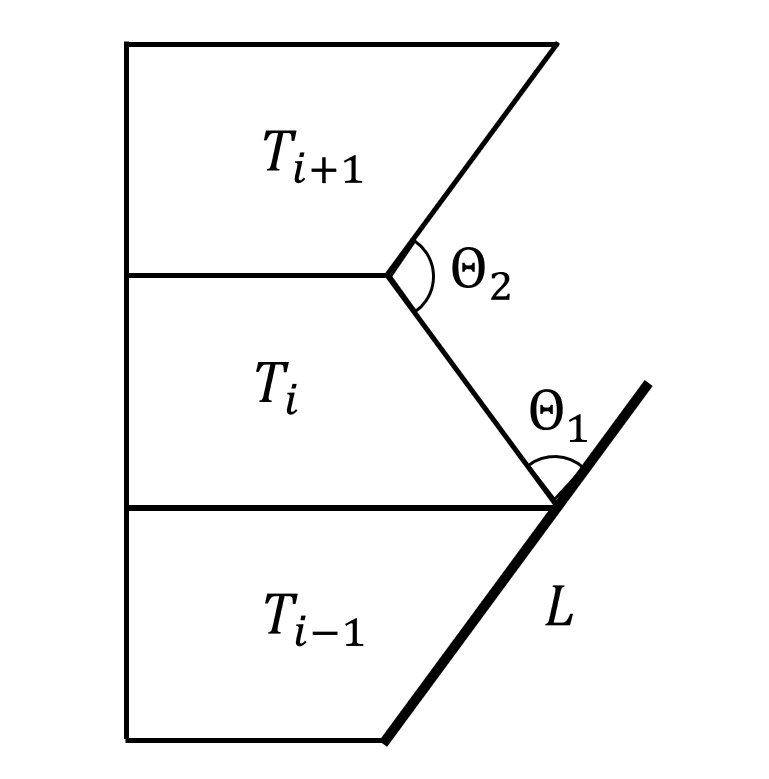}
    \caption{Two angles $\Theta_1$ and $\Theta_2$}
    \label{fig:L1.1.1.2-1}
\end{figure}
    
    Let $\Theta_1$ be the angle between $L$ and $\ml{T_i}$ which is not filled with $\theta\left(T_{i-1}\right)$ and $\theta\left(T_i\right)$, and $\Theta_2$ be the angle between $\ml{T_{i+1}}$ and $\ml{T_i}$ (see Figure \ref{fig:L1.1.1.2-1}). Since $\ml{T_{i+1}}$ and $L$ are parallel, either $\Theta_1$ or $\Theta_2$ is acute when $\theta(T)\neq\pi/4$. Then, considering that $\mathfrak{T}$ is main-to-main, $\ml{T_i}$ is perfectly covered by the main leg of some tile $T_i'\neq T_i$ in $\mathfrak{T}$ when $\theta(T)\neq\pi/4$. Here, $T_i'$ cannot be negative; otherwise, its interior intersects with $L$. Thus, we obtain the desired result. To show the same for the case $\theta(T)=\pi/4$, assume, for the sake of contradiction, that $\ml{T_i}$ is perfectly covered only with $\ssl{\te}$'s and $\ub{\te}$'s. Since $\Theta_1=\Theta_2=\pi/2$, either $\ml{\te}$ or $\lb{\te}$ lies on $\ml{T_i}$ as a proper subset of $\ml{T_i}$ if $\ub{\te}$ lies on $\ml{T_i}$. In addition, there are no integers $n$ satisfying $n\left|\ssl{T}\right|=\left|\ml{T}\right|$, for $\left|\ssl{T}\right|=1/\sqrt{2}$ is irrational while $\left|\ml{T}\right|=1$ is not. Hence, we derive a contradiction. For this reason, when $\theta(T)=\pi/4$, the edge $\ml{T_i}$ is perfectly covered in the desired way. Therefore, the first claim is true.

    \paragraph{Claim 2:} Assume, for the sake of contradiction, that $\ml{T_{\rho-1}}$ is properly covered by edges of some tiles in $\mathfrak{T}$. Then, $\ml{T_{\rho}}$ is perfectly covered by edges of some tiles in $\mathfrak{T}$ other than $T_{\rho}$. Since $\ov{t_2}{t_3}$ contains an endpoint of $\ml{T_{\rho}}$ in its interior, the acute angle between $\ov{t_2}{t_3}$ and $\ml{T_{\rho}}$ is filled with a single $\theta\left(\te\right)$. From the assumption that $\mathfrak{T}$ is main-to-main, it then follows that $\ml{T_{\rho}}$ is perfectly covered by the main leg of some negative tile in $\mathfrak{T}$. However, this implies that edges of some tiles in $\mathfrak{T}$ other than $T_{\rho-1}$ perfectly covers $\ml{T_{\rho-1}}$, which is a contradiction. Hence, the second claim holds.

    \paragraph{Claim 3:} This is clear if $T_i$ is positive. Thus, we assume that $T_i$ is negative. Let $\Theta$ be the angle between $\ml{T_{i-1}}$ and $\ml{T_i}$ not filled with two obtuse internal angles of $T_{i-1}$ and $T_i$ (see Figure \ref{fig:L1.1.1.2-2}). Since $\theta(T)\neq\pi/3$, and since $\theta(T)<\pi/2$, we have $\Theta\neq\pi-\theta(T)$ and
    $$
    \Theta=2\cdot\theta(T)<\pi<\left(\pi-\theta(T)\right)+\theta(T).
    $$
    Considering that $\theta(T)$ is the smallest internal angle of $T$, this implies that $\Theta$ cannot be filled with internal angles of tiles in $\mathfrak{T}$ including at least one obtuse internal angle of $\te$; that is, $\Theta$ is filled with two $\theta\left(\te\right)$'s. Therefore, we obtain the desired result.

          \begin{figure}[H]
    \centering
    \includegraphics[width=0.4\linewidth]{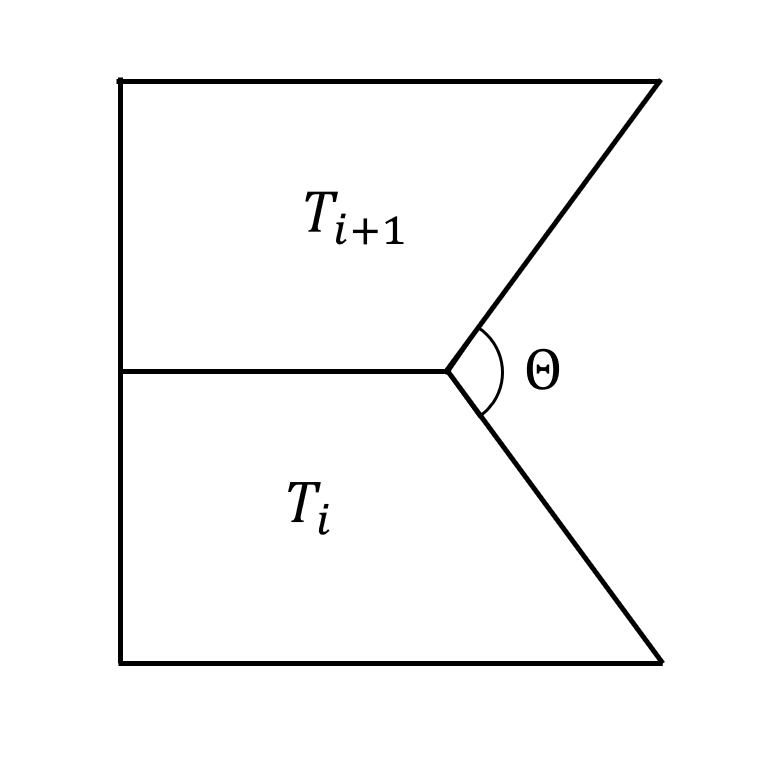}
    \caption{The angle $\Theta$ between $\ml{T_i}$ and $\ml{T_{i+1}}$}
    \label{fig:L1.1.1.2-2}
\end{figure}
    
    \paragraph{Claim 4:} Let $T'\in\mathfrak{T}$ be the tile, not $T_{i-1}$, whose main leg is perfectly covering $\ml{T_{i-1}}$ as assumed by the hypothesis. Then, $\ub{T_{i-1}}$ is perfectly covered by some base of $T_i$ because of $\ub{T'}$ and $\ov{s_1}{s_2}$. Thus, $T_i$ is negative. Moreover, $\ub{T'}$ lies on $\ml{T_i}$ or vice versa. Since $\mathfrak{T}$ is strictly main-to-main, $\ml{T_i}$ is properly covered by edges of some tiles in $\mathfrak{T}$ including $\ub{T'}$. Therefore, the fourth claim is true.

    Now, we prove the problem statement for $\rho\geq 2$. First, suppose that $T_1$ is negative, and assume that $\ml{T_1}$ is properly covered by edges of some tiles in $\mathfrak{T}$. If $\rho=2$, then Claim 1 and Claim 2 give us a contradiction. Thus, we further assume that $\rho\geq 3$. If $\theta(T)\neq\pi/3$, then we can derive a contradiction from Claim 1 and Claim 3. Moreover, if $\theta(T)=\pi/3$, and if $\mathfrak{T}$ is strictly main-to-main, then applying the results of Claim 1 and Claim 4 alternatively (and using the result of Claim 2 at last if $T_{\rho-1}$ is negative) leads to a contradiction. Therefore, the main leg of some tile $T_1'\neq T_1$ in $\mathfrak{T}$ perfectly covers $\ml{T_1}$. Considering the acute angle between $\ml{T_1}$ and $\ov{t_1}{t_4}$, we can observe that $T_1'$ is positive.

    Next, suppose that $T_1$ is positive. Due to $T_2$, the edge $\ml{T_1}$ cannot be properly contained by edges of some tiles in $\mathfrak{T}$. If $T_2$ is positive, then $\ml{T_1}$ is perfectly covered by the main leg of some negative tile in $\mathfrak{T}$. Thus, we may assume that $T_2$ is negative. Similar to $\Theta$, let $\Theta'$ be the angle between $\ml{T_1}$ and $\ml{T_2}$ not filled with the obtuse internal angles of $T_1$ and $T_2$. We assert that neither the obtuse internal angle nor a right internal angle of $\te$ can fill $\Theta'$ possibly along with other internal angles of $\te$. If this is the case, then $\Theta'$ is filled with two $\theta\left(\te\right)$'s, which implies that $\ml{T_1}$ is perfectly covered by the main leg of some negative tile in $\mathfrak{T}$. Consequently, we can conclude that regardless of the sign of $T_1$, the edge $\ml{T_1}$ is perfectly covered by the main leg of another tile in $\mathfrak{T}$ whose bases are parallel to those of $T_1$.
    
    The obtuse angle part follows from the proof of Claim 3 when $\theta(T)\neq\pi/3$. If $\theta(T)=\pi/3$ and $\mathfrak{T}$ is strictly main-to-main, then $\Theta'$ cannot be filled with a single obtuse internal angle of $\te$. Otherwise, either $\ml{T_1}$ or $\ml{T_2}$ is properly covered by edges of some tiles in $\mathfrak{T}$. The former possibility contradicts the fact that edges of some tiles in $\mathfrak{T}$ perfectly cover $\ml{T_1}$ because of $\ov{t_1}{t_4}$ and $\ml{T_2}$. We can also deduce a contradiction from the second possibility by applying to $T_2$ the argument for the case when $T_1$ is negative. Hence, the assertion regarding the obtuse angle is true.
    
    For the right angle part, observe that $\Theta'\neq\pi/2$ and
    $$
    \Theta'=2\cdot\theta(T)<\frac{\pi}{2}+\theta(T),
    $$
    if $\theta(T)\neq\pi/4$. Thus, the assertion regarding the right angle holds when $\theta(T)\neq\pi/4$. If $\theta(T)=\pi/4$, and if $\Theta'$ is filled with a single right internal angle of some $\te$, then there are only two possibilities: a single $\ub{\te}$ perfectly covers $\ml{T_1}$, or one $\ub{\te}$ and one $\ssl{\te}$ perfectly cover $\ml{T_1}$. For the former possibility, there is a tile $T''\in\mathfrak{T}$ such that $\ub{T''}=\ml{T_1}$ and $\ssl{T''}\subset\ml{T_2}$. Since every edge of $T$ is longer than the line segment $\ml{T_2}\backslash\ssl{T''}$, the edge $\ml{T_2}$ is properly covered by edges of some tiles in $\mathfrak{T}$. For the latter possibility, $\ml{T_2}$ is properly covered by edges of some tiles in $\mathfrak{T}$ including at least one of $\lb{\te}$ and $\ml{\te}$. Therefore, for either possibility, edges of some tiles in $\mathfrak{T}$ properly cover $\ml{T_2}$. If $\rho=2$, then this immediately turns out to be a contradiction. We can also derive a contradiction when $\rho\geq 3$ from Claim 1 and Claim 3 (or Claim 1 and Claim 2 if $\rho=3$). Therefore, the assertion regarding the right angle is also true.
    
    Choose an integer $k=1, 2, \cdots, \rho-1$, and suppose that for each $j=1, 2, \cdots, k$, the edge $\ml{T_j}$ is perfectly covered by the main leg of some tile $T_j'\neq T_j$ in $\mathfrak{T}$, where $T_j'$ is either positive or negative. Then, the closed line segment with endpoints $\left(0,kh\right)$ and $\left(a+b,kh\right)$ contains an endpoint of $\ml{T_{k+1}}$ in its interior and is a union of $\ub{T_k}$ and $\lb{T_k'}$ (if $T_k$ is positive) or $\lb{T_k}$ and $\ub{T_k'}$ (if $T_k$ is negative). Hence, we can apply the argument for $T_1$ directly to $T_{k+1}$ and conclude that $\ml{T_{k+1}}$ is perfectly covered by the main leg of another tile in $\mathfrak{T}$ whose bases are parallel to those of $T_{k+1}$. Therefore, by induction, we can obtain the desired result and confirm that the problem statement holds when $\rho\geq 2$. This ends the proof.
\end{proof}

Similar to MTM, we can derive from STS the non-existence of (strictly) main-to-main tilings of $\hat{S}\left(T,\rho,\alpha\right)$ with congruent copies of $T$ when $\alpha\leq a+b$, although we need further assumption for the case when $T$ is obtuse. The following two lemmas describe this.

\begin{lemma}\label{NOS}
    Let $\rho\geq 3$ be a positive integer, and $\alpha$ be a positive real number such that $\alpha<a+b$. Suppose that $T$ is not isosceles and satisfies the sub-to-sub property. If $T$ is not a $\pi/3$-right trapezoid, then there are no main-to-main tilings of the trapezoid $\hat{S}\left(T,\rho,\alpha\right)$ with congruent copies of $T$. Instead, if $T$ is a $\pi/3$-right trapezoid, then there are no strictly main-to-main tilings of $\hat{S}\left(T,\rho,\alpha\right)$ with congruent copies of $T$.
\end{lemma}
\begin{proof}
    Suppose that $T$ is not a $\pi/3$-right trapezoid and assume, for the sake of contradiction, that there is a main-to-main tiling $\mathfrak{T}$ of $\hat{S}\left(T,\rho,\alpha\right)$ with congruent copies of $T$. If $a<\alpha<a+b$, then the parallelogram $t_1t_2t_2't_1'$ is a cluster in $\mathfrak{T}$ by Lemma \ref{L1.1.1.1} and \ref{L1.1.1.2} but it is not contained in $\hat{S}\left(T,\rho,\alpha\right)$, which is a contradiction. Instead, if $\alpha\leq a$, then it follows from Lemma \ref{L1.1.1} that we can select tiles $T_1, T_2, \cdots, T_{\rho}\in\mathfrak{T}$ satisfying the six properties in Lemma \ref{L1.1.1} with $s_0,s_1,s_2,$ and $s_3$ replaced to $t_4,t_1,t_2,$ and $t_3$, respectively. In particular, when $\alpha<a$, the line segment $\ov{t_2}{t_3}$ lies on $\ub{T_{\rho}}$ as a proper subset if $T$ is not obtuse, and $\ov{t_1}{t_4}$ lies on $\ub{T_1}$ as a proper subset if $T$ is obtuse. However, this is a contradiction, for either $T_1$ or $T_{\rho}$ is not contained in $\hat{S}\left(T,\rho,\alpha\right)$. Thus, we further assume $\alpha=a$.

    If $\alpha=a$ and $T$ is not obtuse, then $\mathfrak{T}\backslash\{T_{\rho}\}$ is a tiling of a trapezoid congruent to $\hat{S}\left(T,\rho-1,b\right)$. The previous argument for the case $a<\alpha<a+b$ is also valid when $\rho=2$; thus, we can derive a contradiction. Similarly, if $\alpha=a$ and $T$ is obtuse, then $\mathfrak{T}\backslash\{T_1\}$ is a tiling of a trapezoid congruent to $\hat{S}\left(T,\rho-1,b\right)$, which is also a contradiction for the similar reason. Therefore, we conclude that for all $0<\alpha<a+b$, no tilings of $\hat{S}\left(T,\rho,\alpha\right)$ with congruent copies of $T$ is main-to-main. The analogous statement for $\pi/3$-right trapezoids can be proven similarly.
\end{proof}

\begin{lemma}\label{NOS2}
    Let $\rho\geq 3$ be a positive integer. Suppose that $T$ is not isosceles and satisfies the sub-to-sub property. Then, there are no main-to-main tilings of $\hat{S}\left(T,\rho,a+b\right)$ with congruent copies of $T$ if one of the following is satisfied.
    \begin{enumerate}
        \item $T$ is neither an obtuse trapezoid nor a $\pi/3$-right trapezoid;
        \item $T$ is obtuse, and there are no non-negative integers $p, q,$ and $r$ satisfying $pa+qb+r=nh$.
    \end{enumerate}
    Instead, if $T$ is a $\pi/3$-right trapezoid, then there are no strictly main-to-main tilings of $\hat{S}\left(T,\rho,a+b\right)$ with congruent copies of $T$.
\end{lemma}

\begin{proof}
    Suppose that one of the two conditions in the first statement is satisfied, and assume, for the sake of contradiction, that there is a main-to-main tiling $\mathfrak{T}$ of the trapezoid $\hat{S}\left(T,\rho,a+b\right)$ with congruent copies of $T$. From Lemma \ref{L1.1.1.1} and $\ref{L1.1.1.2}$, it follows that the parallelogram $t_1t_2t_2't_1'$ is a cluster in $\mathfrak{T}$.  Let $\mathfrak{T}_0$ be the collection of all tiles in $\mathfrak{T}$ not contained in the cluster $t_1t_2t_2't_1'$.

    If the first condition is satisfied, then $T$ is not obtuse, and thus $\mathfrak{T}_0$ is a tiling of the triangle $t_1't_2't_4$. By following the argument for $\ml{\te}$ and $\ub{\te}$ in the proof of Lemma \ref{L1.1.1}, we can observe that $\ml{T^{\mathfrak{T}_0}}$ and $\ub{T^{\mathfrak{T}_0}}$ cannot lie on $\ov{t_1'}{t_2'}$ since $T$ satisfies STS, and since the two internal angles $\angle t_1't_2't_4$ and $\angle t_2't_1't_4$ of the triangle $t_1't_2't_4$ have the size smaller than $\pi-\theta(T)$. Moreover, no $\lb{T^{\mathfrak{T}_0}}$ can lie on $\ov{t_1'}{t_2'}$ due to STS. Thus, the line segment $\ov{t_1'}{t_2'}$ is perfectly covered by the subsidiary legs of some tiles in $\mathfrak{T}_0$ whose bases are parallel to those of $\hat{S}\left(T,\rho,a+b\right)$. Then, the internal angle $\angle t_1't_2't_4$ of the triangle $t_1't_2't_4$ is filled with internal angles of some tiles in $\mathfrak{T}$ including at least one internal angle of size $\pi-\psi(T)$. However, this is a contradiction because
    $$
    \angle t_1't_2't_4=\left(\pi-\theta(T)\right)-\psi(T)<\pi-\psi(T).
    $$

    Next, suppose that the second condition is satisfied. Then, $T$ is obtuse, and therefore $\mathfrak{T}_0$ is a tiling of the triangle $t_1't_2't_3$. Moreover, $\ov{t_1'}{t_2'}$ is perfectly covered by the subsidiary legs of some tiles in $\mathfrak{T_0}$ due to the condition regarding $p, q,$ and $r$. This implies that the acute internal angle $\angle t_2't_1't_3$ of the triangle $t_1't_2't_3$ is filled with internal angles of some tiles in $\mathfrak{T}_0$ including at least one $\psi\left(T^{\mathfrak{T}_0}\right)$. However, this is a contradiction since
    $$
    \angle t_2't_1't_3=\left(\pi-\theta(T)\right)-\left(\pi-\psi(T)\right)=\psi(T)-\theta(T)<\psi(T).
    $$
    In conclusion, under either of the two conditions, no tilings of $\hat{S}\left(T,\rho,a+b\right)$ with congruent copies of $T$ is main-to-main. The second statement also follows from the argument for the first condition in the first statement.
\end{proof}

As we combined the results of Lemma \ref{L1.1} and \ref{NOM} to obtain Lemma \ref{ENDM}, we can combine those of Lemma \ref{L1.1.1.1}, \ref{L1.1.1.2}, \ref{NOS}, and \ref{NOS2} in an analogous way and derive the following corollary which is a generalization of the results of Lemma \ref{NOS} and \ref{NOS2} to trapezoids $\hat{S}\left(T,\rho,\alpha\right)$ with arbitrary $\alpha>0$.

\begin{corollary}\label{ENDS}
    Let $\rho\geq 3$ be a positive integer, and $\alpha$ be a positive real number. Suppose that $T$ is not isosceles and satisfies the sub-to-sub property. Then, there are no main-to-main tilings of the trapezoid $\hat{S}\left(T,\rho,\alpha\right)$ with congruent copies of $T$ if one of the following is true.
\begin{enumerate}
    \item $T$ is not obtuse;
    \item $T$ is obtuse and $\alpha$ is not an integer multiple of $a+b$;
    \item $T$ is obtuse, $\alpha$ is an integer multiple of $a+b$, and for each positive integer $n$, there are no non-negative integers $p, q,$ and $r$ satisfying $pa+qb+r=nh$.
\end{enumerate}
    Instead, if $T$ is a $\pi/3$-right trapezoid, then there are no strictly main-to-main tilings of $\hat{S}\left(T,\rho,\alpha\right)$ with congruent copies of $T$.
\end{corollary}

\subsection{Stair-like sub-to-sub property}\label{SSTSP}
As mentioned in Section \ref{MTMP}, we introduce the SSTS property for right trapezoids. Suppose that $T$ is a right trapezoid. Let $\rho$ be a positive integer, and $\alpha$ and $\beta$ be positive real numbers such that $\alpha\leq\beta$. For each $i=1, 2, \cdots,\rho$, let
$$
w_{2i-1}=((i-1)(b-a),(i-1)h),\quad w_{2i}=((i-1)(b-a),ih)
$$
be points on $\mathds{R}^2$. In addition, put
\begin{align*}
    w_{2\rho+1}&=(\alpha+(\rho-1)(b-a),\rho h)\\
    w_0&=(\beta+(\rho-1)(b-a),0)\\
    w_{-1}&=(\beta+(\rho-1)(b-a),\rho h),
\end{align*}
and let $S^2\left(T,\rho,\alpha,\beta\right)$ be a polygon whose boundary contains the polygonal chain
$$
\left[w_{-1},w_0,w_1,\cdots,w_{2\rho+1}\right]
$$
and whose region contains the $(2\rho+3)$-gon $w_{-1}\cdots w_{2\rho+1}$ (or the $(2\rho+2)$-gon $w_0\cdots w_{2\rho+1}$ if $\alpha=\beta$); see the following figure, for instance.
\begin{figure}[H]
    \centering
    \includegraphics[width=0.4\linewidth]{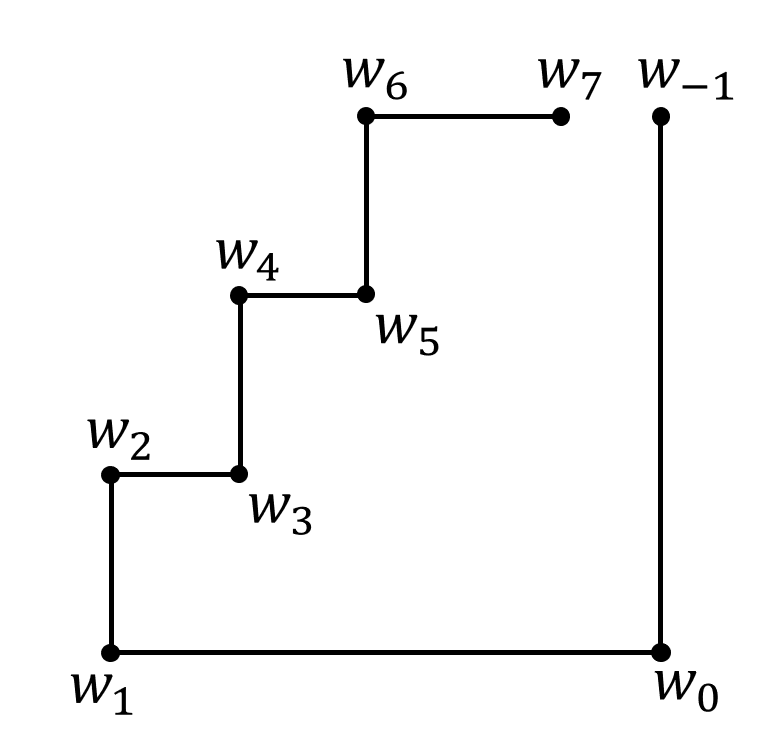}
    \caption{Part of $S^2\left(T,3,\alpha,\beta\right)$}
    \label{fig:STS}
\end{figure}
\noindent
We will say that $T$ satisfies the \textbf{stair-like sub-to-sub property} (abbreviated to \textbf{SSTS}) if it satisfies the following condition: for every positive integer $\rho$, positive real numbers $\alpha\leq\beta$, and polygon $S^2\left(T,\rho,\alpha,\beta\right)$ corresponding to these three numbers, if there is a tiling $\mathfrak{T}$ of $S^2\left(T,\rho,\alpha,\beta\right)$ with congruent copies of $T$, then for each $i=1, 2, \cdots,\rho$, the line segment $\ov{w_{2i-1}}{w_{2i}}$ is perfectly covered by a single $\ssl{\te}$. We denote the unique $(2\rho+2)$-gon $S^2\left(T,\rho,\alpha,\beta\right)$ corresponding to the three numbers $\rho$ and $\alpha=\beta$ by $\hat{S}^2(T,\rho,\alpha)$.

One can immediately notice that SSTS is a re-expression of the property that we have sought: only subsidiary legs of tiles in $\mathfrak{T}$ can perfectly cover each $\ssl{T_i}$ in Lemma \ref{L1}, provided that $T$ is a right trapezoid. The following lemma provides sufficient conditions for $T$ satisfying SSTS.

\begin{lemma}\label{LESTS}
    Suppose that $T$ is a right trapezoid. Then, it has the stair-like sub-to-sub property if one of the following conditions is satisfied.
    \begin{enumerate}
        \item There are no non-negative integers $p$ and $q$ such that $p\leq q+1$ and $pa+qb=h$;
        \item $a=h$ and $\theta(T)\neq\pi/4$.
    \end{enumerate}
\end{lemma}
\begin{proof}
As explained earlier, choose a positive integer $\rho$, two real numbers $\alpha,\beta$ such that $\alpha\leq\beta$, and a polygon $S^2\left(T,\rho,\alpha,\beta\right)$ corresponding to these three numbers. Recall also the points $w_{-1},\cdots,w_{2\rho+1}$ on $\mathds{R}^2$. We assume that there is a tiling $\mathfrak{T}$ of $S^2\left(T,\rho,\alpha,\beta\right)$ with congruent copies of $T$ such that for some $i=1, 2, \cdots,\rho$, the line segment $\ov{w_{2i-1}}{w_{2i}}$ is not perfectly covered by a single $\ssl{\te}$ and then derive a contradiction.

\paragraph{Condition 1:} Suppose that the first condition in the problem statement holds. If $\rho=1$, then the line segment $\ov{w_1}{w_2}$ is perfectly covered by edges of some tiles in $\mathfrak{T}$ other than $\ssl{\te}$'s and $\ml{\te}$'s since $\left|\ssl{\te}\right|=h$ and $\left|\ml{\te}\right|>h$. Let $\mathcal{U}$ and $\mathcal{L}$ be the collection of tiles $\te$ such that $\ub{\te}\subset\ov{w_1}{w_2}$ and $\lb{\te}\subset\ov{w_1}{w_2}$, respectively. We claim that $|\mathcal{U}|\leq|\mathcal{L}|$. For each tile in $\mathcal{U}$, the acute angle between its main leg and $\ov{w_1}{w_2}$ is filled with a single main acute angle of a tile in $\mathcal{L}$. Moreover, the main acute angle of a tile in $\mathcal{L}$ cannot fill more than one such acute angle generated by tiles in $\mathcal{U}$. Such a correspondence gives us an injective map from $\mathcal{U}$ to $\mathcal{L}$. This proves our claim. However, the result of the claim contradicts the assumed first condition because
$$
|\mathcal{U}|a+|\mathcal{L}|b=|\mathcal{U}|\cdot\left|\ub{T}\right|+|\mathcal{L}|\cdot\left|\lb{T}\right|=\left|\ov{w_1}{w_2}\right|=h.
$$

Next, assume that $\rho\geq 2$ and for some $i=2, 3, \cdots,\rho$, the line segment $\ov{w_{2i-1}}{w_{2i}}$ is not perfectly covered by a single $\ssl{\te}$. We may further assume that $\ov{w_1}{w_2}$ is perfectly covered by a single $\ssl{\te}$ because of the previous discussion for $\rho=1$. We then assert that $\ov{w_{2i-1}}{w_{2i}}$ cannot be perfectly covered by edges of tiles in $\mathfrak{T}$. Assume, for the sake of contradiction, that some edges of tiles in $\mathfrak{T}$ perfectly cover $\ov{w_{2i-1}}{w_{2i}}$. Since $\left|\ssl{\te}\right|=h$ and $\left|\ml{\te}\right|>1$, this must be done without $\ssl{\te}$'s and $\ml{\te}$'s lying on $\ov{w_{2i-1}}{w_{2i}}$. Let $\mathcal{U'}$ and $\mathcal{L'}$ be the set of tiles $\te$ such that $\ub{\te}\subset\ov{w_{2i-1}}{w_{2i}}\backslash\{w_{2i-1}\}$ and $\lb{\te}\subset\ov{w_{2i-1}}{w_{2i}}$, respectively. Then, there is an injective map from $\mathcal{U'}$ to $\mathcal{L'}$ analogous to that from $\mathcal{U}$ to $\mathcal{L}$ defined for the case $\rho=1$. Thus, it follows that $\left|\mathcal{U'}\right|\leq\left|\mathcal{L'}\right|$ and that $\left|\mathcal{U'}\right|a+\left|\mathcal{L'}\right|b$ equals either $h$ or $h-a$. This contradicts the assumed first condition; therefore, we can confirm our assertion is true.

We now take a look at the line segment $\ov{w_{2i-2}}{w_{2i-1}}$ of length $b-a$. Note that $b-a<\left|\lb{T}\right|$ and $b-a=\cos{\theta(T)}<\left|\ml{T}\right|$. Moreover, by the assertion, there is a closed line segment $L$ which is a union of edges of some tiles in $\mathfrak{T}$ properly covering $\ov{w_{2i-1}}{w_{2i}}$, and $\ov{w_{2i-2}}{w_{2i-1}}$ is normal to $L$ and $\ov{w_{2i-3}}{w_{2i-2}}$ (see Figure \ref{fig:SSTSCON1}). These facts imply that $\ov{w_{2i-2}}{w_{2i-1}}$ is perfectly covered by some $\ssl{\te}$'s; otherwise if $\ub{T'}\subset\ov{w_{2i-2}}{w_{2i-1}}$ for some tile $T'\in\mathfrak{T}$, then the acute angle between $\ub{T'}$ and $\ov{w_{2i-2}}{w_{2i-1}}$ is filled with a single $\theta\left(\te\right)$, so either $\ml{\te}$ or $\lb{\te}$ also lies on $\ov{w_{2i-2}}{w_{2i-1}}$. Thus, at least one base of $\te$ lies on $\ov{w_{2i-3}}{w_{2i-2}}$ or vice versa. For this reason, $\ov{w_{2i-3}}{w_{2i-2}}$ is not perfectly covered by a single $\ssl{\te}$. Inductively, we can deduce that for each $j=1, 2, \cdots, i$, the line segment $\ov{w_{2j-1}}{w_{2j}}$ is not perfectly covered by a single $\ssl{\te}$. However, this contradicts the assumption that $\ov{w_1}{w_2}$ is perfectly covered by a single $\ssl{\te}$. This ends the proof for the first condition.

\begin{figure}[H]
    \centering
    \includegraphics[width=0.5\linewidth]{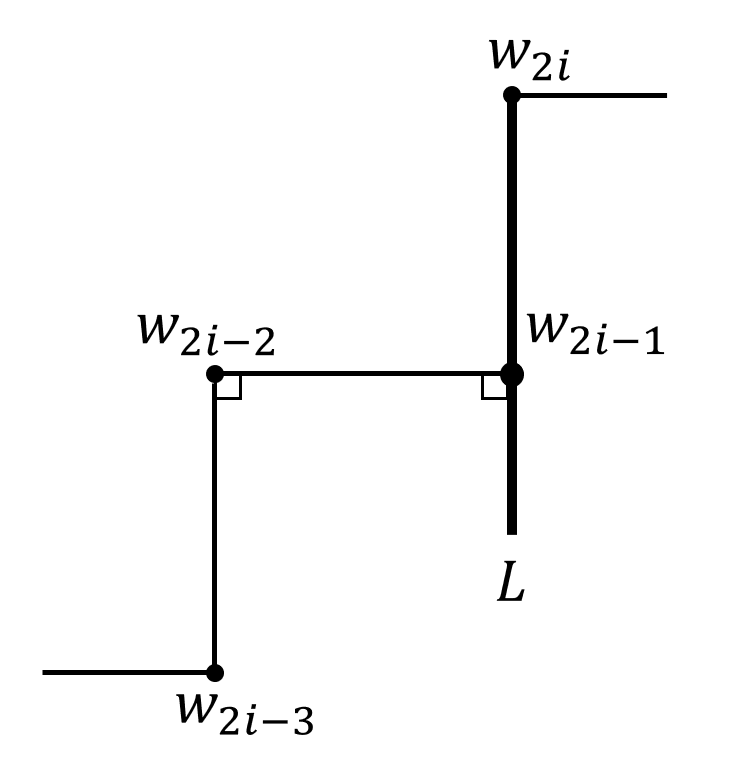}
    \caption{$\ov{w_{2i-2}}{w_{2i-1}}$ normal to $L$ and $\ov{w_{2i-3}}{w_{2i-2}}$}
    \label{fig:SSTSCON1}
\end{figure}

\paragraph{Condition 2:} Suppose that the second condition in the problem statement holds. If $\rho=1$, then $\ml{\te}$ and $\lb{\te}$ cannot lie on the line segment $\ov{w_1}{w_2}$ since $\left|\ml{\te}\right|>h$ and $\left|\lb{\te}\right|>h$. The same goes for $\ub{\te}$ because of the obtuse internal angle of $T$ between $\ub{T}$ and $\ml{T}$. Hence, we obtain the desired result for $\rho=1$.

For the case $\rho\geq 2$, if edges of some tiles in $\mathfrak{T}$ properly covers $\ov{w_{2i-1}}{w_{2i}}$, then we can apply the previous argument for $\ov{w_{2i-2}}{w_{2i-1}}$ under the first condition and then observe that $\ov{w_{2i-3}}{w_{2i-2}}$ is not perfectly covered by a single $\ssl{\te}$. Thus, it suffices to show the following: if the line segment $\ov{w_{2i-1}}{w_{2i}}$ is perfectly covered by a single $\ub{\te}$ for some $i=2, 3,\cdots,\rho$, then $\ov{w_{2i-3}}{w_{2i-2}}$ is not perfectly covered by a single $\ssl{\te}$. Suppose that there is a tile $T''\in\mathfrak{T}$ whose upper base perfectly covers $\ov{w_{2i-1}}{w_{2i}}$ (see Figure \ref{fig:SSTSCON2}). Put $L'=\ov{w_{2i-2}}{w_{2i-1}}$, and let $T'''\in\mathfrak{T}$ be the tile such that one of its edges lies on $L'$ and contains the point $w_{2i-2}$. Due to $\ml{T''}$ and $\ov{w_{2i-3}}{w_{2i-2}}$, the line segment $L$ is perfectly covered by edges of some tiles in $\mathfrak{T}$. Furthermore, neither $\ml{\te}$ nor $\lb{\te}$ lies on $L'$ because of their length. Thus, either $\ssl{T'''}$ or $\ub{T'''}$ lies on $L'$. If $\ssl{T'''}\subset L'$, then it is done since one of the two bases of $T'''$ covers $\ov{w_{2i-3}}{w_{2i-2}}$. On the other hand, if $\ub{T'''}\subset L'$, then we have $\ub{T'''}=L'$, for otherwise at least one of $\ml{\te}$ and $\lb{\te}$ lies on $L'$ while a single $\theta\left(\te\right)$ is filling the acute angle between $\ml{T'''}$ and $L'$. Moreover, the internal angle $\angle w_{2i-2}w_{2i-1}w_{2i}$ of $S^2\left(T,\rho,\alpha,\beta\right)$ of size $3\pi/2$ is filled with some internal angles of $\te$ including two obtuse internal angles of $T''$ and $T'''$. However, this is impossible when $\theta(T)<\pi/4$ because
$$
\frac{3\pi}{2}-\left(\pi-\theta\left(T''\right)\right)-\left(\pi-\theta\left(T'''\right)\right)=2\theta(T)-\frac{\pi}{2}<0.
$$
Similarly, when $\theta(T)>\pi/4$, no internal angles of $\te$ can fill the acute angle between $\ml{T''}$ and $\ml{T'''}$ since its size is $2\theta(T)-\pi/2$, smaller than $\theta(T)$. In either possibility, we obtain a contradiction, and thus $\ub{T'''}$ cannot lie on $L'$ in the first place. This ends the proof for the second condition.

\begin{figure}[H]
    \centering
    \includegraphics[width=0.5\linewidth]{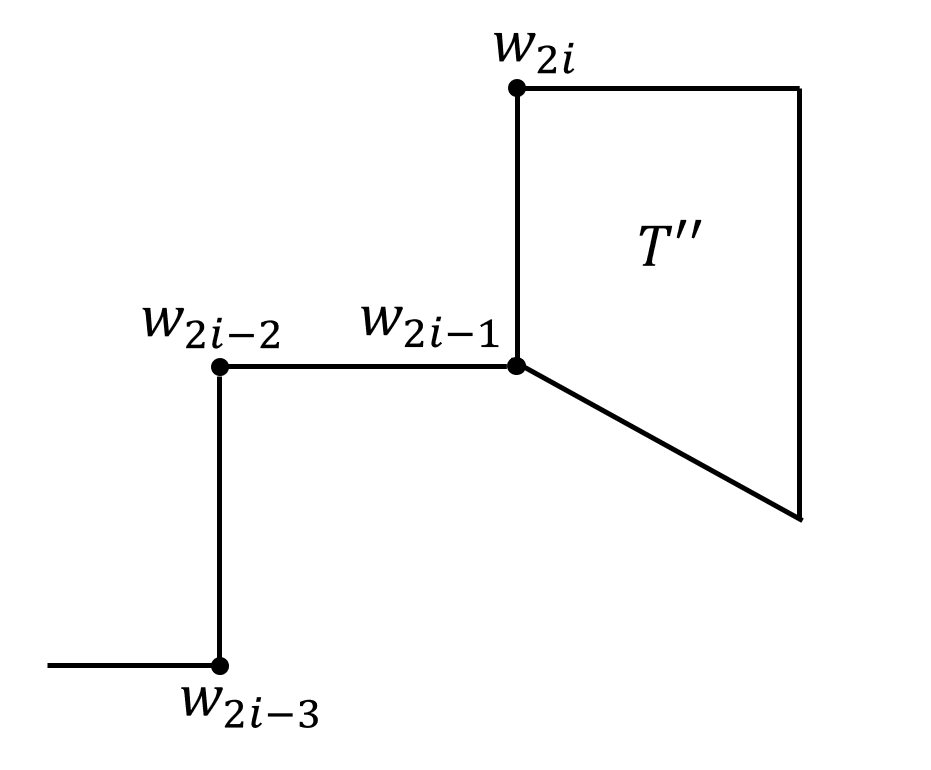}
    \caption{$\ub{T''}$ perfectly covering $\ov{w_{2i-1}}{w_{2i}}$}
    \label{fig:SSTSCON2}
\end{figure}
\end{proof}

Unlike MTM and STS, utilizing SSTS for perfectly covering each $\ssl{T_i}$ in Lemma \ref{L1} does not give us a cluster parallelogram as in Lemma \ref{L1.1} or Lemma \ref{L1.1.1.1}. However, by combining MTM and SSTS, we can obtain a stair-like cluster polygon, as shown in the following lemma.

\begin{figure}[H]
    \centering
    \includegraphics[width=0.6\linewidth]{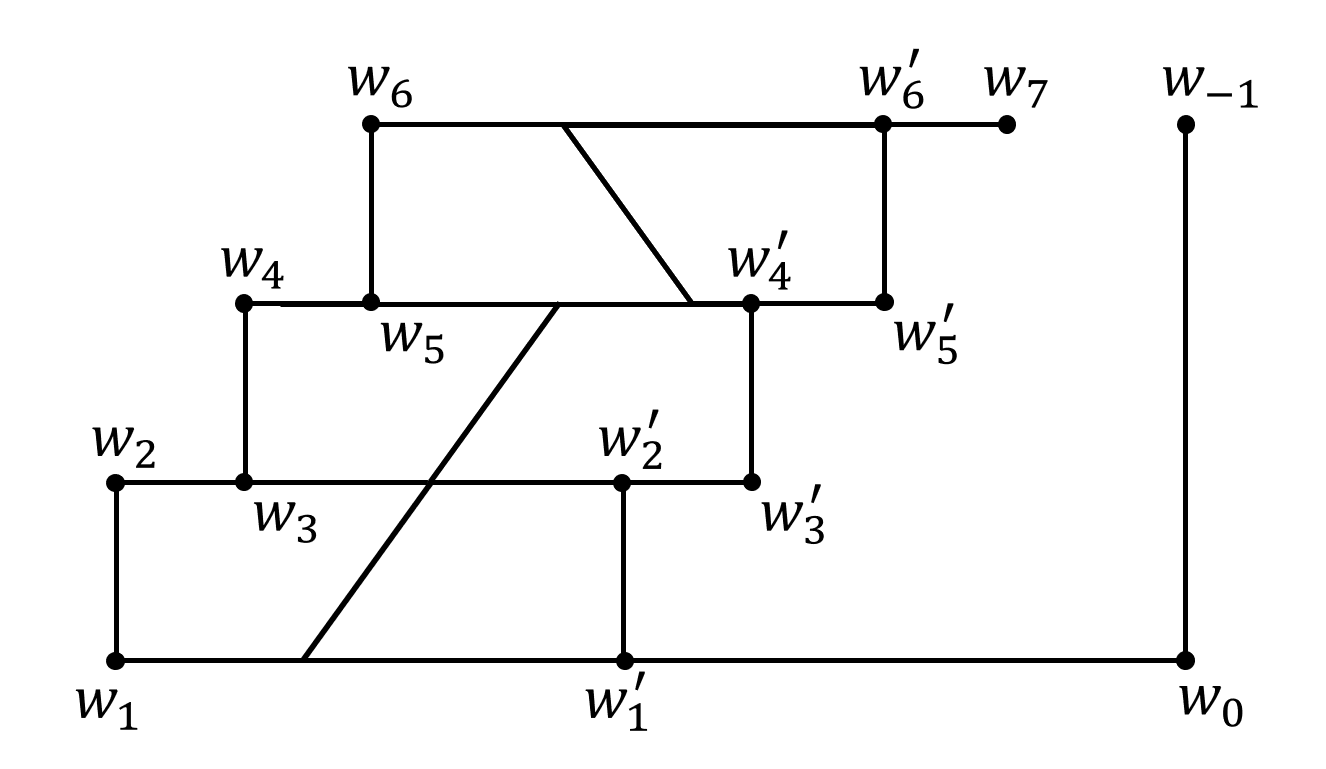}
    \caption{Cluster $w_1w_2\cdots w_6w_6'w_5'\cdots w_1'$ when $\rho=3$}
    \label{fig:L2}
\end{figure}

\begin{lemma}\label{L2}
    Suppose that $T$ is a right trapezoid satisfying both main-to-main and stair-like sub-to-sub properties. Let $\rho$ be a positive integer, and $\alpha$ and $\beta$ be positive real numbers such that $\beta\geq\alpha>b$; let $S^2\left(T,\rho,\alpha,\beta\right)$ be a polygon corresponding to these three numbers and the $2\rho+3$ points $w_{-1},w_0,\cdots,w_{2\rho+1}$ (or $2\rho+2$ points if $\alpha=\beta$) in the way explained earlier. If there is a tiling $\mathfrak{T}$ of $S^2\left(T,\rho,\alpha,\beta\right)$ with congruent copies of $T$, then the $4\rho$-gon
    $$
    w_1w_2\cdots w_{2\rho}w_{2\rho}'w_{2\rho-1}'\cdots w_1'
    $$
    is a cluster in $\mathfrak{T}$, where $w_i':=w_i+(a+b,0)$ for each $i=1, 2,\cdots, 2\rho$ (see Figure \ref{fig:L2}, for instance).
\end{lemma}

\noindent
The proof of the above lemma is long, so we postpone providing it until we bring and prove the next three results. Similar to Lemma \ref{NOM} and \ref{NOS}, MTM and SSTS combined imply the non-existence of tilings of $\hat{S}^2\left(T,\rho,\alpha\right)$ with congruent copies of $T$ when $\alpha<a+b$. This is described in the following lemma.

\begin{lemma}\label{L3}
    Let $\rho\geq 3$ be a positive integer, and $\alpha$ be a positive real number such that $\alpha\leq a+b$. If $T$ is a right trapezoid satisfying both main-to-main and stair-like sub-to-sub properties, then $\hat{S}^2\left(T,\rho,\alpha\right)$ cannot be tiled with congruent copies of $T$.
\end{lemma}
\begin{proof}
    Recall the points $w_{0},w_1,\cdots,w_{2\rho+1}$ in $\mathds{R}^2$ which are all of the vertices of $\hat{S}^2\left(T,\rho,\alpha\right)$. Assume, for the sake of contradiction, that there is a tiling $\mathfrak{T}$ of  $\hat{S}^2\left(T,\rho,\alpha\right)$ with congruent copies of $T$. If $b<\alpha< a+b$, then by Lemma \ref{L2}, the $4\rho$-gon
    $$
    w_1w_2\cdots w_{2\rho}w_{2\rho}'w_{2\rho-1}'\cdots w_1'
    $$
    is a cluster in $\mathfrak{T}$, where $w_i'=w_i+(a+b,0)$ for each $i=1, 2, \cdots, 2\rho$. However, since $\alpha<a+b$, the two points $w_{2\rho}'$ and $w_{2\rho-1}'$ are not contained in $\hat{S}^2\left(T,\rho,\alpha\right)$, which is a contradiction.
    
    Next, suppose that $\alpha\leq b$. Since $T$ satisfies SSTS, there is a tile $T_1\in\mathfrak{T}$ such that $\ssl{T_1}=\ov{w_{2\rho-1}}{w_{2\rho}}$. From this, it immediately follows that the distance between two parallel line segments $\ov{w_{2\rho-1}}{w_{2\rho}}$ and $\ov{w_{0}}{w_{2\rho+1}}$ is at least $b$, which is possible only if $\alpha\geq b$. Thus, we may exclude the case $\alpha<b$ and assume $\alpha=b$. Moreover, we may further assume that $T_1$ is negative, for otherwise, we obtain a right-triangular area inside $\hat{S}^2\left(T,\rho,b\right)$ that cannot be tiled by tiles in $\mathfrak{T}$.

    \begin{figure}[H]
    \centering
    \includegraphics[width=0.4\linewidth]{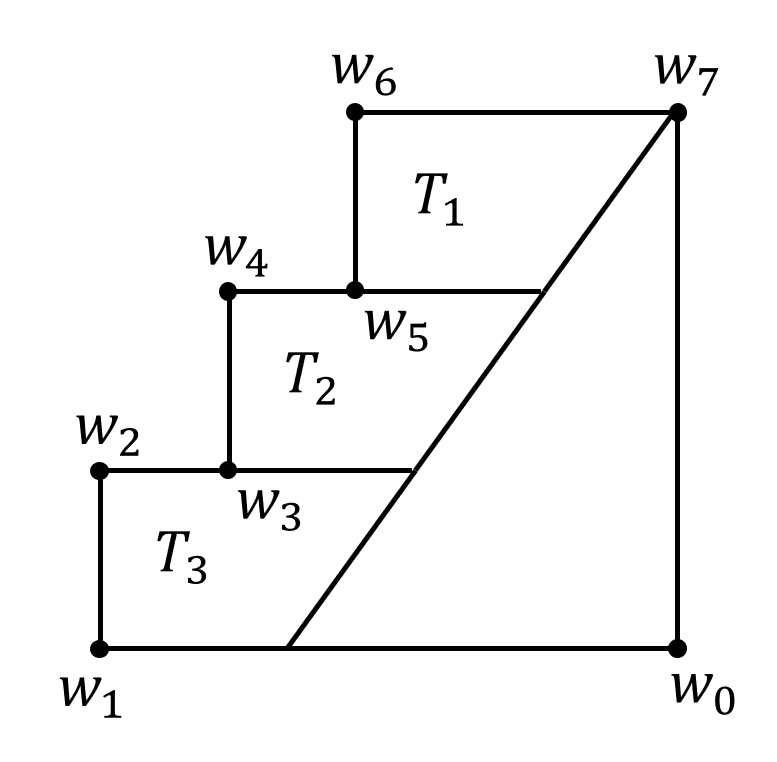}
    \caption{$\hat{S}^2\left(T,3,b\right)$ and three negative tiles $T_1,T_2,$ and $T_3$}
    \label{fig:L3-1}
\end{figure}

    By SSTS, for each $i=2, 3, \cdots, \rho$, we can select a tile $T_i\in\mathfrak{T}$ such that $\ssl{T_i}=\ov{w_{2\rho-2i+1}}{w_{2\rho-2i+2}}$. If $T_i$ is negative for all $i=2, 3, \cdots, \rho$, then $\cup_{i=1}^{\rho}\ml{T_i}$ is a closed line segment (see Figure \ref{fig:L3-1}, for instance). Moreover, the interior of $\ov{w_0}{w_1}$ contains an endpoint of $\cup_{i=1}^{\rho}\ml{T_i}$, and the acute angle between $\ov{w_0}{w_1}$ and $\cup_{i=1}^{\rho}\ml{T_i}$ has the same size as $\theta(T)$. Thus, by Lemma \ref{L1}, each $\ml{T_i}$ is perfectly covered by the main leg of some positive tile in $\mathfrak{T}$. However, this is a contradiction because the positive tile whose main leg perfectly covers $\ml{T_1}$ is not contained in $\hat{S}^2\left(T,\rho,b\right)$.

\begin{figure}[H]
    \centering
    \includegraphics[width=0.4\linewidth]{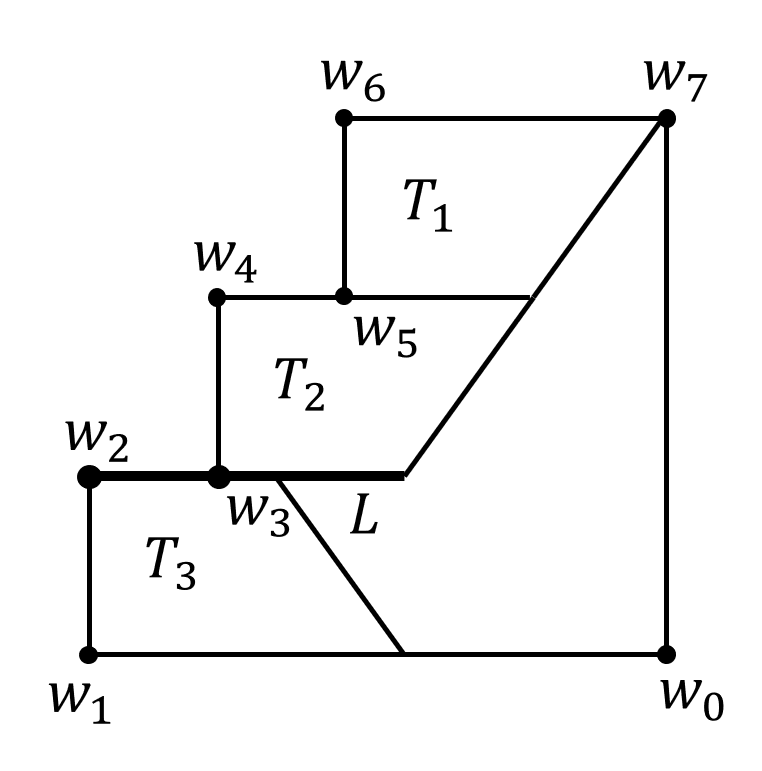}
    \caption{$\hat{S}^2\left(T,3,b\right)$ and three tiles $T_1,T_2,$ and $T_3$ when $j=3$}
    \label{fig:L3-2}
\end{figure}
    
    Instead, suppose that for some $j=2, 3, \cdots, \rho$, the tile $T_j$ is positive, and $T_1, \cdots, T_{j-1}$ are negative. Put $L=\ov{w_{2\rho-2j+2}}{w_{2\rho-2j+3}}\cup\ub{T_{j-1}}$ (see Figure \ref{fig:L3-2}, for instance). Then, $L$ is a closed line segment of length $b$, and thus $\ub{T_j}\subset L$ and the interior of $L$ contains an endpoint of $\ml{T_j}$. Here, the acute angle between $\ml{T_j}$ and $L$ is filled with a single $\theta\left(\te\right)$, which implies that either $\ml{\te}$ or $\lb{\te}$ lies on $L$. Since $\left|L\backslash\ub{T_j}\right|=b-a<b$, and since $b-a=\cos\theta(T)<1$, it then follows that edges of some tiles in $\mathfrak{T}$ properly cover $L$. Therefore, we can apply Lemma \ref{L1} to the line segment $\cup_{i=1}^{j-1}\ml{T_i}$ and derive a contradiction as we did for the case when $T_i$ is negative for all $i=1, 2, \cdots,\rho$. In conclusion, $\hat{S}^2\left(T,\rho,b\right)$ cannot be tiled with congruent copies of $T$.

    Finally, if $\alpha=a+b$, then similar to the case $b<\alpha<a+b$, the aforementioned $4\rho$-gon is a cluster in $\mathfrak{T}$; let $\mathfrak{T}_0$ be the collection of tiles in $\mathfrak{T}$ not contained in that cluster. Then, $\mathfrak{T}_0$ is a tiling of remaining region of $\hat{S}^2\left(T,\rho,a+b\right)$ not covered by the cluster $4\rho$-gon, which is congruent to $\hat{S}^2\left(T,\rho-1,b-a\right)$. We can derive a contradiction by applying the preceding argument of excluding the case $\alpha<b$ (which is also valid for $\rho=2$). This ends the proof.
\end{proof}

We can conduct an inductive process analogous to that of Corollary \ref{ENDM} and \ref{ENDS} by combining the results of Lemma \ref{L2} and \ref{L3}, which gives us the following result.

\begin{corollary}\label{ENDSS}
    Let $\rho\geq 3$ be a positive integer, and $\alpha$ be a positive real number. If $T$ is a right trapezoid satisfying both main-to-main and stair-like sub-to-sub properties, then $\hat{S}^2\left(T,\rho,\alpha\right)$ cannot be tiled with congruent copies of $T$.
\end{corollary}

Finally, we can obtain the following result for $\hat{M}(T,\rho,\alpha)$ from Lemma \ref{L1} and the above corollary, which resembles Corollary \ref{ENDM}.

\begin{corollary}\label{ENDSSS}
    Choose a positive integer $\rho\geq 3$ and a positive real number $\alpha$. Suppose that $T$ is a right trapezoid that satisfies the main-to-main and stair-like sub-to-sub properties. Then, there are no tilings $\mathfrak{T}$ of the trapezoid $\hat{M}\left(T,\rho,\alpha\right)$ with congruent copies of $T$.
\end{corollary}
\begin{proof}
     Assume, for the sake of contradiction, that there is a tiling $\mathfrak{T}$ of the trapezoid $\hat{M}\left(T,\rho,\alpha\right)$. By Lemma \ref{L1}, we can select tiles $T_1,T_2,\cdots,T_{\rho}\in\mathfrak{T}$ satisfying the three properties in Lemma \ref{L1} with $u_0,u_1,$ and $u_2$ replaced to $t_4,t_1,$ and $t_2$, respectively. Then, $\ub{T_{\rho}}$ lies on $\ov{t_2}{t_3}$. This is a contradiction when $\alpha<a$ since $\left|\ov{t_2}{t_3}\right|=\alpha$. Thus, we may assume $\alpha\geq a$. Put $\mathfrak{T}_0=\mathfrak{T}\backslash\{T_1,T_2,\cdots,T_{\rho}\}$. If $\alpha>a$, then $\mathfrak{T}_0$ is a tiling, with congruent copies of $T$, of a $(2\rho+2)$-gon congruent to $\hat{M}\left(T,\rho,\alpha-a\right)$, and thus we derive a contradiction from Corollary \ref{ENDSS}. Instead, if $\alpha=a$, then $\mathfrak{T}_0$ is a tiling, with congruent copies of $T$, of a $2\rho$-gon congruent to $\hat{M}\left(T,\rho-1,b-a\right)$. By Corollary \ref{ENDSS}, this is a contradiction if $\rho\geq 4$. We can obtain the same result for the case $\rho=3$ considering the argument for $\hat{S}^2\left(T,\rho-1,b-a\right)$ in the last paragraph of the proof of Lemma \ref{L3}. Therefore, $\hat{M}\left(T,\rho,\alpha\right)$ cannot be tiled with congruent copies of $T$ in the first place.
\end{proof}

We end this section by proving Lemma \ref{L2}.

\paragraph{Proof of Lemma \ref{L2}}

    We use an induction over $\rho$. First suppose that $\rho=1$ and that there is a tiling $\mathfrak{T}_0$ of $S^2\left(T,1,\alpha,\beta\right)$ with congruent copies of $T$. By SSTS, there is a tile $T'\in\mathfrak{T}_0$ such that $\ssl{T'}=\ov{w_1}{w_2}$. If $T'$ is positive, then two parallel line segments $\ov{w_0}{w_1}$ and $\ov{w_2}{w_3}$ both contain an endpoint of $\ml{T'}$ in their interior. Moreover, the acute angle between $\ml{T'}$ and $\ov{w_2}{w_3}$ has the same size as $\theta(T)$, and thus, by Lemma \ref{L1}, the main leg of some negative tile in $\mathfrak{T}_0$ perfectly covers $\ml{T'}$. Hence, the rectangle $w_1w_2w_2'w_1'$ is a cluster in $\mathfrak{T_0}$. This argument can also be applied to the case when $T'$ is negative since $\left|\ov{w_2}{w_3}\right|>\left|\lb{T}\right|$. Therefore, the problem statement holds when $\rho=1$.

    Choose a positive integer $m\geq 2$ to prove the inductive step, and suppose that the problem statement holds for $\rho=1, 2, \cdots, m-1$. We show that the statement also holds for $\rho=m$. Suppose that there is a tiling $\mathfrak{T}$ of $S^2\left(T,m,\alpha,\beta\right)$ with congruent copies of $T$. From SSTS, it follows that $\ov{w_{2m-1}}{w_{2m}}$ is perfectly covered by the subsidiary leg of some tile $T_1\in\mathfrak{T}$. Note that
    $$
    \left|\ov{w_{2m-2}}{w_{2m-1}}\right|+\left|\ov{w_{2m-1}}{w_{2m-1}'}\right|=(b-a)+(a+b)=2b>b.
    $$
    Thus, by the induction hypothesis, it suffices to show that $\ml{T_1}$ is perfectly covered by the main leg of another tile in $\mathfrak{T}$ whose bases are parallel to those of $T_1$.

    First, suppose that $T_1$ is negative. For each $i=2, 3, \cdots, m$, the line segment $\ov{w_{2m-2i+1}}{w_{2m-2i+2}}$ is perfectly covered by the subsidiary leg of some tile $T^i\in\mathfrak{T}$ since $T$ satisfies SSTS. If $T^i$ is negative for all $i=2, 3, \cdots, m$, then Lemma \ref{L1} implies that each $\ml{T^i}$ is perfectly covered by the main leg of some positive tile in $\mathfrak{T}$, as we desired. Instead, suppose that for some $j=2, 3, \cdots, m$, the tile $T^j$ is positive and $T^2,T^3,\cdots,T^{j-1}$ are negative. Then, we can observe that the closed line segment $\ov{w_{2m-2j+2}}{w_{2m-2j+3}}\cup\ub{T_{j-1}}$ is properly covered by edges of some tiles in $\mathfrak{T}$ by directly following the argument for the line segment $L$ in the proof of Lemma \ref{L3}. Thus, Lemma \ref{L1} implies that for each $i=1, 2, \cdots, j-1$, the edge $\ml{T^i}$ is perfectly covered by the main leg of some positive tile in $\mathfrak{T}$. Therefore, the main leg of some positive tile in $\mathfrak{T}$ perfectly covers $\ml{T_1}$ if $T_1$ is negative.

    Next, suppose that $T_1$ is positive. Then, the acute angle between the line segment $\ov{w_{2m}}{w_{2m+1}}$ and $\ml{T_1}$ has the same size as $\theta(T)$, and thus it is filled with a single $\theta\left(\te\right)$; let $T_1'\in\mathfrak{T}$ be the tile whose main acute angle is filling that acute angle. If $\ml{T_1'}$ is lying on $\ml{T_1}$, then it is done. Hence, we instead assume that $\lb{T_1'}$ is lying on $\ml{T_1}$ or vice versa and derive a contradiction. Note that $\ml{T_1}\neq\lb{T_1'}$ because of MTM. Put $E_1=\lb{T_1'}$. Since $T$ satisfies MTM, provided that $\ov{w_0}{w_1}$ is long enough and $\lb{T_1'}\varsubsetneq\ml{T_1}$, we can select tiles $T_2',T_3',\cdots,T_{n_1}'\in\mathfrak{T}$ ($n_1>1$) and corresponding edges $E_2,E_3,\cdots,E_{n_1}$ satisfying the following properties (see Figure \ref{fig:L2-1}, for instance).
    \begin{enumerate}
        \item $E_i$ is an edge of $T_i'$ for each $i=1,2,\cdots,n_1$;
        \item $E_1,E_2,\cdots,E_{n_1}$ properly covers $\ml{T_1}$ in the order $\left(E_1,E_2,\cdots,E_{n_1}\right);$
        \item $E_{n_1}\cap\ml{T_1}$ contains more than one point, and $E_{n_1}\not\subset\ml{T_1}$.
    \end{enumerate}
    If $\ml{T_1}\subset\lb{T_1'}$, then we put $n_1=1$. When $\ov{w_0}{w_1}$ is too short, it may be possible that we cannot select such tiles and their edges satisfying the above three conditions and at the same time not going over the boundary of $S^2\left(T,m,\alpha,\beta\right)$. If this is the case, then it follows from MTM that $\ml{T_1}$ cannot be covered perfectly or properly by edges of tiles in $\mathfrak{T}$ other than $T_1$. However, this is a contradiction since $\ml{T_1}$ is not lying on the boundary of $S^2\left(T,m,\alpha,\beta\right)$. For this reason, we may assume in further proof that $\ov{w_0}{w_1}$ is long enough to forestall such a problem.

    \begin{figure}[H]
    \centering
    \includegraphics[width=0.6\linewidth]{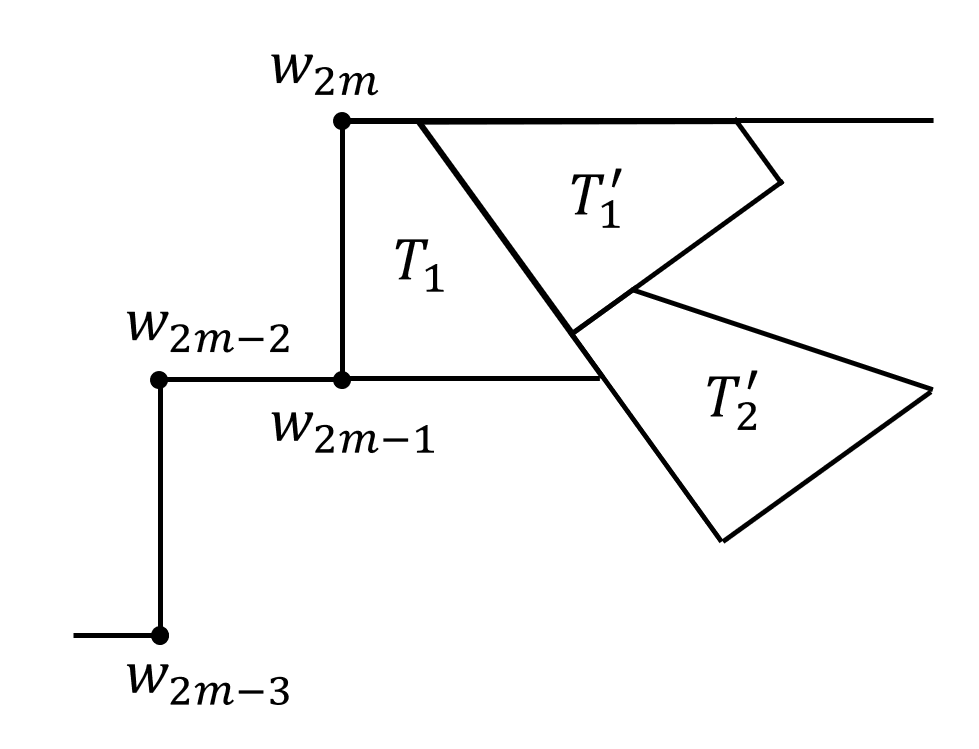}
    \caption{Tiles $T_1,T_1',T_2',\cdots,T_{n_1}'$ with $n_1=2$}
    \label{fig:L2-1}
\end{figure}

    Put $L_1=\lb{T_1}\cup\ov{w_{2m-2}}{w_{2m-1}}$. Since $L_1$ is a line segment of length $2b-a$, and since $2b-a>b$, it follows from the induction hypothesis for $\rho=m-1$ that the $(4m-4)$-gon 
    $$
    w_1w_2\cdots w_{2m-2}w_{2m-2}'w_{2m-3}'\cdots w_1'
    $$
    is a cluster in $\mathfrak{T}$. Here, we derive a contradiction in advance for all but the case $a=b/3$.

    \paragraph{Case $a>b/2$:} If $a>b/2$, then $\left|L_1\right|<a+b$. Thus, $\ov{w_{2m-2}}{w_{2m-2}'}$ properly covers $L_1$, which implies that $\ml{T_1}$ cannot be properly covered by edges of tiles in $\mathfrak{T}$. However, this is a contradiction since $E_1,E_2,\cdots,E_{n_1}$ properly cover $\ml{T_1}$.

    \paragraph{Case $a=b/2$:} If $a=b/2$, then $L_1$ is perfectly covered by $\ov{w_{2m-2}}{w_{2m-2}'}$. Moreover, since $T$ satisfies SSTS, and since
    $$
    \left|\ov{w_{2m-4}'}{w_{2m-3}'}\right|=b-a=a<b,
    $$
    there is a tile $T''\in\mathfrak{T}$ such that $\ssl{T''}=\ov{w_{2m-5}'}{w_{2m-4}'}$ and either $\ub{T''}$ or $\lb{T''}$ covers $\ov{w_{2m-4}'}{w_{2m-3}'}$ (see Figure \ref{fig:L2-2}). If $\ub{T''}=\ov{w_{2m-4}'}{w_{2m-3}'}$, then the line segment $\ov{w_{2m-3}'}{w_{2m-2}'}$ cannot be properly covered by edges of tiles in $\mathfrak{T}$ because of $\ml{T''}$ and $\cup_{i=1}^{n_1}E_i$. The same result holds when $\lb{T''}$ properly covers $\ov{w_{2m-4}'}{w_{2m-3}'}$. Thus, in either case, the edge $\lb{\te}$ lies on $\ov{w_{2m-3}'}{w_{2m-2}'}$ since the acute angle $\Theta$ between the two line segments $\ov{w_{2m-3}'}{w_{2m-2}'}$ and $\cup_{i=1}^{n_1}E_i$ is filled with one or more $\theta(\te)$'s, and since $\left|\ml{T}\right|>\left|\ssl{T}\right|$. This implies that $b\leq h$, that is, $\theta(T)>\pi/4$. However, this is a contradiction because the acute angle $\Theta$ is filled with at least one $\theta(T)$ while
    $$
    \Theta=\frac{\pi}{2}-\theta(T)<\frac{\pi}{4}<\theta(T).
    $$

\begin{figure}[H]
\captionsetup[subfigure]{labelformat=empty}
      \centering
	   \begin{subfigure}{0.4\linewidth}
		\includegraphics[width=\linewidth]{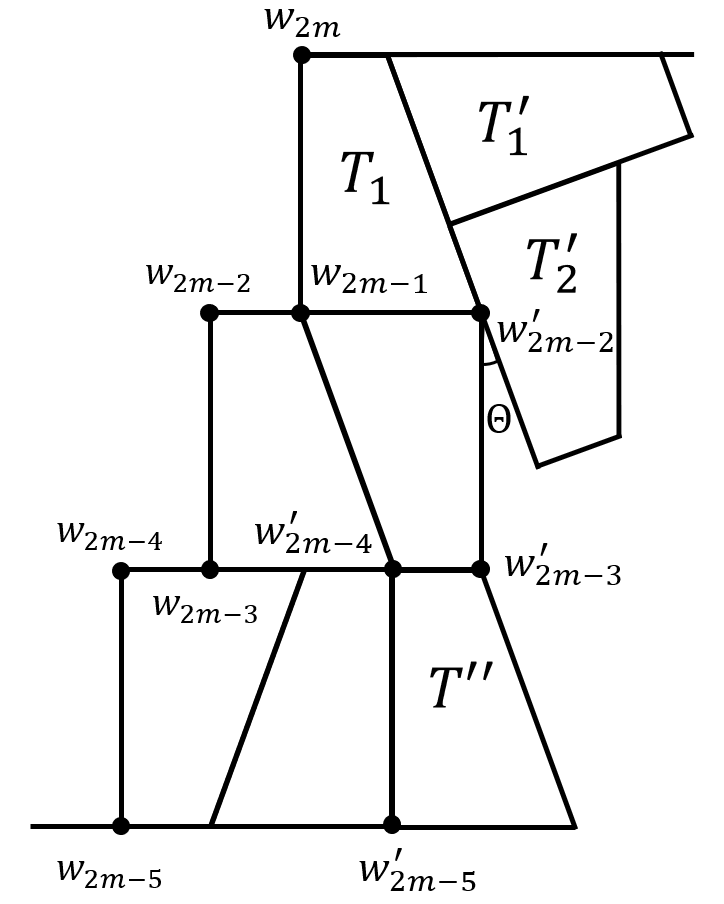}
		\caption{$\ub{T''}$ covers $\ov{w_{2m-4}'}{w_{2m-3}'}$}
		\label{fig:L2-2.1}
	   \end{subfigure}
	   \begin{subfigure}{0.4\linewidth}
		\includegraphics[width=\linewidth]{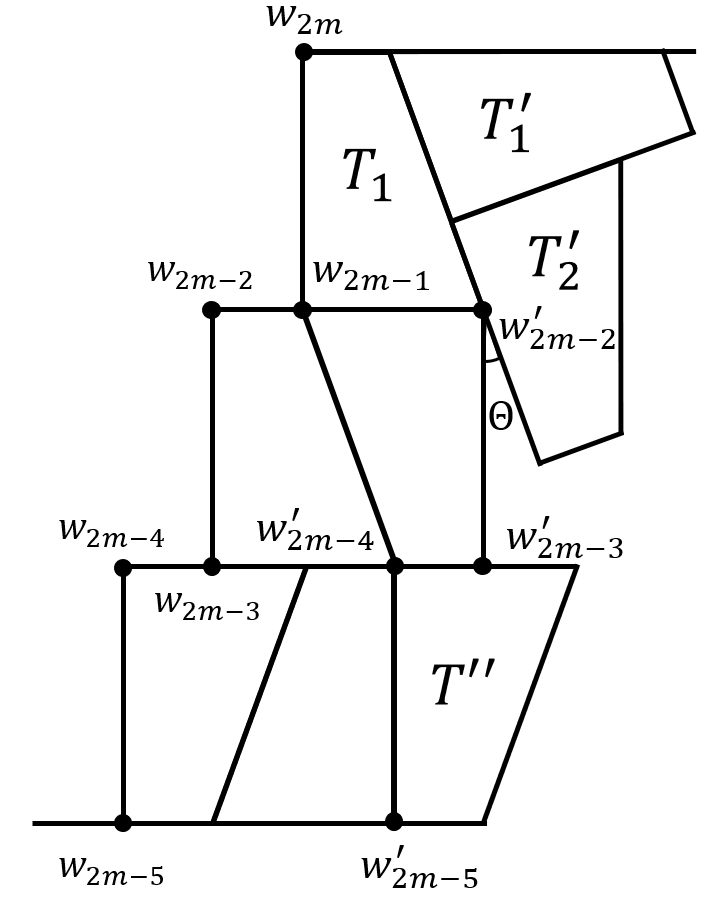}
		\caption{$\lb{T''}$ covers $\ov{w_{2m-4}'}{w_{2m-3}'}$}
		\label{fig:L2-2-1}
	    \end{subfigure}
	\caption{Covering $\ov{w_{2m-4}'}{w_{2m-3}'}$ with a base of $T''$ when $a=b/2$}
	\label{fig:L2-2}
\end{figure}

    \paragraph{Case $a<b/2$ and $a\neq b/3$:} Suppose that $a<b/2$ and $a\neq b/3$. Let $M_1$ be the closure of $L_1\backslash\ov{w_{2m-2}}{w_{2m-2}'}$, which is a closed line segment. The length of $M_1$ is 
    $$
    \left|M_1\right|=\left|L_1\right|-\left|\ov{w_{2m-2}}{w_{2m-2}'}\right|=(2b-a)-(a+b)=b-2a.
    $$
    Moreover, because of the two line segments $\ov{w_{2m-3}'}{w_{2m-2}'}$ and $\cup_{i=1}^{n_1}E_i$, edges of some tiles in $\mathfrak{T}$ perfectly cover $M_1$. In addition, from SSTS, it follows that there is a tile $T_2\in\mathfrak{T}$ such that $\ssl{T_2}=\ov{w_{2m-3}'}{w_{2m-2}'}$ and either $\ub{T_2}$ or $\lb{T_2}$ lies on $M_1$. Given that $b>b-2a$, what lies on $M_1$ is $\ub{T_2}$. This is possible only if $b-2a\geq a$, and thus we have a contradiction when $b/3<a<b/2$. On the other hand, if $a<b/3$, then a single $\theta\left(\te\right)$ fills the acute angle between $M_1$ and $\ml{T_2}$. While doing so, either $\ml{\te}$ or $\lb{\te}$ lies on the closure of $M_1\backslash\ub{T_2}$. This is also a contradiction, for the length of the line segment $M_1\backslash\ub{T_2}$ is $b-3a$, shorter than $b$ and $1$.
    
    From this point, the only remaining case is $a=b/3$. Suppose that  $a=b/3$. For each $i=1, 2,\cdots, 2m$ and $n=0, 1, 2, \cdots$, put
    $$
    w_{i}^{(n)}=w_i+(n(a+b),0),
    $$
    and for each $k=1, 2, \cdots, m$, put
    \begin{align*}
        A_k&=\ov{w_{2m-2k}^{(k-1)}}{w_{2m-2k+1}^{(k-1)}},\\
        B_k&=\ov{w_{2m-2k}^{(k-1)}}{w_{2m-2k}^{(k)}},\\
        S_k&=\ov{w_{2m-2k+1}^{(k-1)}}{w_{2m-2k+2}^{(k-1)}}.
    \end{align*}
    Choose an integer $j=2, 3, \cdots, m$. Assume that for each $i=1, 2, \cdots, j-1$
    the $(4m-4i)$-gon
    $$
    w_{1}^{(i-1)}w_2^{(i-1)}\cdots w_{2m-2i}^{(i-1)}w_{2m-2i}^{(i)}w_{2m-2i-1}^{(i)}\cdots w_{1}^{(i)}
    $$
    is a cluster in $\mathfrak{T}$ and that we already selected positive tiles $T_1,T_2,\cdots,T_j\in\mathfrak{T}$, tiles $T_1',T_2',\cdots,T_{n_{(j-1)}}'\in\mathfrak{T}$ with $1\leq n_1\leq n_2\leq\cdots\leq n_{(j-1)}$, and corresponding edges $E_1,E_2,\cdots,E_{n_{(j-1)}}$ satisfying the following properties (see Figure \ref{fig:L2-3}).
        \begin{enumerate}
        \item $\ssl{T_i}=S_i$ for each $i=1, 2, \cdots, j$;
        \item $\ub{T_1}\subset\ov{w_{2m}}{w_{2m+1}}$, and the closure of $\left(\lb{T_i}\cup A_i\right)\backslash B_i$ coincides with $\ub{T_{i+1}}$ for each $i=1, 2, \cdots, j-1$;
        \item the main legs of $T_1,T_2,\cdots,T_j$ perfectly cover the closed line segment $\cup_{i=1}^{j}\ml{T_i}$ in the order $\left(\ml{T_1},\ml{T_2},\cdots,\ml{T_j}\right)$;
        \item $E_i$ is an edge of $T_i'$ for each $i=1,2,\cdots,n_{(j-1)}$;
        \item $E_1,E_2,\cdots,E_{n_{(j-1)}}$ properly covers the closed line segment $\cup_{i=1}^{j-1}\ml{T_i}$ in the order $\left(E_1,E_2,\cdots,E_{n_{(j-1)}}\right);$
        \item $E_{n_{(j-1)}}\cap\ml{T_{j-1}}$ contains more than one point, and $E_{n_{(j-1)}}\not\subset\ml{T_{j-1}}$.
    \end{enumerate}
        \begin{figure}[H]
    \centering
    \includegraphics[width=0.55\linewidth]{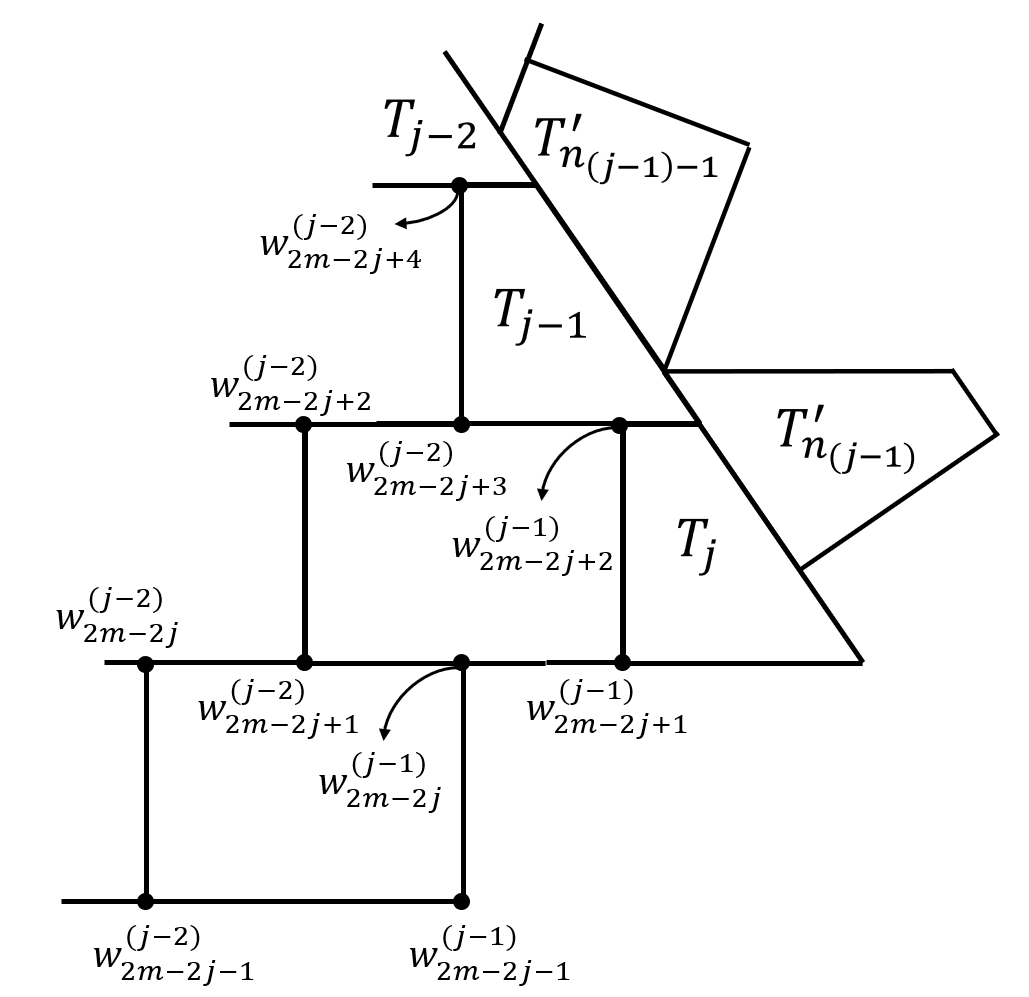}
    \caption{Tiles $T_1, T_2, \cdots, T_j, T_1', T_2', \cdots, T_{n_{(j-1)}}'$}
    \label{fig:L2-3}
\end{figure}
    Note that the above assumption is true when $j=2$. To check this, consider the line segment $M_1$ and the positive tile $T_2$ given in the preceding argument for the case when $a<b/2$ and $a\neq b/3$. Since $\left|M_1\right|=a$, the edge $\ub{T_2}$ perfectly covers $M_1$. This implies the second and third properties. The other properties immediately follow from the selection of $T_1',T_2'\cdots,T_{n_1}'$ and $E_1,E_2,\cdots,E_{n_1}$ that we previously done.
    
    For the rest of the proof, we show that unless $\ov{w_0}{w_1}$ is too short, we can select a positive tile $T_{j+1}\in\mathfrak{T}$, tiles $T_{n_{(j-1)}+1}',T_{n_{(j-1)}+2}',\cdots,T_{n_j}'\in\mathfrak{T}$ with $n_j\geq n_{(j-1)}$, and corresponding edges $E_{n_{(j-1)}+1},E_{n_{(j-1)}+2},\cdots,E_{n_j}$ satisfying the above assumption with $j$ replaced to $j+1$. If this is done, then by induction over $j$, we obtain positive tiles $T_1,T_2,\cdots,T_m\in\mathfrak{T}$ satisfying the second and third properties in the above assumption with $j$ replaced to $m$. Here, the closed line segment $\cup_{i=1}^m\ml{T_i}$ is perfectly covered by edges of some tiles in $\mathfrak{T}$ other than $T_1, T_2, \cdots, T_m$ because $\lb{T_m}\subset\ov{w_0}{w_1}$ and the interior of $\ov{w_{2m}}{w_{2m+1}}$ contains an endpoint of $\cup_{i=1}^m\ml{T_i}$. However, it follows from MTM that this is a contradiction since $E_1=\lb{T_1'}$ already lies on $\cup_{i=1}^m\ml{T_i}$. Hence, we can conclude that $\lb{T_1'}$ cannot lie on $\ml{T_1}$ in the first place and that $\ml{T_1}$ is perfectly covered by the main leg of some negative tile in $\mathfrak{T}$. This completes the inductive step over $\rho$. The case when $\ov{w_0}{w_1}$ is too short will be excluded at some point during the proof.

    Since $T$ satisfies MTM, provided that $\ov{w_0}{w_1}$ is long enough and $\cup_{i=1}^{n_{(j-1)}}E_{i}\varsubsetneq\cup_{i=1}^{j}\ml{T_i}$, we can select tiles $T_{n_{(j-1)}+1}',T_{n_{(j-1)}+2}'\cdots,T_{n_j}'\in\mathfrak{T}$ $(n_j>n_{(j-1)})$ and corresponding edges $E_{n_{(j-1)}+1},E_{n_{(j-1)}+2},\cdots,E_{n_j}$ satisfying the following properties.
        \begin{enumerate}
        \item $E_i$ is an edge of $T_i'$ for each $i=n_{(j-1)}+1,n_{(j-1)}+2,\cdots,n_j$;
        \item $E_1,E_2,\cdots,E_{n_j}$ properly covers the closed line segment $\cup_{i=1}^{j}\ml{T_i}$ in the order $\left(E_1,E_2,\cdots,E_{n_j}\right);$
        \item $E_{n_j}\cap\ml{T_j}$ contains more than one point, and $E_{n_j}\not\subset\ml{T_j}$.
    \end{enumerate}
    If $\cup_{i=1}^{j}\ml{T_i}\subset\cup_{i=1}^{n_{(j-1)}}E_i$, then we put $n_j=n_{(j-1)}$. We can immediately confirm that these selected tiles and edges satisfy the fourth to sixth properties of the above assumption with $j$ replaced to $j+1$. Similar to the situation selecting the tiles $T_2',T_3',\cdots,T_{n_1}'$ and their edges, every problem regarding the short length of $\ov{w_0}{w_1}$ and the boundary of $S^2\left(T,m,\alpha,\beta\right)$ leads to a contradiction. Therefore, we assume $\ov{w_0}{w_1}$ is long enough to avoid such problems.

    Put $L_j=\lb{T_j}\cup A_j$. This line segment is of length $2b-a$ greater than $b$, and thus we can apply the induction hypothesis for $\rho=m-j$ and observe that the $(4m-4j)$-gon
    $$
    w^{(j-1)}_1w^{(j-1)}_2\cdots w^{(j-1)}_{2m-2j}w^{(j)}_{2m-2j}w^{(j)}_{2m-2j-1}\cdots w^{(j)}_{1}
    $$
    is a cluster in $\mathfrak{T}$. Then, the closure $M_j$ of $L_j\backslash B_j$ is a closed segment of length $b-2a=a$. Moreover, by SSTS, $S_{j+1}$ is perfectly covered by the subsidiary leg of some tile in $\mathfrak{T}$ not contained in that clutser $(4m-4j)$-gon. Hence, similar to $M_1$, there is a tile $T_{j+1}\in\mathfrak{T}$ such that $\ssl{T_{j+1}}=S_{j+1}$ and $\ub{T_{j+1}}=M_j$ because of $S_{j+1}$ and $\cup_{i=1}^{n_j}E_i$. This shows that the first to third properties of the above assumption with $j$ replaced to $j+1$ are satisfied. This ends the proof.

\section{General boundary condition for reptile unit trapezoids}\label{GTP}
In this section, we prove Theorem \ref{GT}. Choose a positive integer $\mu\geq 3$. Suppose that $T$ is a unit trapezoid, not a parallelogram, with
$$
\left|\ssl{T}\right|=h,\quad\left|\ub{T}\right|=a,\quad\left|\lb{T}\right|=b.
$$
As explained in Section \ref{MRS}, it suffices to show that $\mu T$ cannot be tiled with congruent copies of $T$ if $a\notin\qh$ or $b\notin\qh$. Before going on, we prove the following lemma, which will be frequently used during further discussion.

\begin{lemma}\label{GTL}
    Let $\rho\geq 3$ be a positive integer, and $Q$ be a unit trapezoid. Suppose that there is a tiling $\mathfrak{Q}$ of $\mu Q$ with congruent copies of $Q$. Then, then there are non-negative integers $p, q, r,$ and $s$ satisfying $p<\mu$ and 
$$
p\left|\ub{Q}\right|+q\left|\lb{Q}\right|+r+s\left|\ssl{Q}\right|=\rho\left|\ub{Q}\right|.
$$
\end{lemma}
\begin{proof}
    If there are no such integers, then none of $\left|\lb{Q}\right|, 1,$ and $\left|\ssl{Q}\right|$ equals $\left|\ub{Q}\right|$. Thus, no edges of tiles in $\mathfrak{Q}$ other than $\ub{Q^{\mathfrak{Q}}}$ can lie on $\ub{\rho Q}$. That is to say, $\ub{\rho Q}$ is perfectly covered by at least three $\ub{Q^{\mathfrak{Q}}}$'s. Thus, we can select three tiles $Q_1,Q_2,Q_3\in\mathfrak{Q}$ whose upper bases lie on $\ub{\rho Q}$ in a row in the order $\left(\ub{Q_1},\ub{Q_2},\ub{Q_3}\right)$. Here, the acute angle between $\ml{Q_2}$ and $\ub{\rho Q}$ has the same size as $\theta(Q)$, and either $\theta\left(Q_1\right)$ or $\theta\left(Q_3\right)$ fills that acute angle; without loss of generality, suppose that this is done by $\theta\left(Q_1\right)$. Then, either $\ml{Q_1}$ or $\lb{Q_1}$ lies on $\ub{\rho Q}$, and thus $\ub{Q_1}$ and $\ub{Q_2}$ cannot lie on $\ub{\rho Q}$ in a row, which is a contradiction. This ends the proof.
\end{proof}
\noindent
 
 Now, assume, for the sake of contradiction, that there is a tiling $\mathfrak{T}$ of $\mu T$ with congruent copies of $T$ but $a\notin\qh$ or $b\notin\qh$. We consider two cases separately: first, $a\notin\qh$ and $b\in\qh$, and second, $b\notin\qh$.

\subsection{Case 1: $a\notin\qh$, $b\in\qh$}\label{ARG}

Suppose that $a\notin\qh$ and $b\in\qh$. Let $p, q, r,$ and $s$ be non-negative integers satisfying
$$
pa+qb+r+sh=\mu a.
$$
Since
$$
(\mu-p)x=qb+r+sh\in\qh,
$$
and since $a\notin\qh$, we have $p=\mu$. However, this contradicts the result of Lemma \ref{GTL}. Therefore, $\mu T$ cannot be tiled with congruent copies of $T$ if $a\notin\qh$ and $b\in\qh$.

\subsection{Case 2: $b\notin\qh$}
Suppose that $b\notin\qh$. For each $x\in\qh$, let $p_x,q_x,r_x,$ and $s_x$ be some non-negative integers satisfying
$$
p_xa+q_xb+r_x+s_xh=x.
$$
Moreover, let $\hat{p},\hat{q},\hat{r},$ and $\hat{s}$ be some non-negative integers such that $\hat{p}<\mu$ and
$$
\hat{p}a+\hat{q}b+\hat{r}+\hat{s}h=\mu a;
$$
Lemma \ref{GTL} ensures such integers' existence.

We first show that $q_x=0$ for all $x\in\qh$. If $a\in\qh$, then $q_x=0$ since we have
    $$
    q_xb=x-\left(p_xa+r_x+s_xh\right)\in\qh.
    $$
For the case $a\notin\qh$, assume, for the sake of contradiction, that $q_x\geq 1$. If $p_x=0$, then it follows that
    $$
    b=\frac{1}{q_x}\left(x-r_x-s_xh\right)\in\qh,
    $$
which is a contradiction. Instead, if $p_x\geq 1$, then we have
    $$
    \frac{\hat{q}}{\mu-\hat{p}}b+\frac{1}{\mu-\hat{p}}(\hat{r}+\hat{s}h)=a=-\frac{q_x}{p_x}b+\frac{1}{p_x}\left(x-r_x-s_xh\right)
    $$
since $\hat{p}<\mu$. Here, $\hat{q}/(\mu-\hat{p})\geq 0$ and $q_x/p_x>0$, and thus it follows that
    $$
    b=\left(\frac{\hat{q}}{\mu-\hat{p}}+\frac{q_x}{p_x}\right)^{-1}\left(\frac{1}{p_x}\left(x-r_x-s_xh\right)-\frac{1}{\mu-\hat{p}}(\hat{r}+\hat{s}h)\right)\in\qh,
    $$
    which is a contradiction. Therefore, $q_x=0$ if $b\notin\qh$. Observe that $q_{\rho}=q_{\rho h}=0$ for all positive integer $\rho$. This implies that $T$ satisfies MTM and $\mathfrak{T}$ is a sub-to-sub tiling. Considering that $\mu T$ is congruent to the trapezoid $\hat{M}\left(T,\mu,\mu a\right)$, we then obtain a contradiction from Corollary \ref{ENDM} when $T$ is not a right trapezoid.

To show the same result for the case when $T$ is a right trapezoid, it suffices to show further that $T$ satisfies SSTS. Then, we can deduce a contradiction from the result of Corollary \ref{ENDSSS}. We use the result of Lemma \ref{LESTS}. If $a=h$, then $\theta(T)\neq\pi/4$, for otherwise we have
$$
b=\cos\theta(T)+a=\cos\theta(T)+h=\cos\theta(T)+\sin\theta(T)=\sqrt{2}\in\mathds{Q}_{\sqrt{2}/2}=\qh.
$$
Thus, $T$ satisfies the second condition in Lemma \ref{LESTS} if $a=h$. For the case $a\neq h$, assume, for the sake of contradiction, that there exist non-negative integers $p'$ and $q'$ such that $p'\leq q'+1$ and $p'a+q'b=h$. From $q_h=0$, it first follows that $q'=0$. We then have $p'\geq 2$ since $h>0$ and $a\neq h$. However, this contradicts the assumed inequality $p'\leq q'+1$. This shows that $T$ satisfies the first condition in Lemma \ref{LESTS} if $a\neq h$. Therefore, $T$ satisfies SSTS.

In conclusion, $\mu T$ cannot be tiled with congruent copies of $T$ if $b\notin\qh$. This ends the proof of Theorem \ref{GT}.

\section{Reptile unit trapezoids with the subsidiary leg of irrational length}\label{GT2P}
In this section, we prove Theorem \ref{GT2}. Let $T$ be a unit trapezoid, not a $\pi/3$-right trapezoid, with $\left|\ssl{T}\right|=h\notin\mathds{Q}$, $\left|\ub{T}\right|=a,$ and $\left|\lb{T}\right|=b>1$. Suppose that for some integer $\mu\geq 3$, there is a tiling $\mathfrak{T}$ of $\mu T$ with congruent copies of $T$. Since $\lb{T}$ is longer than $\ml{T}$ and $\ssl{T}$, the tiling $\mathfrak{T}$ is main-to-main and sub-to-sub. Moreover, $\mu T$ is congruent to the trapezoids $\hat{M}\left(T,\mu,\mu a\right)$ and $\hat{S}\left(T,\mu,\mu a\right)$, and thus we can apply the results of Corollary \ref{ENDM} and \ref{ENDS}. Consequently, there will be a contradiction if $T$ satisfies MTM or the three conditions in Corollary \ref{ENDS}.

We first show that one of the first two conditions in Theorem \ref{GT2} holds if $T$ is not obtuse and satisfies neither MTM nor STS. If $c_{h,a}>0$ and $c_{h,b}>0$, then $T$ satisfies MTM since $h$ is irrational. Similarly, $T$ satisfies STS if $c_{1,a}>0$ and $c_{1,b}>0$. Combining these facts, we can observe that $T$ belongs to one of the following cases.
\begin{enumerate}
    \item $c_{1,a}=0$, $c_{h,a}>0$, $c_{h,b}\leq 0$;
    \item $c_{1,a}>0$, $c_{h,a}=0$, $c_{1,b}\leq 0$;
    \item $c_{1,a}>0$, $c_{h,a}<0$, $c_{1,b}\leq 0$;
    \item $c_{1,a}<0$, $c_{h,a}>0$, $c_{h,b}\leq 0$.
\end{enumerate}
We can confirm that $T$ satisfies MTM for the second case and satisfies STS for the first case. Therefore, the only possibilities are the third and fourth cases. This completes the proof for the case when $T$ is not obtuse.

Next, suppose that $T$ is obtuse and does not satisfy MTM, STS, and the last condition in Theorem \ref{GT2}. If $c_{h,a}>0$ and $c_{h,b}>0$, then $T$ satisfies MTM. On the other hand, if $c_{1,a}>0$ and $c_{1,b}>0$, then there are no integers $n\leq\mu$ such that $n(a+b)=\mu a$ unless $c_{h,a}$ and $c_{h,b}$ are both positive. Since $T$ does not satisfy MTM by assumption, $T$ must satisfy both STS and the second condition in Corollary \ref{ENDS} if $c_{1,a}>0$ and $c_{1,b}>0$. Hence, as of the case when $T$ is not obtuse, $T$ belongs to one of the four cases above. Here, $T$ satisfies MTM for the second case. Furthermore, for the first case, $T$ satisfies STS, and there are no integers $n\leq\mu$ such that $n(a+b)=\mu a$; that is, the second condition in Corollary \ref{ENDS} holds. Therefore, we can rule out every possibility except for the third and fourth cases. This completes the proof for the case when $T$ is obtuse.

\section{Angular condition for reptile $I(\vartheta,a)$ with rational $\vartheta/\pi$}\label{MI}
In this section, we prove Corollary \ref{Cii}. Suppose that $I(\vartheta,a)$ is a reptile. What immediately follows from Theorem \ref{GT} is that both $a$ and $a+2\cos{\vartheta}$ are rational, and thus so is $\cos{\vartheta}$. According to the result of Tangsupphathawat in \cite[Theorem 3.3]{tangsupphathawat2014algebraic}, if $\Theta$ is an acute angle such that both $\Theta/\pi$ and $\cos{\Theta}$ are rational, then $\Theta=\pi/3$. This shows that $\vartheta=\pi/3$ if $\vartheta/\pi\in\mathds{Q}$. This ends the proof of the corollary.

\section{Right trapezoids}\label{RP}
In this section, we prove the main results for unit right trapezoids: Corollary \ref{Rii} (Section \ref{RP1}), Theorem \ref{RR} (Section \ref{RP2} to \ref{RP4}), and Theorem \ref{NEW} (Section \ref{RP5}).

\subsection{Angular condition for reptile $R(\vartheta,a)$}\label{RP1}

Suppose that $R(\vartheta,a)$ is a reptile. Then, by Theorem \ref{GT}, both $a$ and $a+\cos\vartheta$ are in $\qh$, and thus there are rational numbers $p$ and $q$ such that $\cos\vartheta=p+q\sin{\vartheta}$. Observe that
$$
(q^2+1)\sin^2\vartheta+2pq\sin\vartheta+p^2-1=\cos^2\vartheta+\sin^2\vartheta-1=0,
$$
which implies that $\sin\vartheta$ is an algebraic integer of degree at most two, and so is $\cos\vartheta$. This proves Corollary \ref{Rii}.

\subsection{Angular condition for reptile $R(\vartheta,a)$ with rational $\vartheta/\pi$}\label{RP2}

In \cite[Theorem 3.3]{tangsupphathawat2014algebraic}, it is shown that if $\Theta$ is an acute angle such that $\Theta/\pi\in\mathds{Q}$ and $\cos\Theta$ is an algebraic integer of degree at most $2$, then $\Theta$ equals one of $\frac{\pi}{6},\frac{\pi}{5},\frac{\pi}{4},\frac{\pi}{3},$ and $\frac{2\pi}{5}$. Thus, if $R(\vartheta,a)$ is a reptile and $\vartheta/\pi$ is rational, then $\vartheta$ equals one of these five values. At first, we can rule out $\frac{\pi}{5}$ and $\frac{2\pi}{5}$ since the sine of each of these two values is not an algebraic integer of degree at most $2$ as the following equations show (in fact, both are algebraic integers of degree $4$).
\begin{align*}
    \sin{\frac{\pi}{5}}&=\cos{\left(\frac{\pi}{2}-\frac{\pi}{5}\right)}=\cos{\frac{3\pi}{10}}\\
    \sin{\frac{2\pi}{5}}&=\cos{\left(\frac{\pi}{2}-\frac{2\pi}{5}\right)}=\cos{\frac{\pi}{10}}
\end{align*}
In addition, $\frac{\pi}{6}$ can also be excluded because $\sin{\frac{\pi}{6}}$ is rational whereas $\cos{\frac{\pi}{6}}$ is irrational. Hence, we conclude that $R(\vartheta,a)$ is a reptile and $\vartheta/\pi\in\mathds{Q}$ only if $\vartheta$ equals either $\frac{\pi}{3}$ or $\frac{\pi}{4}$.

\subsection{Uniqueness of reptile $R\left(\frac{\pi}{4},a\right)$}\label{RP3}
Let $T$ be a unit $\pi/4$-right trapezoid $R\left(\frac{\pi}{4},a\right)$. Suppose that $T$ is rep-$\mu^2$ for some positive integer $\mu\geq 3$. Then, $\mu T$ is tiled with congruent copies of $T$, and thus, by Theorem \ref{GT}, there are rational numbers $\alpha$ and $\beta$ such that $a=\alpha+\beta/\sqrt{2}$. We first show that $\alpha\beta\geq 0$. For this, we assume, for the sake of contradiction, that $\alpha\beta<0$. From Lemma \ref{GTL}, it follows that there are non-negative integers $p, q, r,$ and $s$ satisfying $p<\mu$ and
$$
p\left(\alpha+\frac{\beta}{\sqrt{2}}\right)+q\left(\alpha+\frac{\beta+1}{\sqrt{2}}\right)+r+\frac{1}{\sqrt{2}}s=\mu\left(\alpha+\frac{\beta}{\sqrt{2}}\right).
$$
We then have
$$
p+q+\frac{1}{\alpha}r=\mu=p+\left(1+\frac{1}{\beta}\right)q+\frac{1}{\beta}s
$$
since $\alpha$ and $\beta$ are both nonzero and rational. The above equation gives us the equation
$$
\frac{1}{\beta}q-\frac{1}{\alpha}r+\frac{1}{\beta}s=0,
$$
which implies that $q=r=s=0$ and $p=\mu$, for $\alpha$ and $\beta$ have opposite signs. However, this contradicts the assumption $p<\mu$. Hence, we have $\alpha\beta\geq 0$, as desired. Note that $\alpha\geq 0$ and $\beta\geq 0$ but not $\alpha=\beta=0$, for $a=\alpha+\beta/\sqrt{2}>0$.

Next, we prove that $T$ satisfies MTM and SSTS if $a\neq 1/\sqrt{2}$. If this is the case, then it follows from Corollary \ref{ENDSSS} that $a=1\sqrt{2}$ since $\mu T$ is congruent to the trapezoid $\hat{M}\left(T,\mu,\mu a\right)$. Suppose that $a\neq 1/\sqrt{2}$. To show that $T$ satisfies MTM, choose a positive integer $\rho$, and let $p', q', r',$ and $s'$ be non-negative integers such that
$$
p'\left(\alpha+\frac{\beta}{\sqrt{2}}\right)+q'\left(\alpha+\frac{\beta+1}{\sqrt{2}}\right)+r'+\frac{1}{\sqrt{2}}s'=\rho.
$$
We then have
$$
p'\beta+q'(\beta+1)+s'=0.
$$
since $\alpha,\beta\in\mathds{Q}$. Considering that $\beta\geq 0$, we obtain $q'=s'=0$, which shows that $T$ satisfies MTM. For SSTS, we show that the first condition of Lemma \ref{LESTS} holds. Assume, for the sake of contradiction, that there are non-negative integers $p''$ and $q''$ satisfying $p''\leq q''+1$ and
$$
p''\left(\alpha+\frac{\beta}{\sqrt{2}}\right)+q''\left(\alpha+\frac{\beta+1}{\sqrt{2}}\right)=\frac{1}{\sqrt{2}}.
$$
Then, we first have $q''=0$ because $\alpha+(\beta+1)/\sqrt{2}>1/\sqrt{2}$, and thus $p''=1/(\alpha\sqrt{2}+\beta)$ equals either $0$ or $1$. This implies that $\alpha=0$ and $\beta=1$. However, this contradicts the assumption $a\neq 1/\sqrt{2}$. Hence, $T$ must satisfy the first condition in Lemma \ref{LESTS} and therefore satisfies SSTS.

In conclusion, if $R\left(\frac{\pi}{4},a\right)$ is a reptile, then $a$ equals $1/\sqrt{2}$.

\subsection{Reduction to five candidates of reptile $R\left(\frac{\pi}{3},a\right)$}\label{RP4}

Let $T$ be a unit $\pi/3$-right trapezoid $R\left(\frac{\pi}{3},a\right)$. Suppose that $T$ is rep-$\mu^2$ for some positive integer $\mu\geq 3$. Then, there is a tiling $\mathfrak{T}$ of $\mu T$ with congruent copies of $T$. Thus, from Theorem \ref{GT}, it follows that $a=\alpha+\frac{\sqrt{3}}{2}\beta$ for some rational numbers $\alpha$ and $\beta$. The structure of the proof is as follows. First, we show that $\alpha\beta\geq 0$ as in Section \ref{RP3}. Next, we observe that $T$ satisfies MTM and SSTS unless $\beta$ vanishes, which implies that $\beta=0$ due to Corollary \ref{ENDSSS}. Finally, we show that $T$ satisfies STS and $\mathfrak{T}$ is a strictly main-to-main tiling if $\alpha\neq 1,\frac{1}{2},\frac{1}{4},\frac{1}{6},\frac{1}{8}$.

For the first step, assume that $\alpha\beta<0$. Let $p,q,r,$ and $s$ be non-negative integers such that $p<\mu$ and
$$
p\left(\alpha+\frac{\sqrt{3}}{2}\beta\right)+q\left(\alpha+\frac{1}{2}+\frac{\sqrt{3}}{2}\beta\right)+r+\frac{\sqrt{3}}{2}s=\mu\left(\alpha+\frac{\sqrt{3}}{2}\beta\right).
$$
The existence of such integers follows from Lemma \ref{GTL}. Since $\alpha$ and $\beta$ are both nonzero, this gives us the equation
$$
p+\left(1+\frac{1}{2\alpha}\right)q+\frac{1}{\alpha}r=\mu=p+q+\frac{1}{\beta}s
$$
which can be reduced to the equation
$$
\frac{1}{2\alpha}q+\frac{1}{\alpha}r-\frac{1}{\beta}s=0.
$$
Considering that $\alpha$ and $\beta$ have opposite signs, we obtain $q=r=s=0$ and $p=\mu$, which contradicts the assumption $p<\mu$. Hence, we have $\alpha\beta\geq 0$.

Next, we show that $T$ satisfies MTM and SSTS if $\beta\neq 0$. Suppose that $\beta\neq 0$. Since $a=\alpha+\frac{\sqrt{3}}{2}\beta>0$, we have $\alpha\geq 0$ and $\beta>0$. Choose a positive integer $\rho$, and let $p',q',r',$ and $s'$ be non-negative integers such that
$$
p'\left(\alpha+\frac{\sqrt{3}}{2}\beta\right)+q'\left(\alpha+\frac{1}{2}+\frac{\sqrt{3}}{2}\beta\right)+r'+\frac{\sqrt{3}}{2}s'=\rho.
$$
Since $\alpha$ and $\beta$ are both rational, we have
$$
p'\beta+q'\beta+s'=0.
$$
For $\beta$ is positive, the above equation implies that $p'=q'=s'=0$. Thus, $T$ satisfies MTM. To prove that $T$ satisfies SSTS, it suffices to show that the first condition in Lemma \ref{LESTS} holds if $a\neq\sqrt{3}/2$. Suppose that $a\neq\sqrt{3}/2$ and assume that there are non-negative integers $p''$ and $q''$ such that $p''\leq q''+1$ and
$$
p''\left(\alpha+\frac{\sqrt{3}}{2}\beta\right)+q''\left(\alpha+\frac{1}{2}+\frac{\sqrt{3}}{2}\beta\right)=\frac{\sqrt{3}}{2}.
$$
Then, from $\alpha,\beta\in\mathds{Q}$, it follows that
$$
p''\alpha+q''\left(\alpha+\frac{1}{2}\right)=0,
$$
and thus $q''=0$ since $\alpha\geq 0$. Furthermore, $p''$ equals either $0$ or $1$. Both possibilities are contradictory because $p''=0$ implies $0=\frac{\sqrt{3}}{2}$, and $p''=1$ implies $\alpha+\frac{\sqrt{3}}{2}\beta=\frac{\sqrt{3}}{2}$. Hence, $T$ must satisfy one of the two conditions in Lemma \ref{LESTS} and therefore satisfies SSTS.

To complete the final step, suppose that $\alpha\neq 1,\frac{1}{2},\frac{1}{4},\frac{1}{6},\frac{1}{8}$. From the preceding step, it follows that $\beta=0$; otherwise, $T$ cannot be rep-$\mu^2$ by Corollary \ref{ENDSSS}. Thus, we can first confirm that $T$ satisfies STS. For the remaining part of the final step, it suffices to show that for every closed line segment $L$ of length $1$ contained in $\mu T$, no edges of tiles in $\mathfrak{T}$ other than $\ml{\te}$ can perfectly cover $L$. Assume, for the sake of contradiction, that $L$ is perfectly covered by $u$ upper bases, $v$ lower bases, and $w$ subsidiary legs of tiles in $\mathfrak{T}$, where $u, v,$ and $w$ are non-negative integers. Then, these three integers satisfy the equation
    $$
    u\alpha+v\left(\alpha+\frac{1}{2}\right)+\frac{\sqrt{3}}{2}w=1.
    $$
    From $\alpha\in\mathds{Q}$, it immediately follows that $w=0$. Moreover, $v$ equals either $0$ or $1$ since $\alpha+1/2>1/2$. If $v=0$, then $u\geq 3$, for $\frac{1}{\alpha}\neq 1, 2$. On the other hand, if $v=1$, then we have $u\geq 4$ since
    $$
    u=u+v\left(1+\frac{1}{2\alpha}\right)-\left(1+\frac{1}{2\alpha}\right)=\frac{1}{\alpha}-\frac{1}{2\alpha}-1=\frac{1}{2\alpha}-1,
    $$
    and since $\frac{1}{2\alpha}\neq 1, 2, 3, 4$. However, no three $\ub{\te}$'s can lie on $L$ in a row, as shown in the proof of Lemma \ref{GTL}. Moreover, four $\ub{\te}$'s and one $\lb{\te}$ cannot lie on $L$ in a row in the order
    $$
    \left(\ub{\te}, \ub{\te},\lb{\te}, \ub{\te}, \ub{\te}\right).
    $$
    Therefore, we derive a contradiction in either case. In conclusion, $L$ can be perfectly covered only with a single $\ml{\te}$. This completes the final step.
    
    In conclusion, if $R\left(\frac{\pi}{3},a\right)$ is a reptile, then $a$ equals one of $1,\frac{1}{2},\frac{1}{4},\frac{1}{6},$ and $\frac{1}{8}$. This ends the proof of Theorem \ref{RR}.

\subsection{$R\left(\frac{\pi}{3},\frac{1}{8}\right)$ is a reptile}\label{RP5}

Put $T=R\left(\frac{\pi}{4},\frac{1}{8}\right)$. Since the length of each edge of $T$ is a number in $\mathds{Q}_{\sqrt{3}}$, it is obvious that $T$ is not rep-$n$ if $1<n<25$ and $\sqrt{n}$ is neither an integer nor an integer multiple of $\sqrt{3}$. Thus, it suffices to check whether $T$ is rep-$n$ when $n=3, 4, 9, 12, 16$. Assume that a tiling $\mathfrak{T}$ of $\sqrt{n}T$ with congruent copies of $T$ exists. If $n$ equals either $3$ or $12$, then $\ml{\sqrt{n}T}$ is perfectly covered by some $\ssl{\te}$s, which is impossible. On the other hand, if $n$ equals one of $4, 9, 16$, then $\left|\ub{\sqrt{n}T}\right|<\left|\lb{T}\right|$. Hence, due to Lemma \ref{L1.1.1}, there is a tile $T'\in\mathfrak{T}$ such that $\ub{T'}\subset\ub{\sqrt{n}T}$ and $\ssl{T'}\subset\ssl{\sqrt{n}T}$. This implies that either $\lb{\te}$ or $\ml{\te}$ lies on the closure of $\ub{\sqrt{n}T}\backslash\ub{T'}$, which is a contradiction since the length of the closure is less than or equal to $\frac{1}{2}$. Therefore, $\sqrt{n}T$ cannot be tiled with congruent copies of $T$ if $n=3, 4, 9, 12, 16$. This completes the proof of Theorem \ref{NEW}.

\bibliographystyle{IEEEtran}
\bibliography{ref}

\begin{thebibliography}{1}
\providecommand{\url}[1]{#1}
\csname url@samestyle\endcsname
\providecommand{\newblock}{\relax}
\providecommand{\bibinfo}[2]{#2}
\providecommand{\BIBentrySTDinterwordspacing}{\spaceskip=0pt\relax}
\providecommand{\BIBentryALTinterwordstretchfactor}{4}
\providecommand{\BIBentryALTinterwordspacing}{\spaceskip=\fontdimen2\font plus
\BIBentryALTinterwordstretchfactor\fontdimen3\font minus \fontdimen4\font\relax}
\providecommand{\BIBforeignlanguage}[2]{{%
\expandafter\ifx\csname l@#1\endcsname\relax
\typeout{** WARNING: IEEEtran.bst: No hyphenation pattern has been}%
\typeout{** loaded for the language `#1'. Using the pattern for}%
\typeout{** the default language instead.}%
\else
\language=\csname l@#1\endcsname
\fi
#2}}
\providecommand{\BIBdecl}{\relax}
\BIBdecl

\bibitem{langford19401464}
C.~D. Langford, ``1464. uses of a geometric puzzle,'' \emph{The Mathematical Gazette}, vol.~24, no. 260, pp. 209--211, 1940.

\bibitem{golomb1964replicating}
S.~W. Golomb, ``Replicating figures in the plane,'' \emph{The Mathematical Gazette}, vol.~48, no. 366, pp. 403--412, 1964.

\bibitem{osburg2004selbstahnliche}
I.~Osburg, ``Selbst{\"a}hnliche polyeder,'' Ph.D. dissertation, Friedrich-Schiller-Universität Jena, Fakultät für Mathematik und Informatik, 2004.

\bibitem{betke1976zerlegungen}
U.~Betke, ``Zerlegungen konvexer polygone,'' 1976, unpublished manuscript.

\bibitem{laczkovich2023quadrilateral}
M.~Laczkovich, ``Quadrilateral reptiles,'' \emph{Beitr{\"a}ge zur Algebra und Geometrie/Contributions to Algebra and Geometry}, vol.~64, no.~4, pp. 945--967, 2023.

\bibitem{sallows2014more}
L.~Sallows, ``More on self-tiling tile sets,'' \emph{Mathematics Magazine}, vol.~87, no.~2, pp. 100--112, 2014.

\bibitem{snover1991rep}
S.~L. Snover, C.~Waiveris, and J.~K. Williams, ``Rep-tiling for triangles,'' \emph{Discrete mathematics}, vol.~91, no.~2, pp. 193--200, 1991.

\bibitem{tangsupphathawat2014algebraic}
P.~Tangsupphathawat, ``Algebraic trigonometric values at rational multipliers of $\pi$,'' \emph{Acta et Commentationes Universitatis Tartuensis de Mathematica}, vol.~18, no.~1, pp. 9--18, 2014.

\end{thebibliography}

\end{document}